\documentclass[a4paper,12pt]{amsart}
\usepackage[italian,english]{babel}

\usepackage{geometry}
\usepackage{enumitem}
\usepackage{mathrsfs}
\usepackage{comment}
\usepackage[colorlinks, linkcolor=blue,anchorcolor=Periwinkle,
citecolor=blue,urlcolor=Emerald]{hyperref}
\usepackage{leftidx}

\usepackage{amsmath}
\usepackage{amsthm}
\usepackage[T1]{fontenc}
\usepackage[utf8]{inputenc}
\usepackage{amssymb}
\usepackage{csquotes}

\usepackage{mathrsfs}

\usepackage{graphicx}
\usepackage{tikz-cd}
\tikzset{dynkdot/.style={circle,draw,scale=.38}}

\newtheorem{prop}{Proposition}[section]
\newtheorem{lem}[prop]{Lemma}
\newtheorem{teo}[prop]{Theorem}
\newtheorem{cor}[prop]{Corollary}

\theoremstyle{definition}
\newtheorem{de}[prop]{Definition}

\newtheorem{rem}[prop]{Remark}
\newtheorem{ex}[prop]{Example}

\newtheorem{con}[prop]{Construction}

{\left\lbrace\begin{array}{@{}l@{}}}%
{\end{array}\right.}

\title{A quantum cluster algebra structure on the semi-derived Hall algebra}


\author{Alessandro Contu}

\address{
Université Paris Cité, Sorbonne Université,
CNRS,
Institut de Mathématiques de Jussieu- Paris Rive Gauche,
F-75013 Paris, France
}
\email{alessandro.contu@imj-prg.fr}

\usepackage{comment}

\begin{document}

\maketitle

\begin{abstract}
Using Hernandez--Leclerc's isomorphism between the derived Hall algebra of a representation-finite quiver $Q$ and the quantum Grothendieck ring of the quantum loop algebra of the Dynkin type of $Q$, we lift the (quantum) cluster algebra structure of the quantum Grothendieck ring to the semi-derived Hall algebra, introduced by Gorsky, of the category of bounded complexes of projective modules over the path algebra of $Q$. We also construct a braid group action on the semi-derived Hall algebra, lifting Kashiwara--Kim--Oh--Park's braid group action on the quantum Grothendieck ring.
\end{abstract}

\section{introduction}

\subsection{Semi-derived Hall algebras}
Let $\mathcal{A}$ be an abelian \emph{finitary} category, that is, a category  whose morphism spaces and $\mathrm{Ext}^1$-spaces have finite cardinality. To the category $\mathcal{A}$, Ringel \cite{Ringel_Hall_algebras_1990} associated the Hall algebra $\mathcal{H}(\mathcal{A})$, defined as the $\mathbb{Q}$-vector space with basis the isomorphism classes of objects of $\mathcal{A}$ and with multiplication:
\[ [A]*[B] = \sum_{[C]} g_{A,B}^C [C],\]
where the coefficients $g_{A,B}^C$ essentially count the extensions of $B$ by $A$ with middle term $C$ (see \ref{de_Hall_algebra}
for the precise definition). Ringel's definition was extended to exact categories by Hubery \cite{Hubery2006}. Moreover, for suitable dg categories, Toën constructed an analogous algebra, called the \emph{derived Hall algebra} \cite{Toen2006}. His construction was later extended to suitable triangulated categories by Xiao and Xu~\cite{Xiao-Xu2008}. Notice that the definitions of (derived) Hall algebras that we can find in the literature often differ by a twist of the multiplication, a rescaling of the generators and/or the passage to the opposite algebra (cf. section ~\ref{sec_Hall}). In the following, we write $\mathcal{H}_{tw}$ and $\mathcal{DH}_{tw}$ to highlight the presence of one of these twists.

One of the reasons for the interest in Hall algebras is their close link with the theory of quantum groups first established by Ringel~\cite{Ringel_Hall_algebra_quant_group_1990}. Let $\mathfrak{g}$ be a simply laced finite-dimensional simple Lie algebra of Dynkin type $\Delta$, let $Q$ be a quiver with underlying diagram $\Delta$ and let $\Bbbk$ be a finite field with $q$ elements. Denote the square root of $q$ by $v$. Write $\mathcal{H}_{tw}(Q)$ for the (twisted) Hall algebra of the category $\mathrm{mod}(\Bbbk Q)$ of finite-dimensional modules over the path algebra $\Bbbk Q$. 
Let $U_v(\mathfrak{n})$ be the \emph{positive part} of the quantum group $U_v(\mathfrak{g})$. Ringel realized $U_v(\mathfrak{n})$ as an Hall algebra, that is, he provided an isomorphism 
\[ U_q(\mathfrak{n}) \xrightarrow{\sim} \mathcal{H}_{tw}(Q)\]
sending the generator $E_i$ to the class $[S_i]$ of the simple module of at the vertex~$i$. 
Toën's construction of the derived Hall algebra was partially motivated by the problem of realizing the whole quantum group $U_v(\mathfrak{g})$ as a Hall algebra. This was not achieved, but Hernandez-Leclerc \cite{HL_quantum_Groth_rings_derived_Hall} proved that there is an isomorphism
\begin{equation}
\label{intro_eq_HL_iso}
\Phi:\mathcal{K}_v(\mathscr{C}_\mathfrak{g}^{\mathbb{Z}})\xrightarrow{\sim} \mathcal{DH}_{tw}(Q),
\end{equation}
where $\mathcal{K}_v(\mathscr{C}_\mathfrak{g}^{\mathbb{Z}})$ is the specialization at $t=v$ of the quantum Grothendieck ring of the category $\mathscr{C}_{\mathfrak{g}}^{\mathbb{Z}}$ (a full subcategory of the category of finite-dimensional modules of type 1 of the quantum loop algebra of $\mathfrak{g}$) and $\mathcal{DH}_{tw}(Q)$ is the (twisted) derived Hall algebra of the derived category $D^b(Q)$  of $\mathrm{mod}(\Bbbk Q)$.
Let us discuss the isomorphism $\Phi$ in more detail. By the theory of $q$-characters \cite{FrenkelReshetikhin98, Frenkel_Mukhin_combinatorics_q_characters}, the simple objects of the category $\mathscr{C}_\mathfrak{g}^{\mathbb{Z}}$ are parametrized by certain monic monomials in the variables $(Y_{(i,p)})_{i\in I_\mathfrak{g}, p\in \mathbb{Z}}$, called \emph{dominant monomials}. For any dominant monomial $m$, let $L(m)$ be the associated simple module. It is the head of the \emph{standard module} $M(m)$. In the quantum Grothendieck ring of $\mathscr{C}_\mathfrak{g}^{\mathbb{Z}}$, we have analogues $L_t(m)$ and $M_t(m)$ of the isomorphism classes of these modules. They are called, respectively, the \emph{$(q,t)$-characters} of the simple module $L(m)$ and of the standard module $M(m)$. They satisfy a trianagular system of equations (\cite[Thm.~8.1]{Nakajima2004})
\[
L_t(m)=M_t(m)+\sum_{m<m'}a_{m,m'}M_t(m'),
\]
where the coefficients $a_{m,m'}$ are in $t\mathbb{Z}[t]$ and the relation $<$ is Nakajima's order on dominant monomials, cf subsection \ref{subsec_categoryC0}.  Both the $(q,t)$-characters of the simple modules and of the standard modules form bases of the quantum Grothendieck ring.
On the other hand, the twisted derived Hall algebra $\mathcal{DH}_{tw}(Q)$ is generated by the isomorphism classes  $Z_{M,i}$ of the shifted modules $M[i]$, $M\in \mathrm{mod}(\Bbbk Q),\ i\in \mathbb{Z}$. For an object $V$ of $D^b(Q)$ with a decomposition $V=\bigoplus_{k} M_k[i_k]$ into a direct sum of shifted modules $M_k[i_k]$, let $Z_V$ be the element 
\[Z_V =  \prod^\curvearrowright_k Z_{M_k,i_k},\]
where the factors are ordered so that 
\[\mathrm{Ext}^1_{\mathcal{D}^b(Q)}(M_{k'}[i_{k'}],M_k[i_k])=0\]
if $Z_{M_k,i_k}$ precedes $Z_{M_{k'},i_{k'}}$. Then the set of the elements $Z_V$, where $V$ runs through the isoclasses of $D^b(Q)$, is a basis of the twisted derived Hall algebra. By a fundamental result of Happel \cite{Happel_articolo_1987}, the dominant monomials of the form $Y_{i,p}$ correspond bijectively to the indecomposable objects $V(i,p)$ of the derived category $D^b(Q)$. Then, the isomorphism $\Phi$ has the following properties 
\begin{itemize}
\item[$(i)$] for any dominant monomial $Y_{i,p}$, the $(q,t)$-character of the simple module $L(Y_{i,p})$ is mapped by $\Phi$ to a scalar multiple of $Z_{V(i,p)}$;
\item[$(ii)$] the basis of $(q,t)$-characters of standard modules is mapped by $\Phi$ to a rescaling of the tautological basis of $\mathcal{DH}_{tw}(Q)$ given by the isoclasses of all objects of $D^b(Q)$.
\end{itemize}

The Hall algebra interpretation of the whole quantum group $U_v(\mathfrak{g})$ was eventually achieved by Bridgeland \cite{Bridgeland2013}. Let 
$C_{\mathbb{Z}/2}(\mathcal{P})$ be the category of $2$-periodic complexes of projective $\Bbbk Q$-modules

\[
\begin{tikzcd}
P_0 \arrow[r, "d_0", shift left=2]  & P_1 \arrow[l, swap, "d_1"', shift left=2]
\end{tikzcd}, \ \ \ d_0\circ d_1 = 0, d_1\circ d_0 = 0.
\]

For any indecomposable object $M$ of $\mathrm{mod}(\Bbbk Q)$, fix a minimal projective resolution
\[ 0 \rightarrow P_M \xrightarrow{f_M} Q_M \rightarrow M.\]
The indecomposable objects of the category $C_{\mathbb{Z}/2}(\mathcal{P})$ are of the following forms
\begin{align}
C_M: \begin{tikzcd}[ampersand replacement=\&]
P_M \arrow[r, "f_M", shift left=2]  \& Q_M \arrow[l, swap, "0"', shift left=2],
\end{tikzcd}
 & C_M^*: \begin{tikzcd}[ampersand replacement=\&]
Q_M \arrow[r, "0", shift left=2]  \& P_M \arrow[l, swap, "f_M"', shift left=2]\end{tikzcd},\\
K_P: \begin{tikzcd}[ampersand replacement=\&]
P \arrow[r, "\mathrm{Id}", shift left=2]  \& P \arrow[l, swap, "0"', shift left=2],
\end{tikzcd} & K_P^*: \begin{tikzcd}[ampersand replacement=\&]
P \arrow[r, "0", shift left=2]  \& P \arrow[l, swap, "\mathrm{Id}"', shift left=2], \end{tikzcd}
\end{align}
where $M\in \mathrm{mod}(\Bbbk Q)$ and $P\in \mathcal{P}$ are indecomposable. 
Bridgeland \cite{Bridgeland2013}
considered the localization of the (twisted) Hall algebra $\mathcal{H}_{tw}(C_{\mathbb{Z}/2}(\mathcal{P}))$ at the classes of contractible complexes and took its quotient by the 2-sided ideal generated by the differences \[ [K_P]-[K_P^*],\]
where $P$ is indecomposable projective. Denote the resulting algebra by $\overline{\mathcal{SDH}}_{tw}(C_{\mathbb{Z}/2}(\mathcal{P}))$.
He proved \cite{Bridgeland2013} that there is an isomorphism 

\[ U_v(\mathfrak{g}) \xrightarrow{\sim} \overline{\mathcal{SDH}}_{tw}(C_{\mathbb{Z}/2}(\mathcal{P})),\]
given, up to scalar factors, by the assignments
\[ E_i \mapsto [C_{S_i}][K_{P_{S_i}}]^{-1} ;\ \ F_i \mapsto [C^*_{S_i}][K^*_{P_{S_i}}]^{-1} ; \ \ K_i \mapsto [K_{Q_{S_i}}][K^*_{P_{S_i}}]^{-1}.\]

Inspired by Bridgeland's work, Gorsky \cite{Gorsky2013semiderived, Gorsky2018} defined the semi-derived Hall algebra $\mathcal{SDH(\mathcal{F}})$ of a Frobenius\footnote{A \emph{Frobenius category} is an exact category with enough projectives and enough injectives and where the classes of projectives and of injectives coincide.} category $\mathcal{F}$ as the localization of the Hall algebra $\mathcal{H}(\mathcal{F})$ at the classes of the projective-injective objects. Let $\underline{\mathcal{F}}$ be the stable category of $\mathcal{F}$. Gorsky \cite[Thm.~4.2]{Gorsky2018} realised the derived Hall algebra of $\underline{\mathcal{E}}$ as a quotient of the semi-derived Hall algebra $\mathcal{SDH(\mathcal{F}})$ (more precisely, of a twisted version of it), via an isomorphism
\[\mathcal{SDH(F)}/I(\mathcal{P(F)})\cong \mathcal{DH(\underline{F})},\]
where $I(\mathcal{P(F)})$ is the two-sided ideal of $\mathcal{SDH(F)}$ generated by the differences $[P]-1$, where $P$ is projective in $\mathcal{F}$.

As an example of Gorsky's theorem, consider the category $C^b(\mathcal{P})$ of bounded complexes of projective $\Bbbk Q$-modules. It is a Frobenius category whose projective-injective objects are exactly the contractible complexes. The associated stable category is the homotopy category of bounded complexes of projectives $K^b(\mathcal{P})$, canonically equivalent to the bounded derived category $D^b(Q)$. Following \cite{Gorsky2013semiderived, Zhang2022}, a system of generators of the semiderived Hall algebra $\mathcal{SDH}(C^b(\mathcal{P}))$ is given by certain elements $E_{M,i} \text{ and }\ K_{\alpha,i}$, for $i\in \mathbb{Z}, M \in \mathrm{mod}(\Bbbk Q), \alpha\in \mathcal{K}(\mathrm{mod}(\Bbbk Q))$. These elements are analogues of those defined by Bridgeland in the 2-periodic case. In particular, the elements $K_{\alpha,i}$ are Laurent monomials in the representatives of certain contractible complexes. Therefore, by Gorsky's theorem, we have a surjective map \footnote{it is a morphism of modules over the subalgebra generated iver $Q$ by the isomorphism classes of contractibles.}
\begin{equation}
\label{intro_eq_surjetive_misha_C^b(P)}
\mathcal{SDH}(C^b(\mathcal{P})) \twoheadrightarrow \mathcal{DH}(Q),\  
E_{M,i} \mapsto Z_{M,i},\
K_{\alpha,i} \mapsto 1.
\end{equation}
Let us recall that we have the isomorphism $
\Phi:\mathcal{K}_v(\mathscr{C}_\mathfrak{g}^{\mathbb{Z}})\xrightarrow{\sim} \mathcal{DH}_{tw}(Q)$ of  (\ref{intro_eq_HL_iso}). Moreover, thanks to the quantized version proved by Fujita--Hernandez--Oya--Oh \cite{HFOO_iso_quant_groth_ring_clust_alg} of a result of Kashiwara--Kim--Oh--Park \cite{KKOP_mon_cat_quant_aff_II}, the quantum Grothendieck ring $\mathcal{K}_v(\mathscr{C}_\mathfrak{g}^{\mathbb{Z}})$ carries a quantum cluster algebra structure, which yields an isomorphic structure on the twisted derived Hall algebra. It is now  a natural question to ask whether and how we can lift this quantum cluster algebra structure from the derived to the semi-derived twisted Hall algebra along the map (\ref{intro_eq_surjetive_misha_C^b(P)}). This article is dedicated to providing such a lift. 

\subsection{A quantum cluster algebra structure on the semi-derived Hall algebra}
To start with, in section \ref{sec_semiderived_Hall_complexes}, we apply a twist to the product of the semi-derived Hall algebra $\mathcal{SDH}(C^b(\mathcal{P}))$, in such a way that, using Gorsky's theorem we obtain a surjective algebra morphism 
\begin{equation}
\label{intro_eq_surjetive_misha_C^b(P)_twisted}
\mathcal{SDH}_{tw}(C^b(\mathcal{P})) \twoheadrightarrow \mathcal{DH}_{tw}(Q),\  
E_{M,i} \mapsto Z_{M,i},\
K_{\alpha,i} \mapsto 1.
\end{equation}
Next, we want to lift the quantum cluster variables of the quantum Grothendieck ring, which are $(q,t)$-characters of simple modules, to the semi-derived Hall algebra. Via Happel's theorem, we associate an object $V(m)$ of $D^b(Q)$ to each dominant monomial $m$. Then, the image of the $(q,t)$-character of the simple module $L(m)$ under the isomorphism (\ref{intro_eq_HL_iso}) is of the form 
\begin{equation}
\label{intro_formula_image_qt_simple}
\Phi (L_t(m))= a(m) Z_{V(m)} +\sum_{m'<m} a(m,m') Z_{V(m')},
\end{equation}
where $a(m)$ and $a(m,m')$ are certain coefficients in $\mathbb{Q}(v^{1/2})$.
The twisted derived Hall algebra is graded by the Grothendieck group of the derived category $D^b(Q)$, while the twisted semiderived Hall algebra is graded by the Grothendieck group of the category of bounded complexes of projectives $C^b(\mathcal{P})$. We show (cf. Remark \ref{order_comparison}) that the elements $\Phi (L_t(m))$ are homogeneous with respect to the grading by $\mathcal{K}(D^b(Q))$. Therefore, it is natural to try to lift these elements to homogeneous elements of the twisted semi-derived Hall algebra. To start with, for any object $V$ of $D^b(Q)$, similarly to the definition of $Z_V$ above, we define elements $E_{V}$ and $K_V$ of the semi-derived Hall algebra in such a way that $\pi^\mathcal{H}(E_{V})=Z_{V}$ and $K_V$ is in the subgroup of the invertible elements of $\mathcal{SDH}_{tw}(C^b(\mathcal{P}))$ generated by the representatives of the contractible complexes, denoted by $\mathcal{SDH}_{tw}(C^b(\mathcal{P}))^\times_K$. 

Then we define the \emph{semi-derived $(q,t)$-character} of the simple module associated to the dominant monomial $m$ as

\begin{equation}
\label{intro_eq_semiderived_Lv(m)}
    \mathcal{L}_v(m)= a(m) E_{V(m)} +\sum_{m'<m,\ a(m,m')\neq 0} a(m,m') E_{V(m')}K(m,m'),
\end{equation}
where $a(m)$ and $a(m,m')$ are the coefficients defined in  (\ref{intro_formula_image_qt_simple}) and, for any $m\leq m'$, $K(m,m')$ is the only element of $\mathcal{SDH}_{tw}(C^b(\mathcal{P}))_K^\times$ such that $\mathrm{deg}(E_{V(m)})=\mathrm{deg}(E_{V(m')}K(m,m'))$. 

As a next step, we lift the exchange relations of the quantum Grothendieck ring to the semi-derived Hall algebra. These exchange relations are all of the form: 

\begin{equation}
\label{intro_eq_exchange_relation_qt_characters}
L_t(m_k)L_t(m_k') = a(t)\overrightarrow{\prod_{i\rightarrow k}} L_t(m_i) +   b(t)\overrightarrow{\prod_{j\leftarrow k}} L_t(m_j), \end{equation} 
where the coefficients $a(t),b(t)$ belong to $\mathbb{Z}[t^{\pm 1/2}]$, $( L_t(m_k),L_t(m_k'))$ is the exchange pair and, for each arrow $i\rightarrow k$ (resp. $j \leftarrow k$), $L_t(m_i)$ (resp. $L_t(m_j)$) is the cluster variable associated to the vertex $i$ (resp. j). Moreover, one of the following holds:
\[ \begin{cases}
    m_km_k'= \prod_{i\rightarrow k} m_i; \\
    mm'\geq \prod_{j\leftarrow k} m_j;
    \end{cases}
    \text{  or  }\ \begin{cases} m_km_k'= \prod_{j\leftarrow k} m_j;\\
    m_km_k'\geq \prod_{i\rightarrow k} m_i.
    \end{cases}\]
We prove in Proposition \ref{prop_semiderived_exchange_rel} that in the semi-derived Hall algebra, we have the following possible lifts of (\ref{intro_eq_exchange_relation_qt_characters}) :
\begin{itemize}
    \item   If $m_km_k'= \prod_{i\rightarrow k} m_i$, then we have 

\begin{equation}
\label{intro_eq_semiderived_exchange_relation}
    \mathcal{L}_v(m_k)\mathcal{L}_v(m_k') = a(v)\overrightarrow{\prod_{i\rightarrow k}} \mathcal{L}_v(m_i) +   b(v)K(m_km_k', \prod_{j\leftarrow k} m_j)\overrightarrow{\prod_{j\leftarrow k}} \mathcal{L}_v(m_j). 
\end{equation} 

\item If $m_km_k'= \prod_{k\rightarrow i} m_i$, then we have

\begin{equation}
\label{intro_eq_semiderived_exchange_relation2}
\mathcal{L}_v(m)\mathcal{L}_v(m') = a(v)K(mm', \prod_{i\rightarrow k} m_i)\overrightarrow{\prod_{i\rightarrow k}} \mathcal{L}_v(m_i)+   b(v)\overrightarrow{\prod_{j\leftarrow k}} \mathcal{L}_v(m_j). 
\end{equation} 
\end{itemize}
In order to study the cluster algebra structure on the Grothendieck rings of certain subcategories $\mathscr{C}_\mathfrak{g}^{[a,b],\mathfrak{s}}$ of the category $\mathscr{C}_\mathfrak{g}^\mathbb{Z}$, where $[a,b]$ is an integer interval, Kashiwara--Kim--Oh--Park introduced the combinatorics of \emph{chains of $i$-boxes}. For each chain of $i$-boxes $\mathfrak{C}=(c)_k$ associated to the interval $[a,b]$, they proved the existence of a matrix $B(\mathfrak{C})$ and of an isomoprhism $\varphi(\mathfrak{C}):\mathcal{A}(Q(\mathfrak{C}))\xrightarrow{\sim} \mathcal{K}(\mathscr{C}_\mathfrak{g}^{[a,b],\mathfrak{s}})$, where $\mathcal{A}(Q(\mathfrak{C}))$ is the cluster algebra associated to the matrix $B(\mathfrak{C})$. We have that the category $\mathscr{C}_\mathfrak{g}^{[-\infty,\infty],\mathfrak{s}}$ coincides with $\mathscr{C}_\mathfrak{g}^\mathbb{Z}$.
Therefore, for each chain of $i$-boxes $\mathfrak{C}$ of range $[-\infty,\infty]$, we are searching for a quiver $\widetilde{Q}(\mathfrak{C})$ and a $\Lambda$-matrix $\widetilde{\Lambda}(\mathfrak{C})$ such that the associated quantum cluster algebra is isomorphic to the twisted semi-derived Hall algebra and such that, specializing the $K_{\alpha,i}$ to 1, we obtain the quantum cluster algebra structure of the quantum Grothendieck ring.
Keeping this in mind, equations (\ref{intro_eq_semiderived_exchange_relation}) and (\ref{intro_eq_semiderived_exchange_relation2}) inspire us to define $\widetilde{Q}(\mathfrak{C})$ by adding certain frozen vertices to the quiver $Q(\mathfrak{C})$, whose associated frozen quantum cluster variables are meant to correspond to certain $K_{\alpha,i}$. 
As a first step, we assume that the interval $[a,b]$ is bounded on the right. In this case, Kashiwara--Kim--Oh--Park define explicitly the quiver associated to the chain of $i$-boxes $\mathfrak{C}_-^{[a,b]}$. For any other chain of $i$-boxes $\mathfrak{C}'$ on the interval $[a,b]$, the quiver $Q(\mathfrak{C}')$ can be obtained from the quiver $Q(\mathfrak{C}_-^{[a,b]})$ through a finite sequence of mutations. In analogy to \cite{HFOO_iso_quant_groth_ring_clust_alg}, we define a $\Lambda$-matrix $\Lambda^{[a,b]}$ such that the associated quantum cluster $\mathbb{Q}(v^{1/2})$-algebra $A_v(Q(\mathfrak{C}_-^{[a,b]}),\Lambda^{[a,b]})$ is isomorphic to the specialized quantum Grothendieck ring $\mathcal{K}_v(\mathscr{C}_\mathfrak{g}^{[a,b],\mathfrak{s}})$ (cf.~Theorem~\ref{teo_quantum_analogue_monoidal_cat}). We define a quiver $\widetilde{Q}^{[a,b],\mathfrak{s}}$ (and a companion $\Lambda$-matrix) such that the associated quantum cluster algebra $\mathcal{A}_v(\widetilde{Q}^{[a,b],\mathfrak{s}})$ fits into a commutative square of the form

\begin{equation}
\label{intro_commutative_square_[a,b]}
\begin{tikzcd} \mathcal{A}_v(\widetilde{Q}^{[a,b],\mathfrak{s}}) \arrow[d, "\pi", twoheadrightarrow] \arrow[rr, dashed, "\Tilde{\theta}"] & &\mathcal{SDH}_{tw}(C^b(\mathcal{P}))^{[a,b],\mathfrak{s}} \arrow[d, "\pi^{\mathcal{H}}", twoheadrightarrow] \\
A_v(Q(\mathfrak{C}_-^{[a,b]}),\Lambda^{[a,b]}) \arrow[r, "\sim", "\theta"' ] & \mathcal{K}_v(\mathscr{C}_\mathfrak{g}^{[a,b],\mathcal{D}_\mathcal{Q},\underline{w}_0}) \arrow[r, "\Phi"', "\sim"] &\mathcal{DH}_{tw}(Q)^{[a,b],\mathfrak{s}}.
\end{tikzcd}
\end{equation}

In more detail, the quiver $\widetilde{Q}^{[a,b],\mathfrak{s}}$ is defined
by adding to the quiver $Q(\mathfrak{C}^{[a,b]}_-)$ certain frozen vertices $\widetilde{S}_{j,z}$ $(j\in I_\mathfrak{g}),\ z\in\mathbb{Z}$,  in such a way that, up to replacing each frozen cluster variable $X_{\widetilde{S}_{j,z}} $ with $K_{\overline{S}_j,z}$ and each initial non-frozen quantum cluster variable $X$ of $\mathcal{A}_v(\widetilde{Q}^{[a,b],\mathfrak{s}})$ with $\mathcal{L}_v(m)$, where  $m\in \mathcal{M}^+$ is the dominant monomial such that $\theta(\pi(X))=L_t(m)$, each exchange relation associated to the mutation of an initial cluster variable is of the form (\ref{intro_eq_semiderived_exchange_relation}) and (\ref{intro_eq_semiderived_exchange_relation2}). The $\Lambda$-matrix associated to the quiver$\widetilde{Q}^{[a,b],\mathfrak{s}}$ is defined by extending the $\Lambda$-matrix associated to the quiver $Q(\mathfrak{C}^{[a,b]}_-)$  by entries equal to zero. This construction allows us to endow the resulting quantum cluster algebra $\mathcal{A}_v(\widetilde{Q}^{[a,b],\mathfrak{s}})$ with a $\mathcal{K}(C^b(\mathcal{P})$)-grading, compatible with that of the semi-derived Hall algebra. In this setting, we prove the following:

\begin{teo}[Theorem \ref{teo_quantum_clust_semiderived_[a,b]}]
\label{intro_teo_quantum_clust_semiderived_[a,b]}
    There is an isomorphism of $\mathcal{K}(C^b(\mathcal{P}))$-graded $\mathbb{Q}(v^{1/2})$-algebras 
    \[\widetilde{\theta}: \mathcal{A}_v(\widetilde{Q}^{[a,b],\mathfrak{s}}) \xrightarrow{\sim} \mathcal{SDH}_{tw}(C^b(\mathcal{P}))^{[a,b],\mathfrak{s}} \]
which makes the diagram (\ref{intro_commutative_square_[a,b]}) commute.
\end{teo}

One central step in the proof of Theorem \ref{intro_teo_quantum_clust_semiderived_[a,b]} is showing that, if $X$ is a quantum cluster variable such that $\Tilde{\theta}(X)=\mathcal{L}_v(m)$, then we have the equality
\begin{equation}
\label{intro_eq_degree_cluster_Ltm}
\mathrm{deg}(X)=\mathrm{deg}(\mathcal{L}_v(m))
\end{equation}
in the Grothendieck group of $C^b(\mathcal{P})$.
To prove this, we rely on the observation that the quiver $\widetilde{Q}^{[a,b],\mathfrak{s}}$ is a subquiver of the principally framed quiver (cf. section \ref{subsect_green_red_vertices}) of $Q(\mathfrak{C}^{[a,b]}_-)$, which, ultimately, enables us to prove that, at each mutation, the arrows incident to the frozen vertices are placed in such a way that, under $\widetilde{\theta}$, the exchange relation is of the form (\ref{intro_eq_semiderived_exchange_relation}) or (\ref{intro_eq_semiderived_exchange_relation2}).

Let $\mathfrak{C}=(\mathfrak{c}_k)_{k\geq 1}$ be a chain of $i$-boxes of range $[-\infty,\infty]$ and consider the sequence $(\mathfrak{C}_s)_{s\geq 1}$ of chains 
of $i$-boxes defined by $\mathfrak{C}_s=(\mathfrak{c}_k)_{1\leq k\leq s}$.
For any $s\geq 1$, let $[a_s,b_s]$ be the range of $\mathfrak{C}_s$. Let $\widetilde{Q}(\mathfrak{C}_s)$ be the quiver obtained from $\widetilde{Q}(\mathfrak{C}_-^{[a_s,b_s]})\cong \widetilde{Q}^{[a_s,b_s],\mathfrak{s}}$ using a mutation sequence which transforms the full subquiver $Q(\mathfrak{C}_-^{[a_s,b_s]})$ into the full subquiver $Q(\mathfrak{C}_s)$ Then, as a consequence of Theorem \ref{intro_teo_quantum_clust_semiderived_[a,b]}, we have commutative diagrams

\[
\begin{tikzcd}
    \mathcal{A}_v(\widetilde{Q}(\mathfrak{C})_s) \arrow[d, twoheadrightarrow] \arrow[r, "\sim"]  &\mathcal{SDH}_{tw}(C^b(\mathcal{P}))^{[a_s,b_s],\mathfrak{s}}\arrow[d, twoheadrightarrow] \\
    \mathcal{A}_v(Q(\mathfrak{C}_s)) \arrow[r, "\sim"] &\mathcal{DH}_{tw}(Q)^{[a_s,b_s],\mathfrak{s}}.
\end{tikzcd}
\]
By defining the quiver $\widetilde{Q}(\mathfrak{C})$ as the colimit of the quivers $\widetilde{Q}(\mathfrak{C}_s)$ and considering the colimit of these diagrams, we obtain the main result of section \ref{sec_quantu_clust_semider}:

\begin{teo}[Corollary \ref{cor_clust_semiderived}]   
    There is a unique isomorphism of $\mathcal{K}(C^b(\mathcal{P}))$-graded $\mathbb{Q}(v^{1/2})$-algebras 
    \[\widetilde{\theta}: \mathcal{A}_v(\widetilde{Q}(\mathfrak{C})) \xrightarrow{\sim} \mathcal{SDH}_{tw}(C^b(\mathcal{P})) \]
such that, for any frozen vertex $\widetilde{S}_{j,z}$ and $i \in K(\mathfrak{C})$, the image of a frozen quantum cluster variable $X_{S_{j,z}}$ is $K_{\overline{S}_j,z}$ and that of a quantum cluster variable $X$ is $\mathcal{L}_v(m)$,
where $m$ is the dominant monomial such that $\varphi_v(\mathfrak{C})(\pi(X))=L_t(m)$.
Moreover, the isomorphism $\widetilde{\theta}$ makes the following diagram commute:

\[
\begin{tikzcd}
    \mathcal{A}_v(\widetilde{Q}(\mathfrak{C})) \arrow[d, "\pi", twoheadrightarrow] \arrow[rr, dashed, "\Tilde{\theta}"] & &\mathcal{SDH}_{tw}(C^b(\mathcal{P}))\arrow[d, "\pi^{\mathcal{H}}", twoheadrightarrow] \\
    \mathcal{A}_v(Q(\mathfrak{C})) \arrow[r, "\sim", "\varphi_v(\mathfrak{C})"'] & \mathcal{K}_v(\mathscr{C}_\mathfrak{g}^{\mathbb{Z}}) \arrow[r, "\Phi"', "\sim"] &\mathcal{DH}_{tw}(Q).
\end{tikzcd}
\]
\end{teo}

\subsection{Braid group action on the semi-derived Hall algebra}
Since its inception, quantum group theory has been awash with braid group actions. Among these, a prime example is Lusztig's braid group action on quantum groups and their representations \cite{Lusztig_book}. In \cite[Thm.~2.3]{KKOP_braid_2020} Kashiwara--Kim--Oh--Park announced formulas for a braid group action on the quantum Grothendieck ring of $\mathscr{C}_\mathfrak{g}$, which were later proved by Jang--Lee--Oh in \cite[Thm.~8.1]{Jang-Lee_Oh_braid_virtual_2023}. In section~\ref{sec_braid} we lift their braid group action to the twisted semi-derived Hall algebra.

For any $i$ in $I_\mathfrak{g}$ and $m$ in $\mathbb{Z}$, we write $E_{i,m}$ for a certain rescaling of $E_{S_i,m}$ and $K_{i,m}$ for $K_{\overline{S}_i,m}$, where $S_i$ and $\overline{S}_i$ are, respectively, the simple $\Bbbk Q$-module associated to the vertex $i$ and its representative in the Grothendieck group of $\mathrm{mod}(\Bbbk Q)$. The $(E_{i,m})$ and the $(K_{i,m})$ generate the twisted semi-derived Hall algebra (cf.~Proposition ~\ref{prop_presenentation_semiderived_E_i}).

\begin{teo}[Theorem \ref{teo_braid}]
\label{intro_teo_braid}
The braid group $B_\mathfrak{g}$ acts on the twisted semi-derived Hall algebra $\mathcal{SDH}_{tw}(Q)$ by the following formulas
\[\sigma_i(E_{j,m})=\begin{cases}
E_{i,m+1}K_{i,m}^{-1}, & (\alpha_i,\alpha_j)=2,\\[0.4cm]
\frac{v^{1/2}E_{i,m}E_{j,m}-v^{-1/2}E_{j,m}E_{i,m}}{v-v^{-1}}, &(\alpha_i,\alpha_j)=-1,\\[0.4cm]
E_{j,m}, &\text{otherwise}.
\end{cases}\]

\[\sigma_i(K_{j,m})=\begin{cases}
K_{i,m}^{-1}, & (\alpha_i,\alpha_j)=2,\\
K_{i,m}K_{j,m}, &(\alpha_i,\alpha_j)=-1,\\
K_{j,m}, &\text{otherwise}.
\end{cases}\]
In particular, when we quotient the semi-derived Hall algebra by the ideal $I((K_{i,m}-1)_{i,m})$, then, via the isomorphism of Theorem \ref{teo_isoHL}, the action specializes to Kashiwara-Kim-Oh-Park's action  on the quantum Grothendieck ring. 
\end{teo}

The proof of Theorem \ref{intro_teo_braid} consists of the (rather long) direct check that the $\sigma_i$ respect the defining relations of the twisted semi-derived Hall algebra.
Notice that the action of the $\sigma_i$ on the $K_{j,m}$ is inspired by Lusztig's action on the quantum group $U_q(\mathfrak{g})$.

\section{Quantum cluster algebra strcutures on quantum Grothendieck rings}
\label{sect_module_categories}
In this section, we recall the definition of the subcategories $\mathscr{C}_{\mathfrak{g}}^{[a,b],\mathfrak{s}}$ of the module category of the quantum loop algebra. Moreover, we show that their quantum Grothendieck rings carry a quantum cluster algebra structure.

Let $\mathfrak{g}$ be a finite-dimensional simple complex Lie algebra of simply-laced type. We write $I_\mathfrak{g}=\{1, \ldots, n\}$ for the set of vertices of its Dynkin diagram and we denote the simple roots of $\mathfrak{g}$ by $(\alpha_i)_{i\in I_\mathfrak{g}}$. Let $q$ be an indeterminate. We write $\mathbb{K}$ for the algebraic closure of the field $\mathbb{Q}(q)$. 

We write $U_q(L\mathfrak{g})$ for the quantum loop algebra associated to $\mathfrak{g}$. It is the $\mathbb{K}$-algebra defined via the generators 
\[ \{K_i^{\pm}| i\in I_{\mathfrak{g}}\}\cup \{X_{i,s}, Y_{i,s}| i\in I_{\mathfrak{g}}, s\in \mathbb{Z} \}\cup \{H_{i,r}|i\in I_{\mathfrak{g}}, r \in \mathbb{Z}-\{0\}\}
\] 
subject to the relations given in \cite[Definition 2.1]{HFOO_propagation_positivity}.

\subsection{The categories $\mathscr{C}_{\mathfrak{g}}^{[a,b],\mathfrak{s}}$}\label{subsec_categoryC0}
Recall that a finite-dimensional $U_q(L\mathfrak{g})$-modules is \emph{of type 1} if the generators $K_1, \dots, K_n$ act semisimply, 
with eigenvalues which are integer powers of $q$. We write $\mathscr{C}_\mathfrak{g}$ for the category of finite-dimensional $U_q(L\mathfrak{g})$-modules of type 1. Thanks to the Hopf algebra structure of $U_q(L\mathfrak{g})$ (cf.~{\cite[Def.-Prop.~9.1.1]{Chari_Pressley_guide}), the category $\mathscr{C}_\mathfrak{g}$ is monoidal. 
Two simple modules $M$ and $N$ of $\mathscr{C}_\mathfrak{g}$ \emph{commute} if there is an isomorphism $M\otimes N \cong N\otimes M$.

Consider the Laurent polynomial ring $\mathbb{Z}[Y_{i,a}^\pm]_{i\in I,a\in \mathbb{K}^\times}$. We call a monomial in this ring \emph{dominant} if it is monic and its reduced expression does not contain negative powers of the $Y_{i,a}$. 
The theory of $q$-characters (see \cite{FrenkelReshetikhin98}) provides us with an injective ring homomorphism $\chi_q: \mathcal{K}(\mathscr{C}_\mathfrak{g}) \rightarrow \mathbb{Z}[Y_{i,a}^\pm]_{i\in I,a\in \mathbb{K}^\times}$ and a parametrisation
$m \mapsto L(m)$ of the simple objects in $\mathscr{C}_\mathfrak{g}$ by the dominant monomials $m$. Notice that the existence of the monomorphism $\chi_q$ implies that Grothendieck ring $\mathcal{K}(\mathscr{C}_\mathfrak{g})$ is commutative. 

Let $\tilde{\varepsilon}$ be a \emph{parity function} on the Dynkin diagram of $\mathfrak{g}$, that is, a function \[\tilde{\varepsilon}:I_\mathfrak{g}\rightarrow \{0,1\}, i\mapsto \tilde{\varepsilon}_i\] such that $\Tilde{\varepsilon}_i\equiv\Tilde{\varepsilon}_j \text{ mod 2}$ whenever the vertices $i$ and
$j$ are linked by an edge. We write $\widehat{I}_\mathfrak{g}$ for the set
\[\widehat{I}_\mathfrak{g}=\{(i,p)\in I_\mathfrak{g}\times\mathbb{Z}\ |\ p-\Tilde{\varepsilon}_i\in 2\mathbb{Z}\}.\]
In the following, we will use the 
simplified notation $Y_{i,p}=Y_{i,q^p}$.
Write $\mathcal{M}$\label{symb_mathcal_M} for the set of  monomials in the variables $(Y_{i,p}^{\pm1})_{(i,p)\in \widehat{I}_{\mathfrak{g}}}$ and let $\mathcal{M}^+$\label{symb_mathcal_M_+} be the subset of dominant monomials contained in $\mathcal{M}$.

We write $\mathscr{C}_\mathfrak{g}^\mathbb{Z}$ for the monoidal Serre subcategory of $\mathscr{C}_\mathfrak{g}$ whose simple objects are the
$L(m)$, where $m$ is a monomial in $\mathcal{M}^+$. 
For $(i,p-1)\in\widehat{I}_{\mathfrak{g}}$, set
\[ A_{i,p} = Y_{i,p-1}Y_{i,p+1}\prod_{(j,s)\in \widehat{I}_{\mathfrak{g}},\ i\sim j} Y_{j,p}^{-1}. \]\label{symb_A_ip}
Nakajima defined a partial order on the set $\mathcal{M}$:
\[m' \leq m \text{ if } m(m')^{-1} \text{ is a monomial in } A_{i,p+1}, (i,p)\in \widehat{I}_{\mathfrak{g}}.\]\label{symb_order_monomials}
In the following, we will denote the induced partial order on $\mathcal{M}^+$ by the same symbol.

For  $(i,p)\in \widehat{I}$ and an integer $k\geq 1$, we define the 
\emph{Kirillov--Reshetikhin module $W_{k,p}^i\in \mathscr{C}_\mathfrak{g}^0$} as
\[W_{k,p}^i=L(Y_{i,p}Y_{i,p+2}\dots Y_{i,p+2(k-1)}).\]
For $k=0$, the module $W_{k,p}^i$ is defined to be the trivial module.

Let $w_0$ be the longest element of the Weyl group of $\mathfrak{g}$. Define the involution $(-)^*: I_\mathsf{g}\rightarrow I_{\mathsf{g}}$
by $w_0(\alpha_{i})=-\alpha_{i^*}$.
Let $\varepsilon$ be an \emph{height function} on the Dynkin diagram of $\mathfrak{g}$, that is, a function $\varepsilon: I_\mathfrak{g}\rightarrow \mathbb{Z},\ i\mapsto\varepsilon_i$ such that $|\varepsilon_i-\varepsilon_j|=1$ if $i\sim j$. The height function $\varepsilon$ induces an orientation of the edges of the Dynkin diagram of $\mathfrak{g}$, such that, if $i\sim j$ and $\varepsilon_i>\varepsilon_j$, then there is an arrow $i\rightarrow j$.
We write $Q_\varepsilon$ for the corresponding quiver. A vertex $i$ in $I_\mathfrak{g}$ is a \emph{sink} of $Q_\varepsilon$ if $\varepsilon_i< \varepsilon_j$ for any $j\sim i$. If $i$ is a sink of $Q_\varepsilon$, we denote by $s_i\varepsilon$ the height function defined by $(s_i\varepsilon)_j=\varepsilon_j+2\delta_{ij}$. A reduced expression $\underline{w}=s_{i_1}\dots s_{i_r}$ of an element of $W_{\mathfrak{g}}$ is \textit{adapted} to $\varepsilon$ if, for any $1\leq k\leq r$, the vertex $i_k$ is a sink of  the quiver associated to the height function $s_{i_{k-1}}s_{i_{k-2}}\dots s_{i_1}\varepsilon$. In the following, we assume that the height function $\varepsilon$ satisfies $\varepsilon_i=\widetilde{\varepsilon}_i\ \mathrm{mod}\ 2$, for any $i$ in $I_\mathfrak{g}$.

Fix a reduced expression $\underline{w}_0=s_{i_1}\ \dots\ s_{i_l}$ of $w_0$ adapted to $\varepsilon$ and extend the sequence
$i_1, \ldots, i_l$ to an infinite sequence $\widehat{\underline{w}}_0=(i_k)_{k\in\mathbb{Z}}$ in such a way that
\[
i_{k+l}=(i_k)^*
\]
for all $k\in \mathbb{Z}$. 
The \emph{admissible sequence} $\mathfrak{s}$ associated to the pair $(\varepsilon,\underline{w}_0)$ is the sequence of pairs $(i_k,p_k)$, $k\in\mathbb{Z}$, where
$i_k\in I_{\mathfrak{g}}$ and $p_k\in \mathbb{Z}$, defined by
    \[
i_k=(\widehat{\underline{w}}_0)_k \text{  and  } p_k= 
\begin{cases}
    (s_{i_{k-1}}\dots s_{i_1}\varepsilon)_{i_k}, & k\geq 1,\\
    (s_{i_{k-1}}^{-1}\dots s_{i_1}^{-1}\varepsilon)_{i_k}, & k\leq 0.
\end{cases}
\]

For $a\leq b\in\mathbb{Z}\sqcup\{\pm\infty\}$, following \cite{KKOP_mon_cat_quant_aff_II}, we define the integer interval $[a,b]=\{k\in\mathbb{Z} \;|\; a\leq k\leq b\}$.
When $a$ and $b$ are integers, the \textit{length} of  the interval $[a,b]$ is $l=b-a+1$. Otherwise, the interval
has \emph{infinite length} and, by abuse of terminology, we call $l=\infty$ the \textit{length of the chain}.

Following \cite[Def.~6.16]{KKOP_mon_cat_quant_aff_II}, the category 
$\mathscr{C}_{\mathfrak{g}}^{[a,b],\mathfrak{s}}$\label{symb_C_[a,b]_D_w0} is defined as
the monoidal Serre subcategory of $\mathscr{C}_\mathfrak{g}$ generated by the fundamental modules 
$L(Y_{i_k,p_k})$, for $k\in [a,b]$.

\begin{rem}
    The category $\mathscr{C}_{\mathfrak{g}}^{[a,b],\mathfrak{s}}$ is a particular case of the categories $\mathscr{C}_{\mathfrak{g}}^{[a,b],\mathcal{D}, \underline{w}_0}$ introduced in \cite[$\S$ 6.3]{kashiwara2021pbw} and used in \cite{KKOP_mon_cat_quant_aff_II}. More precisely, we have
    \[  \mathscr{C}_{\mathfrak{g}}^{[a,b],\mathfrak{s}} =  \mathscr{C}_{\mathfrak{g}}^{[a,b],\mathcal{D}_\mathcal{Q}, \underline{w}_0}, \] where $\mathcal{D}_\mathcal{Q}$ is the complete duality datum associated to the $\mathcal{Q}$-datum induced by $\varepsilon$. Notice that in the general definition of these categories the reduced expression $\underline{w}_0$ does not need to be adapted to $\varepsilon$.
\end{rem}

\subsection{Quantum Grothendieck rings}
\label{sec_def_Groth_rings} 
The \emph{quantum Cartan matrix} $A(q)$ of $\mathfrak{g}$ is the $n\times n$-matrix $A(q)=(q+q^{-1})I_n-J$, where $J$ is the adjacency matrix of the Dynkin graph of $\mathfrak{g}$. The matrix $A(q)$ is invertible and we can write the entries $\tilde{A}_{ij}(q)$ of its inverse $\tilde{A}(q)$ as
\[\tilde{A}(q)_{ij}=\sum_{m\geq 1}a_{ij}(m)q^m\in \mathbb{Z}((q)),\]
for some coefficients $a_{ij}(m)$ in $\mathbb{Z}$ (cf.~\cite[\S~2.4]{HL_quantum_Groth_rings_derived_Hall}).
Following \cite[\S~3]{HL_quantum_Groth_rings_derived_Hall}, for any $m\leq 0$, set $a_{ij}(m)=0$. For $i,j\in I_\mathfrak{g},\ p,s \in \mathbb{Z}$, define
$$\mathcal{N}(i,p;j,s)=
\begin{cases}
a_{ij}(s-p+1)-a_{ij}(s-p-1)&  \text{for } p<s,\\
0& \text{for } p=0,\\
a_{ij}(p-s-1)-a_{ij}(p-s+1)&\text{for } p>s.
\end{cases}$$
Notice that, Since $\tilde{A}(q)$ is symmetric, we have $N(i,p;j,s)=-N(j,s;i,p)$.

\begin{de}\cite[\S~3]{HL_quantum_Groth_rings_derived_Hall}
 The quantum torus $\mathcal{Y}_t$ is the $\mathbb{Z}[t^{\pm 1/2}]$-associative\label{symb_t} algebra presented by the variables $(\leftidx{^t}{Y}_{i,p}^{\pm 1})_{(i,p)\in \widehat{I}_{\mathfrak{g}}}$ and the relations:
\[\leftidx{^t}{Y}_{i,p}\leftidx{^t}{Y}_{i,p}^{-1}=1=\leftidx{^t}{Y}_{i,p}^{-1}\ \leftidx{^t}{Y}_{i,p},\]
\[\leftidx{^t}{Y}_{i,p}\leftidx{^t}{Y}_{j,s}=t^{\mathcal{N}(i,p;j,s)}\ \leftidx{^t}{Y}_{j,s}\leftidx{^t}{Y}_{i,p}.
\]
\end{de}
Following \cite[\S~3]{HL_quantum_Groth_rings_derived_Hall}, we collect the following terminology and notation:
\begin{itemize}
\setlength\itemsep{1.4em}
    \item Let $\mathrm{ev}_{t=1}: \mathcal{Y}_t\rightarrow \mathcal{Y}$ be the surjective $\mathbb{Z}$-algebra homomorphism defined by
    \[ t^{1/2}\mapsto 1 \text{ and } \leftidx{^t}{Y}_{i,p}\mapsto Y_{i,p};\]
    \item a \emph{monomial} $m\in \mathcal{Y}_t$ is a product of the generators $\leftidx{^t}{Y}_{i,p}^{\pm}$ and $t^{\pm 1/2}$;
    \item a monomial $m\in \mathcal{Y}_t$ is \emph{dominant} if $\mathrm{ev}_{t=1}(m)$ is dominant;
    \item the \emph{bar involution} $\overline{(-)}$ is the $\mathbb{Z}$-algebra anti-involution on $\mathcal{Y}_t$ defined by
    \[ \overline{t^{1/2}}=t^{-1/2},\ \ \overline{\leftidx{^t}{Y}_{i,p}}=\leftidx{^t}{Y}_{i,p}.     \]

\item a \emph{commutative monomial} in $\mathcal{Y}_t$ is a monomial of the form 

\[t^{\frac{1}{2}\sum_{(i,p)<(j,s)}u_{i,p}u_{j,s}\mathcal{N}(j,s;i,p)}\underset{(i,p)\in \widehat{I}_{\mathfrak{g}}}{\overrightarrow{\cdot}}\leftidx{^t}{Y}_{i,p}^{u_{i,p}}\]
where $(u_{i,p})_{(i,p)\in \widehat{I}_{\mathfrak{g}}}$ is a family of integers with finitely many nonzero components and $<$ is an arbitrary total order of $\widehat{I}_{\mathfrak{g}}$ such that $(i,p)<(j,s)$ if and only if $p<s$. Since the resulting expression does not depend on the choice of the total order we denote it as $\prod_{(i,p)\in\widehat{I}_{\mathfrak{g}}}\leftidx{^t}{Y}_{i,p}^{u_{i,p}}$;

\item for a monomial $m\in \mathcal{Y}$, we denote by $\leftidx{^t}{m}$ the commutative monomial in $\mathcal{Y}_t$ such that $\mathrm{ev}_{t=1}(\leftidx{^t}{m})=m$.
\end{itemize}

The commutative monomials are bar-invariant and form a basis of $\mathcal{Y}_t$ as a $\mathbb{Z}[t^{\pm 1/2}]$-module.
The non-commutative product of two commutative monomials $\leftidx{^t}{m_1}$ and $\leftidx{^t}{m_2}$ is given by 
\begin{equation}
\label{eq_comum}
\leftidx{^t}{m_1}\leftidx{^t}{m_2}=t^{\frac{1}{2}\mathcal{N}(m_1,m_2)}\leftidx{^t}{(m_1m_2)}=t^{\frac{1}{2}\mathcal{N}(m_1,m_2)}\leftidx{^t}{m_2}\leftidx{^t}{m_1}.
\end{equation}
where 
\begin{equation}
\mathcal{N}(m_1,m_2)=\sum_{(i,p),(j,s)\in \widehat{I}_{\mathfrak{g}}}u_{i,p}(m)u_{j,s}(m_2)N(i,p;j,s)
\end{equation}

\begin{rem}
\label{rem_invariant_Lambda}
Kashiwara–Kim–Oh–Park \cite[Def. 3.6]{KKOP_mon_cat_quant_aff} have associated  an integer denoted by $\Lambda(M,N)$ to each pair of simple modules $M$ and $N$ in $\mathscr{C}_\mathfrak{g}$.
If $m_1$ and $m_2$ are dominant monomials such that the simple modules $L(m_1)$ and $L(m_2)$ commute, then, as shown in \cite[Cor.~5.15]{Fuj_Oh_2021}, we have the equality 
\[ \Lambda(L(m_1),L(m_2)) = \mathcal{N}(m_1,m_2).\]
\end{rem}

For any $i$ in $I$, we define $\mathcal{K}_{i,t}$ as the $\mathbb{Z}[t^{\pm1/2}]$-subalgebra of $\mathcal{Y}_t$ generated by the commutative polynomials
\[\leftidx{^t}{Y}_{i,p}(1+\leftidx{^t}{A}_{i,p+1}^{-1}),\ \leftidx{^t}{Y}_{j,s}^{\pm 1},\quad \big((i,p),(j,s)\in \widehat{I}_{\mathfrak{g}},\ i\neq j\big).\] The \emph{quantum Grothendieck ring} of $\mathscr{C}_\mathfrak{g}^{\mathbb{Z}}$ is defined (cf.~\cite{Hernandez2004_alg_approach,HL_quantum_Groth_rings_derived_Hall}) as the intersection
\[\mathcal{K}_t(\mathscr{C}_\mathfrak{g}^{\mathbb{Z}})=\bigcap_{i\in I}\mathcal{K}_{i,t}.\label{symb_K_t(C_g^z)}\]

As the following result shows, we can think of $\mathcal{K}_t(\mathscr{C}_\mathfrak{g}^{\mathbb{Z}})$ as a deformation of the Grothendieck ring of $\mathscr{C}_\mathfrak{g}^{\mathbb{Z}}$.

\begin{prop}[{\cite[Thm.~6.2]{Hernandez2004_alg_approach}}]
The image of the quantum Grothendieck ring $\mathcal{K}_t(\mathscr{C}_\mathfrak{g}^{\mathbb{Z}})$ under the evaluation map $\mathrm{ev}_{t=1}$ coincides with the image of the Grothendieck ring $\mathcal{K}(\mathscr{C}_\mathfrak{g}^{\mathbb{Z}})$ under the $q$-character map:
\[  \mathrm{ev}_{t=1}(\mathcal{K}_t(\mathscr{C}_\mathfrak{g}^{\mathbb{Z}}) ) = \chi_q ( \mathcal{K}(\mathscr{C}_\mathfrak{g}^{\mathbb{Z}}) ).\]
\end{prop}

For any dominant monomial $m\in\mathcal{M}^+$, we recall some notable elements of the quantum Grothendieck ring $\mathcal{K}_t(\mathscr{C}_\mathfrak{g}^{\mathbb{Z}})$, namely $F_t(m), L_t(m)$ and $M_t(m)$.\\

$F_t(m)$\label{symb_F_t(m)} is the unique element of $\mathcal{K}_t$ such that $\leftidx{^t}{m}$ is the unique dominant monomial in $F_t(m)$ (see \cite[Thm.~5.11]{Hernandez2004_alg_approach} for its existence and uniqueness).

\begin{prop}[{\cite[Thm.~5.11]{Hernandez2004_alg_approach}}]\mbox{}
\begin{itemize}
\setlength\itemsep{1.2em}
    \item For any dominant monomial $m$, the element $F_t(m)$ is invariant under the bar involution.
    \item The elements $(F_t(Y_{i,p}))_{(i,p)\in \widehat{I}_{\mathfrak{g}}}$ generate $\mathcal{K}(\mathscr{C}_\mathfrak{g}^{\mathbb{Z}})$ as a $\mathbb{Z}[t^{\pm 1/2}]$-algebra.
\end{itemize}
\end{prop}

Next, we define $M_t(m)$, the \emph{$(q,t)$-character of the standard module} $M(m)$, as

\[M_t(m)=\leftidx{^t}{m}(\underset{p\in\mathbb{Z}}{\overrightarrow{\cdot}}\prod_{i\in I,  (i,p)\in\widehat{I}_{\mathfrak{g}} }\leftidx{^t}{Y}_{i,p}^{u_{i,p}(m)})^{-1}(\underset{p\in\mathbb{Z}}{\overrightarrow{\cdot}}\prod_{i\in I, (i,p)\in\widehat{I}_{\mathfrak{g}} }F_t(Y_{i,p})^{u_{i,p}(m)}),\label{symb_M_t(m)}
\]

\begin{prop}
\mbox{}
\begin{itemize}
\setlength\itemsep{1.2em}
\item For any dominant monomial $m$, the evalutaion of $M_t(m)$ at $t=1$ coincides with the $q$-character $\chi_q(M(m))$ (\cite[Thm.~6.3]{Hernandez2004_alg_approach})
\item For any dominant monomial $m$, 
\[M_t(m)=\leftidx{^t}{m}+\sum_{m'<m}a_{m'}\leftidx{^t}{m'},\]
where the coefficients $a_{m'}$ are in $\mathbb{N}[t^{\pm 1}]$ (\cite[\S~8]{Nakajima2004}).
\item The elements $(M_t(m))_{m\in \mathcal{M}^+}$ form a $\mathbb{Z}[t^{\pm 1/2}]$-basis of $\mathcal{K}(\mathscr{C}_\mathfrak{g}^{\mathbb{Z}})$ (\cite[Thm.~6.2]{Hernandez2004_alg_approach})
\end{itemize}
\end{prop}

Finally, for a dominant monomial $m\in\mathcal{M}^+$, we recall $L_t(m)$\label{symb_L_t(m)}, the \emph{$(q,t)$-character of the simple module} $L(m)$.
It is the unique bar-invariant element $L_t(m)$ of $\mathcal{K}(\mathscr{C}_\mathfrak{g}^{\mathbb{Z}})$ such that it can be expressed as a linear combination of the form
\[
L_t(m)=M_t(m)+\sum_{m<m'}a_{m,m'}M_t(m'),
\]
where the coefficients $a_{m,m'}$ are in $t\mathbb{Z}[t].$ (see \cite[Thm.~8.1]{Nakajima2004} for its existence and uniqueness).

\begin{prop}
\mbox{}
\label{prop_properties_qtcharact_simple}
\begin{itemize}
\setlength\itemsep{1.2em}
\item If $L(m)$ is a $KR$ module, then $F_t(m)=L_t(m)$ $\mathrm{(}$\cite[Cor.~6.11]{HFOO_iso_quant_groth_ring_clust_alg}$\mathrm{)}$
\item For any dominant monomial $m$, 
\[L_t(m)=\leftidx{^t}{m}+\sum_{m'<m}b_{m'}\leftidx{^t}{m'},\]
where the coefficients $b_{m'}$ are in $\mathbb{N}[t^{\pm 1}]$ $\mathrm{(}$\cite[Cor.~6.12]{HFOO_iso_quant_groth_ring_clust_alg}$\mathrm{)}$.
\item The elements $(L_t(m))_{m\in \mathcal{M}^+}$ form a $\mathbb{Z}[t^{\pm 1/2}]$-basis of $\mathcal{K}(\mathscr{C}_\mathfrak{g}^{\mathbb{Z}})$ $\mathrm{(}$\cite[Thm.~6.2]{Hernandez2004_alg_approach}$\mathrm{)}$
\item Let $m_1$ and $m_2$ be dominant monomials in $\mathcal{M}^+$ and consider the expansion of the product of $L_t(m_1)$ by $L_t(m_2)$ along the basis of the $(L_t(m))_{m\in \mathcal{M}^+}$: 
\[L_t(m_1)*L_t(m_2)=\sum_mc_{m_{1},m_{2}}^m(t^{1/2})L_t(m).\]
Then the coefficients $c_{m_{1},m_{2}}^m(t^{1/2})$ are in $\mathbb{N}[t^{\pm1/2}]$ $\mathrm{(}$\cite[Cor.~10.7]{HFOO_propagation_positivity}$\mathrm{)}$
\end{itemize}
\end{prop}

For any $(i,p)\in \widehat{I}_\mathfrak{g}$ and $k\in\mathbb{Z}$, let $\leftidx{^t}{W}^i_{k,p}$ be the $(q,t)$-character of the simple module $W^i_{k,p}$.
The following results highligths the central role of the Kirillov--Reshetikhin modules in the connection between cluster theory and quantum group theory. In fact, on one side their representatives in the Grothendieck ring $\mathcal{K}(\mathscr{C}_\mathfrak{g}^{\mathbb{Z}})$ satisfies a recursive system of equations called the \emph{$T$-system}, on the other side their $(q,t)$-characters satisfy a $t$-deformation of it, the \emph{quantum $T$-system}.

\begin{teo}[{\cite[Thm.~3.4]{Hernandez2006_solutionsTsystems},\cite[Thm.~6.8]{Aihara_Iyama_silt_mut_tring_cat}}]
\label{teo_quantum_T_system}
Let $k$ be in $\mathbb{Z}_{>0}$.
The following equality holds in $\mathcal{K}(\mathscr{C}_\mathfrak{g}^\mathbb{Z})$:
\[ 
[W^{i}_{k,p}][W^{i}_{k,p+2}]=[W^{i}_{k+1,p}][W^{i}_{k-1,p+2}]+\prod_{\jmath\sim i}[W^{j}_{k-1,p+1}].\]
Moreover, if we fix a total ordering of the set $\{\jmath\ |\ \jmath\sim i \}$, then there exist $a,b\in \frac{1}{2}\mathbb{Z}$ such that the following equality holds in $\mathcal{K}_t(\mathscr{C}_\mathfrak{g}^\mathbb{Z})$:
\[ 
\leftidx{^t}{W}^{i}_{k,p}\leftidx{^t}{W}^{i}_{k,p+2}=t^a\ \leftidx{^t}{W}^{i}_{k+1,p}\leftidx{^t}{W}^{i}_{k-1,p+2}+t^b\prod_{\jmath\sim i}\leftidx{^t}{W}^{j}_{k-1,p+1}.\]
\end{teo}

For any interval $[a,b]\subset \mathbb{Z}$, we define $\mathcal{K}_t(\mathscr{C}_\mathfrak{g}^{[a,b],\mathfrak{s}})$ as the subalgebra of $\mathcal{K}_t(\mathscr{C}_\mathfrak{g}^{\mathbb{Z}})$ generated by the $L_t(Y_{i_k,p_k})$, for $k\in [a,b]$. If follows from the characterization of the $(q,t)$-characters of the simple modules given in \cite[Thm.~4.3]{HFOO_iso_quant_groth_ring_clust_alg} that $L_t(m)$ belongs to $\mathcal{K}_t(\mathscr{C}_\mathfrak{g}^{[a,b],\mathfrak{s}})$ if and only if $m$ is a monomial in the variables $(Y_{i_k,p_k})_{k\in [a,b]}$.

\subsection{Combinatorics of $i$-boxes and quantum seeds}
In this subsection, following \cite{KKOP_mon_cat_quant_aff_II}, we recall the definition and the properties of $i$-boxes, a combinatorial tool which will serve to parametrize the Kirillov--Reshetikhin modules.

Recall that we have fixed a sequence $\widehat{\underline{w}}_0=(i_k)_{k\in\mathbb{Z}}$.
For any $s\in\mathbb{Z}$, we will use the following notations: 

\begin{align}
s^+& =\text{min}\{t\;|\;s<t,\ i_t=i_s\}, \quad\quad   s^-=\text{max}\{t\;|\;s>t,\ i_t=i_s\},\\
s(\jmath)^+&=\text{min}\{t\;|\;s\leq t,\ i_t=\jmath\}, \quad\quad  s(\jmath)^-=\text{max}\{t\;|\;s\geq t,\ i_t=\jmath\}.
\end{align}

\begin{de}[{\cite[Section 4]{KKOP_mon_cat_quant_aff_II}}] \mbox{}
\begin{itemize} 
    \item An {\em $i$-box} is a finite integer interval $[a,b]$ such that $i_a=i_b$. 
    \item The {\em index} of an $i$-box $[a,b]$ is $\jmath=i_a=i_b$.
    \item The {\em $i$-cardinality} of an $i$-box $[a,b]_\jmath$ is the number of times that the index $\jmath$ appears in the sub-interval of $\widehat{\underline{w}}_0$ corresponding to $[a,b]$.
\end{itemize}
\end{de}
For a finite interval $[a,b]$, we write $[a,b\}$ and $\{a,b]$ for the largest $i$-boxes contained in $[a,b]$ with index $i_a$ and $i_b$ respectively:
\[ [a,b\}=[a,b(i_a)^-] \text{  and  } \{a,b]=[a(i_b)^+,b]. \]
When we want to emphasize that an $i$-box $[a,b]$ is of index $\jmath$, we will use the notation $[a,b]_\jmath$.

\begin{de}[{\cite[Def. 5.1]{KKOP_mon_cat_quant_aff_II}}]
Let $[a,b]$ be a possibly infinite interval of length $l$.
A \textit{chain} of $i$-boxes of \textit{range} $[a,b]$ is a sequence of $i$-boxes $\mathfrak{C}=(\mathfrak{c}_k)_{1\leq k< l+1}$ contained in $[a,b]$ and satisfying the following conditions for any $1\leq s< l+1$:
\begin{itemize}
    \item[(i)] The union $\bigcup_{1\leq k\leq s} \mathfrak{c}_k $ is an interval of length $s$;
    \item [(ii)] The $i$-box $\mathfrak{c}_s$ is the largest $i$-box of its own index contained in the interval $\bigcup_{1\leq k\leq s} \mathfrak{c}_k $.
\end{itemize}

For any  $1\leq s < l+1$, the sequence $(\mathfrak{c}_k)_{1\leq k\leq s}$ is a chain of $i$-boxes. We refer to it as a \textit{subchain} of $\mathfrak{C}$ of length $s$.
\end{de}

Notice that datum of a chain of $i$-boxes $\mathfrak{C}=(\mathfrak{c}_k)_{1\leq k< l+1}$ on an interval $[a,b]$ is equivalent to the data of the singleton 
$\{c\}=\mathfrak{c}_1$ and a sequence $E\in \{L,R\}^{l-1}$, such that, for $1\leq k<l$, $E_k$ is equal to $L$ respectively $R$ if adding $\mathfrak{c}_{k+1}$ increases the range of the interval $\bigcup_{1\leq k\leq s} \mathfrak{c}_k $ by one unit on the left respectively on the right. We call the sequence $E$ together with
the $i$-box $[c]$ the \emph{rooted sequence
of expansion operators} associated with the chain of $i$-boxes $\mathfrak{C}$.

For $a\in \mathbb{Z}\cup \{-\infty\}$ and  $b\in\mathbb{Z},\ b\geq a$,  
we write $\mathfrak{C}_{-}^{[a,b]}$ for the chain of $i$-boxes 
\[
\mathfrak{C}_{-}^{[a,b]}=(\mathfrak{c}^-_k)_{1\leq k \leq b-a+1}=([s,b\})_{b\geq s\geq a } \ .
\] 
Notice that the associated pair is $(b;(E_k)_{1\leq k\leq b-a})$, where $E_k=L$ for any $k$.

Let $\mathfrak{c}=[c,d]_i$ be an $i$-box of $i$-cardinality $m$.
Following \cite[Rem.~6.15] {KKOP_mon_cat_quant_aff_II}, we associate to the $i$-box $\mathfrak{c}$ the Kirillov-Reshetikhin module 
\[M(\mathfrak{c})=W^i_{m,p_c}.\]

The \emph{repetition quiver} $\widehat{Q}$ is the quiver with set of vertices $\widehat{I}_{\mathfrak{g}}$ and arrows:

\begin{itemize}
    \item $(i,p) \rightarrow (i,p-2)$ for any $(i,p) \in \widehat{I}_{\mathfrak{g}}$.
    \item $(i,p) \rightarrow (j,p+1)$ for any $(i,p) \in \widehat{I}_{\mathfrak{g}}$ and $j\sim i$.
\end{itemize}
For any interval $[a,b]$, we denote by $\widehat{Q}^{[a,b],\mathfrak{s}}$ the full subquiver of $\widehat{Q}$ on the set of vertices $(i_k,p_k)_{k\in \mathbb{Z}}$. 
For a chain of $i$-boxes $\mathfrak{C}=(\mathfrak{c}_k)$ of range $[a,b]$ and length $l$, set
\begin{align*}
K(\mathfrak{C})&=
    \{k\ |\ 1\leq k <l+1\}\\
K(\mathfrak{C})^{\text{fr}}&=\begin{cases}
\{s \in K(\mathfrak{C})\ |\ \mathfrak{c}_s=[a(\imath)^+,b(\imath)^-] \text{ for some } \imath \in I_{\check{\mathfrak{g}}} \}  &\text{ if } l<\infty;\\
    \emptyset, &\text{ if } l=\infty;
\end{cases};\\
K(\mathfrak{C})^\mu&=K(\mathfrak{C})\backslash K(\mathfrak{C})^{\text{fr}}.  
\end{align*}

We refer to Appendix \ref{appendix_clust_alg} for the terminology and the properties of (quantum) cluster algebras.

\begin{teo}[{\cite[\S~8]{KKOP_mon_cat_quant_aff_II}}]
\label{teo_KKOP_monoidal_cat}
Let $[a,b]$ be an interval of length $l$ and let $\mathfrak{C}=(\mathfrak{c}_k)_{1\leq k<l+1}$ be a chain of $i$-boxes of range $[a,b]$.
\begin{itemize}
    \item[$\mathrm{(i)}$] There exists an ice quiver $Q(\mathfrak{C})$ and an isomorphism $\varphi(\mathfrak{C})$ between the cluster algebra $\mathcal{A}(Q(\mathfrak{C}))$ and the Grothendieck ring $\mathcal{K}(\mathscr{C}_\mathfrak{g}^{[a,b],\mathfrak{s}})$ (cf. \cite[Thm.~8.1]{KKOP_mon_cat_quant_aff_II}). The set of (resp. frozen)  vertices of $Q(\mathfrak{C})$ is identified with $K(\mathfrak{C})$ (resp. $K^{\mathrm{fr}}(\mathfrak{C})$).
    \item[$\mathrm{(ii)}$] Under the isomorphism $\varphi(\mathfrak{C})$, for any $1\leq k\leq l$, the initial cluster variable $x_k$ is sent to the Kirillov-Reshetikhin module $M(\mathfrak{c}_k)$ and each cluster monomial is sent to a simple module (cf. \cite[Thm.~8.1]{KKOP_mon_cat_quant_aff_II}).
    \item[$\mathrm{(iii)}$] the quiver $Q(\mathfrak{C}_-^{[a,b]})$ is isomorphic to $\widehat{Q}^{[a,b],\mathfrak{s}}$ (cf. \cite{HL_Clust_alg_approach_q_char, KKOP_mon_cat_quant_aff_II}).
    \item[$\mathrm{(iv)}$] If $b$ is an integer, the quiver $Q(\mathfrak{C})$ can be obtained from the quiver $\widehat{Q}^{[a,b],\mathfrak{s}}$ by a finite sequence of mutations (cf. \cite[\S~8]{KKOP_mon_cat_quant_aff_II}).
    \item[$\mathrm{(v)}$] If $[a,b]=[-\infty,\infty]$ and $(\mathfrak{C}_s)_{s\geq 1}$ is the sequence of chains 
     of $i$-boxes defined by $\mathfrak{C}_s=(\mathfrak{c}_k)_{1\leq k\leq s}$, then the quiver $Q(\mathfrak{C})$ is the colimit of the sequence $(Q(\mathfrak{C}_s))_{s\geq 1}$ (cf. \cite[\S~8]{KKOP_mon_cat_quant_aff_II}).
\end{itemize}
\end{teo}

\begin{rem}
\label{rem_sequence_canonical_chain_transform}Let $[a,b]$ be an interval of length $l$ and let $\mathfrak{C}=(\mathfrak{c}_k)_{1\leq k<l+1}$ be a chain of $i$-boxes of range $[a,b]$. We denote by $B(\mathfrak{C})$ the exchange matrix associated to the quiver $Q(\mathfrak{C})$.
\begin{itemize}
    \item The statement of \cite[Thm.~8.1]{KKOP_mon_cat_quant_aff_II} actually yields a stronger result than points (i) and (ii) of Theorem \ref{teo_KKOP_monoidal_cat}. In fact, Kashiwara--Kim--Oh--Park  prove that the category $\mathscr{C}_\mathfrak{g}^{[a,b],\mathfrak{s}}$ is a $\Lambda$-\emph{monoidal categorification} of a cluster algebra structure on its Grothendieck ring (see \cite[\S~7]{KKOP_mon_cat_quant_aff_II} for more details). 
\item At the end of \cite[\S~8]{KKOP_mon_cat_quant_aff_II}, Kashiwara--Kim--Oh--Park state the problem of finding an explicit expression of the matrices $B(\mathfrak{C})$. A first solution to this problem has been proposed by the author in \cite{Contu1} by using tools from \emph{additive categorification}. Recently, relying on the combinatorics of $i$-boxes, an alternative solution has been give by Kashiwara--Kim \cite{KK_2024_exchange_matrices_iboxes}.
\item In \cite[\S~4]{Contu1}, the author has described an explicit sequence of mutations from the quiver $\widehat{Q}^{[a,b],\mathfrak{s}}$ 
 to the quiver $Q(\mathfrak{C})$. We refer to it as the sequence of mutations associated to the \emph{canonical chain transformation} $\mathfrak{C}_{-}^{[a,b]}\mapsto \mathfrak{C}$.
\end{itemize} 
\end{rem}

For any chain of $i$-boxes $\mathfrak{C}$ of range $[a,b]$, let $\Lambda(\mathfrak{C})$ be the skew symmetric $K(\mathfrak{C})$-matrix defined by
\[ \Lambda(\mathfrak{C})_{ks} = \Lambda(M(\mathfrak{c}_k),M(\mathfrak{c}_s)).\]
The pair $(B(\mathfrak{C}),\Lambda(\mathfrak{C}))$ is compatible (\cite[Thm.~8.1]{KKOP_mon_cat_quant_aff_II}). Moreover, under the sequence of mutations associated to the \emph{canonical chain transformation} $\mathfrak{C}_{-}^{[a,b]}\mapsto \mathfrak{C}$, the pair $(B(\mathfrak{C}_{-}^{[a,b]}),\Lambda(\mathfrak{C}_{-}^{[a,b]}))$ is sent to the pair $(B(\mathfrak{C}),\Lambda(\mathfrak{C}))$ (cf. {\cite[Prop.~6.4]{KKOP_mon_cat_quant_aff}}).

For any $i$-box $\mathfrak{c}$, we write $L_t(\mathfrak{c})$ for the $(q,t)$-character of the simple module $M(\mathfrak{c})$. 

\begin{teo}
\label{teo_quantum_analogue_monoidal_cat}
Let $[a,b]$ be an interval, where $a\in \mathbb{Z}\cup \{-\infty\}$ and $b\in\mathbb{Z}$, and let $\mathfrak{C}$ be a chain of $i$-boxes with range $[a,b]$. There is a $\mathbb{Z}[t^{\pm 1/2}]$-algebra isomorphism $\mathcal{K}_t(\mathscr{C}_\mathfrak{g}^{[a,b],\mathfrak{s}})$ : 
\[\varphi_t(\mathfrak{C}): \mathcal{A}(Q(\mathfrak{C}),\Lambda(\mathfrak{C})) \xrightarrow{\sim} \mathcal{K}_t(\mathscr{C}_\mathfrak{g}^{[a,b],\mathfrak{s}})\label{symb_varphi_t(C)},
\]
such that the following diagram of algebra homomorphisms commutes:

\[
\begin{tikzcd}
   \mathcal{A}(Q(\mathfrak{C}),\Lambda(\mathfrak{C})) \arrow[d, "\mathrm{ev}_{t=1}"]\arrow[r, "\varphi_t(\mathfrak{C})", "\sim"'] & \mathcal{K}_t(\mathscr{C}_\mathfrak{g}^{[a,b],\mathfrak{s}})\arrow[d, "\mathrm{ev}_{t=1}"]  \\
\mathcal{A}(Q(\mathfrak{C})) \arrow[r, "\sim", "\chi_q\circ\varphi(\mathfrak{C})"'] &\chi_q(\mathcal{K}(\mathscr{C}_\mathfrak{g}^{[a,b],\mathfrak{s}})) .
\end{tikzcd}
\]

In particular, any initial quantum cluster variable $X_s, s\in K(\mathfrak{C})$, satisfies \[\varphi_t(\mathfrak{C})(X_s) = L_t(\mathfrak{c}_{s})\] and, for any quantum cluster monomial $X\in \mathcal{A}(Q(\mathfrak{C}),\Lambda(\mathfrak{C}))$, we have 
\[\varphi_t(\mathfrak{C})(X) = L_t(m),\]
where $m\in \mathcal{M}^+$ is the unique dominant monomial such that $\varphi(\mathfrak{C})(\mathrm{ev}_{t=1}(X))=[L(m)]$.
\end{teo}

\begin{proof}
Consider the function $\overline{\varepsilon}:\Delta_0 \rightarrow \mathbb{Z}$ defined by
\[ \varepsilon(i) = p_{k(i)},\ \text{ where } k(i)=\mathrm{max}\{k\in [a,b]\ |\ i_k = i\}.\]
In general it is not an height function on $(\Delta,\sigma)$, but it is one in the sense of \cite[\S~4] {HFOO_iso_quant_groth_ring_clust_alg}. In particular, we have that the category $\mathscr{C}_{\mathfrak{g}}^{[-\infty,b],\mathfrak{s}}$ coincides with the category $\mathscr{C}_{\leq \overline{\varepsilon}}$ defined in \cite[\S~4] {HFOO_iso_quant_groth_ring_clust_alg}. Therefore, when $a=-\infty$ and $\mathfrak{C}=\mathfrak{C}_-^{[-\infty,b]}$ the theorem is proved in \cite[Thm~5.16, Thm.~6.6]{HFOO_iso_quant_groth_ring_clust_alg}.
Assume now that $a$ is in $\mathbb{Z}$. By definition, the ice quiver $Q(\mathfrak{C}_-^{[a,b]})$ is a subquiver of $Q(\mathfrak{C}_-^{[-\infty,b]})$. In particular, there are no arrows between non-frozen vertices of $Q(\mathfrak{C}_-^{[a,b]})$ and the vertices of $Q(\mathfrak{C}_-^{[-\infty,b]})$ not belonging to $Q(\mathfrak{C}_-^{[a,b]})$. This ensures that the quantum cluster algebra $\mathcal{A}(Q(\mathfrak{C}_-^{[a,b]}),\Lambda(\mathfrak{C}_-^{[a,b]}))$ is a $\mathbb{Z}[t^{\pm 1/2}]$-subalgebra of $\mathcal{A}(Q(\mathfrak{C}_-^{[-\infty,b]}),\Lambda(\mathfrak{C}_-^{[-\infty,b]}))$ (this argument is used sistematically in \cite{Qin_conjectures}).
Denote by $\varphi_t(\mathfrak{C}_-^{[a,b]})$ the composition
\[  \varphi_t(\mathfrak{C}_-^{[a,b]}): \mathcal{A}(Q(\mathfrak{C}_-^{[a,b]}, \Lambda(\mathfrak{C}_-^{[a,b]})) \hookrightarrow{} \mathcal{A}(Q(\mathfrak{C}_-^{[-\infty,b]}, ),\Lambda(\mathfrak{C}_-^{[-\infty,b]})) \overset{\varphi_t(\mathfrak{C}_-^{[-\infty,b]})}{\xrightarrow{\sim}} \mathcal{K}_t(\mathscr{C}_\mathfrak{g}^{[-\infty,b],\mathfrak{s}}).\]
The image of $\varphi_t(\mathfrak{C}_-^{[a,b]})$ coincides with  $\mathcal{K}_t(\mathscr{C}_\mathfrak{g}^{[a,b],\mathfrak{s}})$.
To prove this, we first consider that $\mathcal{K}_t(\mathscr{C}_\mathfrak{g}^{[a,b],\mathfrak{s}})$ is contained in the image of $\varphi_t(\mathfrak{C}_-^{[a,b]})$. This follows from the fact that, for $k\in [a,b]$, we can write $L_t(Y_{i_k,p_k})$ as $L_t([k])$, where $[k]$ is an $i$-box contained in $[a,b]$ that can be obtained through box moves associated to the non-frozen vertices of $Q(\mathfrak{C}_-^{[a,b]})$.
For the reverse inclusion, it suffices to consider by recursion on the exchange relations that, if $X$ is a cluster variable of $\mathcal{A}(Q(\mathfrak{C}_-^{[a,b]}), \Lambda(\mathfrak{C}_-^{[a,b]}))$, then its image through $\varphi_t(\mathfrak{C}_-^{[a,b]})$ is of the form $L_t(m)$, with $m$ a monomial in the variables $(Y_{i_k,p_k})_{k\in [a,b]}$.
It remains to prove the theorem in the case that $\mathfrak{C}$ is an arbitrary chain of $i$-boxes on $[a,b]$. This follows from the fact that the pair $(B(\mathfrak{C}),\Lambda(\mathfrak{C}))$ is obtained from $(B(\mathfrak{C}_-^{[a,b]}), \Lambda(\mathfrak{C}_-^{[a,b]}))$ by iterated mutations.
\end{proof}

\begin{rem}
The statement of Theorem is still true for the more general categories $\mathscr{C}_\mathfrak{g}^{[a,b],\mathcal{D}_\mathcal{Q},\underline{w}_0}$, where the reduced expression $\underline{w}_0$ is not necessarily adapted to $\mathcal{Q}$. The proof is the same.
\end{rem}

\begin{cor}[{\cite[Cor. 6.11]{HFOO_iso_quant_groth_ring_clust_alg}}]
    Let $[a,b]$ be an interval, where $a\in \mathbb{Z}\cup \{-\infty\}$ and $b\in\mathbb{Z}$, and let $\mathfrak{C}$ be a chain of $i$-boxes with range $[a,b]$. Let $m\in \mathcal{M}^+$ be a dominant monomial such that $[L(m)]$ is the image of a cluster monomial of $\mathcal{A}(Q(\mathfrak{C}))$ under $\varphi(\mathfrak{C})$. Then 
    \[ \mathrm{ev}_{t=1}(L_t(m)) = \chi_q(L(m)).\]
\end{cor}

We conclude this section by stating the quantum monoidal categorification theorem for $\mathscr{C}_\mathfrak{g}^{\mathbb{Z}}$.

\begin{teo}[{\cite[Thm.~6.15]{HFOO_iso_quant_groth_ring_clust_alg}}]
\label{teo_quant_clust_Groth_ring}
Let $\mathfrak{C}$ be a chain of $i$-boxes with range $[-\infty,\infty]$. There is a $\mathbb{Z}[t^{\pm 1/2}]$-algebra isomorphism $\varphi_t(\mathfrak{C})$ from the quantum cluster algebra $\mathcal{A}(Q(\mathfrak{C}),\Lambda(\mathfrak{C}))$ to the quantum Grothendieck ring $\mathcal{K}_t(\mathscr{C}_\mathfrak{g}^{\mathbb{Z}})$ : 
\[\varphi_t(\mathfrak{C}): \mathcal{A}(Q(\mathfrak{C}),\Lambda(\mathfrak{C})) \xrightarrow{\sim} \mathcal{K}_t(\mathscr{C}_\mathfrak{g}^{[-\infty,\infty],\mathfrak{s}}),
\]
It satisfies $\varphi_t(X_s) = L_t(\mathfrak{c}_{s})$.
Moreover, for any quantum cluster monomial $X\in \mathcal{A}(Q(\mathfrak{C}),\Lambda(\mathfrak{C}))$, we have 
\[\varphi_t(\mathfrak{C})(X) = L_t(m),\]
where $m\in \mathcal{M}^+$ is the unique dominant monomial such that $\varphi(\mathfrak{C})(\mathrm{ev}_{t=1}(X))=[L(m)]$.
\end{teo}

\section{Derived and semi-derived Hall algebras}
\label{sec_Hall}
In this chapter, for an exact category $\mathcal{E}$, we denote by $\mathcal{K}(\mathcal{E})$ its Grothendieck group and we denote by $\overline{X}$ the image in $\mathcal{K}(\mathcal{E})$ of an isomorphism class $[X]$ of an object $X\in \mathcal{E}$.

Let $\Bbbk$\label{symb_Bbbk} be a finite field. We denote by $v$ the number of its elements and we set $v=\sqrt{|\Bbbk|}$\label{symb_v}.
 We recall the definition of a semi-derived Hall-algebra following \cite{Gorsky2018}. We start by defining the Hall algebra of an exact category in the sense of Hubery \cite{Hubery2006} (the original definition, for abelian categories, is due to Ringel \cite{Ringel_Hall_algebras_1990}).

\begin{de}
\label{de_Hall_algebra}
Let $\mathcal{E}$ be an exact category linear over $\Bbbk$. Assume that $\mathcal{E}$ is \emph{finitary}: for all $M,N\in \mathcal{E}$, we have
\[|\mathrm{Hom}_\mathcal{E}(M,N)|<\infty \text{ and } |\mathrm{Ext}^1_\mathcal{E}(M,N)|<\infty.\]
The $\textit{Hall\ algebra}$ of the category $\mathcal{E}$, denoted by $\mathcal{H}(\mathcal{E})$, is the $\mathbb{Q}$-vector space freely generated by isomorphism classes $[M]$ of objects of $\mathcal{E}$ with the multiplication given by

\[[A]\diamond [C] = \sum_{B\in \text{Iso}(\mathcal{E})} \frac{|\mathrm{Ext}^1_{\mathcal{E}}(A,C)_B|}{|\mathrm{Hom}_\mathcal{E}(C,A)|}|B|,\]
where $\text{Ext}^1_{\mathcal{E}}(A,C)_B$ is the subset of $\text{Ext}^1_{\mathcal{E}}(A,C)$ formed by the classes of extensions whose middle term is isomorphic to $B$. 
\end{de}

It is a result of Ringel \cite{Ringel_Hall_algebras_1990} that the Hall algebra $\mathcal{H}(\mathcal{E})$ is associative, with unit given by $[0]$, the isomorphism class of the zero object of $\mathcal{E}$.

We now define the semiderived Hall algebra following Gorsky \cite{Gorsky2018}.
Let $\mathcal{F}$ be a $\Bbbk$-linear Frobenius category. Assume that it is essentially small, idempotent complete and Hom-finite. Under these conditions, Gorsky proved that the category $\mathcal{F}$ is also finitary, and, therefore, the Hall algebra of $\mathcal{\mathcal{F}}$ is well defined. We denote by $\mathcal{P}(\mathcal{F})$ the full subcategory of the projective-injective objects in $\mathcal{F}$. 

\begin{de}
\label{de_semiderived_hall}
    The $\emph{semi-derived\ Hall\ algebra}$ of $\mathcal{F}$, denoted by $\mathcal{SDH(F)}$, is the localization of the Hall algebra $\mathcal{H}(\mathcal{F})$ at the classes of the objects in $\mathcal{P(F)}$.
\end{de}

Next, we define the derived Hall algebra. First introduced by  Toën for suitable dg categories \cite{Toen2006}, the definition has later been generalised by Xiao and Xu to suitable  triangulated categories~\cite{Xiao-Xu2008}. 

\begin{de}
\label{de_derived_hall}
    Let $\mathcal{T}$ be an essentially small $\Bbbk$-linear triangulated category. Denote by $\Sigma$ the shift functor and assume the following:
    \begin{enumerate}
        \item $\text{dim(Hom}_\mathcal{T}(A,B))<\infty\ \ \ \forall A,B\in \mathcal{T};$
        \item $\mathcal{T}$ is left locally homologically finite, that is, for any $A,B\in \mathcal{T}$, there exists a positive integer $n$, such that for any $i>n$, we have
        \[\text{Ext}^{-i}(A,B)=0\ \ \ \forall i>n.\]
    \end{enumerate}    
    The \emph{derived Hall algebra} of $\mathcal{T}$, denoted by $\mathcal{DH(T)}$, is the $\mathbb{Q}$-vector space freely generated by the isomorphism classes of objects of $\mathcal{T}$ with the multiplication given by

     \begin{equation}
    \label{product_derived_Hall}
    [A]\diamond [C] = \sum_{B\in \text{Iso}(\mathcal{T})} |\text{Ext}^1_\mathcal{T}(C,A)_B|\prod_{i\geq 0}|\text{Ext}_\mathcal{T}^{-i}(C,A)|^{(-1)^{(i-1)}}[B],
    \end{equation}
where $\text{Ext}^1_\mathcal{T}(C,A)_B$ is the subset of $\text{Ext}^1_\mathcal{T}(C,A)=\text{Hom}_\mathcal{T}(C,\Sigma A)$ consisting of those morphisms whose cone is isomorphic to $\Sigma B$.
\end{de}

The algebra $\mathcal{DH(T)}$ turns out to be associative and with unit given by the isomorphism class of the zero object of $\mathcal{T}$ (\cite{Toen2006}, \cite{Xiao-Xu2008}).

\begin{rem}\mbox{}

\begin{enumerate}

    \item[(a)] Let us point out that the algebra $\mathcal{DH(T)}$ of Definition \ref{de_derived_hall} is the Drinfeld dual (in the sense of \cite{Xiao-Xu2015}) of the derived Hall algebra of To\"{e}n and Xiao--Xu. Denote this latter algebra by $DH(\mathcal{T})$. Up to a twist of the product, the algebra $DH(\mathcal{T})$ is the version that also appears in  \cite{HL_quantum_Groth_rings_derived_Hall}, \cite{Sheng2010} and \cite{Xiao-Xu2008}, while its Drinfeld dual $\mathcal{DH(T)}$ is used in \cite{Gorsky2018} and \cite{Zhang2022}). The Drinfeld dual algebra $\mathcal{DH(T)}$ has been introduced by Xiao--Xu \cite{Xiao-Xu2015} themselves, who proved that it coincides with the finite
  field version of the motivic Hall algebra defined by Kontsevich and Soibelman \cite{kontsevich2008stability}. Moreover, they show \cite[Cor.~3.4]{Xiao-Xu2015} that the algebra $\mathcal{DH}(\mathcal{T})$ is isomorphic to the derived Hall algebra $DH(\mathcal{T})$.
The algebra  $DH(\mathcal{T})$ is the $\mathbb{Q}$-vector space freely generated by the isomorphism classes of objects of $\mathcal{T}$ with multiplication given by

    \begin{equation*}
    [A]'\diamond [C]' = \sum_{B\in \text{Iso}(\mathcal{T})} \frac{|\text{Ext}^1_\mathcal{T}(B,C[-1])_{A[1]}|\prod_{i> 0}|\text{Ext}_\mathcal{T}^{-i}(B,C)|^{(-1)^{i}}}{|\mathrm{Aut}(A)|\prod_{i>0}|\text{Ext}_\mathcal{T}^{-i}(A,A)|^{(-1)^{i}}}[B]'.
    \end{equation*}
The isomorphism between $\mathcal{DH(T)}$ and $DH(\mathcal{T})$ is given by the following rescaling of the generators

    \begin{equation}
    \label{eq_rescaling}
        \mathcal{DH(T)}\rightarrow DH(\mathcal{T}),\ 
     [A] \xrightarrow[]{} |\mathrm{Aut}(A)|\prod_{i>0}|\mathrm{Ext}_\mathcal{T}^{-i}(A,A)| [A]'. 
    \end{equation} 
    
    \item[(b)] The above multiplication formulas for $\mathcal{H}(\mathcal{E})$, $\mathcal{DH}(\mathcal{T})$ and $DH(\mathcal{T})$, possibly up to a renormali\-zation of the basis, yield products which are opposite to those appearing in \cite{HL_quantum_Groth_rings_derived_Hall}, \cite{Sheng2010}, \cite{Toen2006} and \cite{Xiao-Xu2008}. On the other hand, these products agree with the ones appearing in  \cite{Bridgeland2013} , \cite{Gorsky2018}, \cite{Xiao-Xu2015}  and \cite{Zhang2022}.

\end{enumerate}
\end{rem}

In \cite{Gorsky2018} Gorsky provided  a link between the semi-derived Hall algebra of a Frobenius category and the derived Hall algebra of the associated stable category (see also \cite[Corollary 5.6]{Lin_Peng2019} for an analogous result with an alternative twist).

\begin{teo}{\cite[Thm. 4.2]{Gorsky2018}}
\label{teo_gorsky}
Assume that the stable category $\underline{\mathcal{F}}$ is left locally homologically finite. 
Denote by $\mathcal{SDH(F)'}$ the algebra obtained from the the semi-derived Hall algebra of $\mathcal{F}$ by twisting the multiplication in the following way:
\begin{equation}
    \label{Gorsky_twist}
    [A]\diamond'[C]=\frac{|\mathrm{Hom}_\mathcal{F}(C,A)|}{|\mathrm{Hom}_{\underline{\mathcal{F}}}(C,A)|}\cdot \prod_{i>0} |\mathrm{Ext}_{\underline{\mathcal{F}}}^{-i}(C,A)|^{(-1)^{(i-1)}}
    [A]\diamond [C].
\end{equation}
Let $I(\mathcal{P(F)})$ be the two-sided ideal of $\mathcal{SDH(F)'}$ generated by $\{[P]-1| P \in \mathcal{P(F)}\}$. Then, each choice of representatives in $\mathcal{F}$ of the isomorphism
classes of the stable category $\underline{\mathcal{F}}$ gives an $\mathbb{Q}$-algebra isomorphism

\begin{equation}
    \mathcal{SDH(F)'}/I(\mathcal{P(F)})\cong \mathcal{DH(\underline{F})}.
\end{equation}
    
\end{teo}

\begin{rem}\mbox{}

\begin{itemize}
\label{rem_grading}
    \item  The Hall algebra $\mathcal{H}(\mathcal{E})$ of definition \ref{de_Hall_algebra} is graded by the Grothendieck group $\mathcal{K(E)}$ (the generators and relations are homogeneous for the natural grading).
    \item The semi-derived Hall algebra of \ref{de_semiderived_hall} inherits a $\mathcal{K(F)}$-graduation from $\mathcal{H(F)}$ since it is obtained as a localization with respect to homogeneous elements.
    \item The derived Hall algebra $\mathcal{DH}(\mathcal{T})$ of definition \ref{de_derived_hall} is graded by the Grothendieck group $\mathcal{K(T)}$ (the generators and relations are homogeneous for the natural grading).
    \end{itemize}
\end{rem}

\section{Semi-derived Hall algebras of categories of complexes}
\label{sec_semiderived_Hall_complexes}

In this section, we construct the semi-derived Hall algebras of certain categories of complexes of modules over path algebras of quivers.
We use the conventions of co-chain complexes (in particular, the differentials are maps $d^i:M^i\rightarrow M^{i+1}$).
We denote by $[1]$ the shift operation, and by $[i]$ the composition $[1]^i$, for any $i\in \mathbb{Z}$.

Let $Q$ be a finite connected acyclic quiver, $\Bbbk Q$ its path algebra, $\text{mod}(\Bbbk Q)$ be the category of finite dimensional right $kQ$-modules and $D^b(Q)$ the associated bounded derived category.
Let $\mathcal{P}$ be the full subcategory of $\mathrm{mod}(\Bbbk Q)$ whose objects are the projective modules. The category of bounded complexes of projectives $C^b(\mathcal{P})$\label{symb_Cb(P)} is a Frobenius category whose projective-injective objects are exactly the acyclic complexes (a bounded complex of projective objects is acyclic if and only if it is contractible). Recall that the associated stable category is the homotopy category of bounded complexes of projectives $K^b(\mathcal{P})$, which is canonically equivalent to the bounded derived category $D^b(Q)$.\\

Let us recall the description of the indecomposable objects of $C^b(\mathcal{P})$. 
Since $\Bbbk Q$ is the path algebra of a quiver, the category $\mathrm{mod}(\Bbbk Q)$ is hereditary. Moreover, since $\Bbbk Q$ is finite-dimensional, for any object $M$ of $\mathrm{mod}(\Bbbk Q)$, we can fix a minimal projective resolution of the form

\begin{equation}
\label{eq_proj_res}
0\rightarrow P_M\xrightarrow[]{f_M}Q_M\xrightarrow[]{g_M} M\rightarrow 0.
\end{equation}
We write $C_M$\label{symb_complex_C_M} for the complex
\[\dots\rightarrow 0\rightarrow P_M\xrightarrow[]{f_M} Q_M\rightarrow 0 \rightarrow \dots,\]
where $Q_M$ is in degree $0$. 
Moreover, for a projective module $P\in \mathcal{P}$, let $K_P$\label{symb_complex_K_P} be the acyclic complex 
\[\dots\rightarrow 0\rightarrow P\xrightarrow[]{\text{Id}_P} P\rightarrow 0 \rightarrow \dots,\]
where $P$ is in degrees $-1$ and $0$. 
Proceeding as in \cite[Lem.~4.2]{Bridgeland2013} we can prove the following:
\begin{lem}
The set
\[ \{K_P[i]|\ i\in\mathbb{Z}, P\in\mathcal{P} \text{ indecomposable}\}\cup \{C_M[i]|\ i\in\mathbb{Z}, M\in\mathrm{mod}(\Bbbk Q) \text{ indecomposable}\} \]
yields a system of representatives of the isomorphism classes of indecomposable objects in $C^b(\mathcal{P})$.
Moreover, any projective-injective object is a direct sum of objects $K_P[i]$.
\end{lem}

We define the \emph{twisted semi-derived Hall algebra} $\mathcal{SDH}_{tw}(C^b(\mathcal{P}))$\label{symb_SDH_twisted_CbP} as the $\mathbb{Q}(v^{1/2})$-algebra obtained from $\mathbb{Q}(v^{1/2})\otimes_{\mathbb{Q}}\mathcal{SDH}(C^b(\mathcal{P}))'$ by twisting the multiplication  by the factor (Sheng factor \cite{Sheng2010}) 
\[ (M,N) \mapsto
v^{\sum_{i\in \mathbb{Z}}(-1)^i\text{dim}_k\text{Hom}_{D^b(Q)}(M,N[i])}.
\]
Similarly, by applying the same twist to the multiplication of the derived Hall algebra $\mathcal{DH}(D^b(Q))$, we define the \emph{twisted derived Hall algebra} $\mathcal{DH}_{tw}(Q)$.

We write $\mathcal{SDH}^0_{tw}(C^b(\mathcal{P}))$ for the subalgebra of $\mathcal{SDH}_{tw}(C^b(\mathcal{P}))$ generated by the elements \[\{K_P[i]|\ i\in\mathbb{Z}, P\in\mathcal{P} \text{ indecomposable}\}.\]

\begin{rem}
Notice that, compared to the semi-derived Hall algebra $\mathcal{SDH}(C^b(\mathcal{P}))$, the product of the twisted semi-derived Hall algebra is twisted by the factor

\begin{equation}
\label{eq_twist}
v^{\sum_{i \in\mathbb{Z}} (-1)^i\text{dim}_k\text{Hom}_{D^b(Q)}(M,N[i])}\cdot v^{2\text{dim}_k\text{Hom}_{C^b(\mathcal{P})}(M,N)-2\sum_{i\leq 0}(-1)^i\text{dim}_k\text{Hom}_{D^b(Q)}(M,N[i])}
\end{equation}
Alternatively, by multiplying (\ref{eq_twist}) with the additional factor
\[v^{-2\text{dim}_k\text{Hom}_{C^b(\mathcal{P})}(M,N)}\cdot v^{-2\sum_{i>0}(-1)^i\text{dim}_k\text{Hom}_{D^b(Q)}(M,N[i])},\]
we find that the twisted semi-derived Hall algebra is obtained by twisting the multiplication used in \cite{Zhang2022} by the factor

\begin{equation}
\label{eq_twist_2}   
v^{-\sum_{i\in \mathbb{Z}}(-1)^i\text{dim}_k\text{Hom}_{D^b(Q)}(M,N[i])}.
\end{equation}
\end{rem}

Let $\alpha$ be an element of the Grothendieck group $\mathcal{K}(\mathrm{mod}(\Bbbk Q))$.
Using (bounded) projective resolutions, we can express $\alpha$ as
\[\alpha=\overline{P}-\overline{Q},\]
for some $P,Q\in \mathcal{P}$. 
For $i\in\mathbb{Z}$, we define the element
\[\label{symb_K_alpha_i}
K_{\alpha,i}=[K_P[i]][K_Q[i]]^{-1}\in \mathcal{SDH}_{tw}(C^b(\mathcal{P})).\]
Moreover, for $M\in\mathrm{mod}(\Bbbk Q)$ and $i\in\mathbb{Z}$, we define the elements 
\[\label{symb_E_M_i} E_{M,i}=[K_{-P_M}[i]][C_M[i]]/a_M \in \mathcal{SDH}_{tw}(C^b(\mathcal{P}))\]
and
\[Z_{M,i}=[M[i]]/a_M \in \mathcal{DH}_{tw}(Q),\]
where 
\[ a_M = \mathrm{Aut}_{\Bbbk Q}(M)\prod_{i>0}|\mathrm{Ext}_{\Bbbk Q}^{-i}(M,M)|.\]
\begin{rem}
    The coefficient $a_M$ does not appear in the definition of the element corresponding to $E_{M,i}$ given in \cite{Bridgeland2013}, \cite{Gorsky2013semiderived} and \cite{Zhang2022}. It plays the role of the rescaling factor in~(\ref{eq_rescaling}).
\end{rem}

Recall Ringel's bilinear form on the Grothendieck group $\mathcal{K}(\mathrm{mod}(\Bbbk Q))$ given by

\[\langle M,N  \rangle = \mathrm{dim(Hom}_{\Bbbk Q}(M,N)) - \mathrm{dim(Ext}_{\Bbbk Q}^1(M,N)),\ \ \ M,N\in \mathrm{mod}\Bbbk Q.\]
We denote by $(-,-)$ its symmetrization:

\[(M,N)= \langle M,N  \rangle+ \langle N,M  \rangle.\]
For $M, N, V, W\in\mathrm{mod}(\Bbbk Q)$, let $\mathrm{Ex}_{N,M}^{V,W}\subset \mathrm{Hom}_{\Bbbk Q}(V,N)\times\mathrm{Hom}_{\Bbbk Q}(N,M)\times\mathrm{Hom}_{\Bbbk Q}(M,W)$ be the set

\[ \mathrm{Ex}_{N,M}^{V,W} = \{(f,g,h)\ |\ 0\rightarrow V\xrightarrow[]{f} N \xrightarrow[]{g} M \xrightarrow[]{h} W \rightarrow 0 \text{ is exact }\}. \]
and define

\[ \gamma_{N,M}^{V,W}= \frac{|\mathrm{Ex}_{N,M}^{V,W}|}{\mathrm{Aut}_{\Bbbk Q}(M)\mathrm{Aut}_{\Bbbk Q}(N)}. \]
For $X,Y,Z\in \mathrm{mod}(\Bbbk Q)$, we set
\[ g_{X,Y}^Z = |\{ X'\subset Z\ |\ X' \cong X, \ Z/X'\cong Y\}|.\]
The \emph{Hall numbers} of $\mathrm{mod}(\Bbbk Q)$ are the numbers $g_{X,Y}^Z$ and $\gamma_{N,M}^{V,W}$.

\begin{teo}[\cite{Zhang2022}]
\label{teo_relations_semiderived}
The twisted semi-derived Hall algebra $\mathcal{SDH}_{tw}(C^b(\mathcal{P}))$ is generated by the elements   $E_{M,i} \text{ and }\ K_{\alpha,i}$ 
, where $i\in \mathbb{Z}, M \in \mathrm{mod}(\Bbbk Q), \alpha\in \mathcal{K}(\mathrm{mod}(\Bbbk Q))$, 
subject to the following relations

\begin{equation*} 
\begin{split}
        K_{\alpha,i} &\text{ is central, }\ \ \    \alpha\in \mathcal{K}(\mathrm{mod}(\Bbbk Q)),\ i\in \mathbb{Z};\\
        E_{M,i}E_{N,i} &= v^{\langle M,N\rangle } \sum_{L} g_{M,N}^L E_{L,i},\ \ \ \ M,N\in \mathrm{mod}(\Bbbk Q), i\in\mathbb{Z};\\
        E_{M,i} E_{N,i+1} &= v^{-\langle N,M\rangle } \sum_{V,W} u^{-\langle V,W\rangle } \gamma_{N,M}^{V,W} E_{W,i+1} E_{V,i} K_{[N]-[V],i}, \ \ \ \ M,N\in \mathrm{mod}(\Bbbk Q), i\in\mathbb{Z};\\
        E_{M,i}E_{N,j} &= v^{(-1)^{i-j}(M,N)}E_{N,j}E_{M,i},\ \ \ \ M,N\in \mathrm{mod}(\Bbbk Q),\ j>i+1\in\mathbb{Z}.
\end{split}
\end{equation*}

\end{teo}

We have constructed the twisted semi-derived Hall algebra in such a way that, specializing the $K_{\alpha,i}$ to 1 and replacing the generators $E_{M,i}$ with the elements $Z_{M,i}$ of the derived Hall algebra, we obtain the generators and the relations of the twisted derived Hall algebra used in \cite{HL_quantum_Groth_rings_derived_Hall} and \cite{Sheng2010}.

\begin{teo}[{\cite{Toen2006},\cite{Sheng2010}}]
The twisted derived Hall algebra $\mathcal{DH}_{tw}(Q)$ is generated by the elements  $Z_{M,i}$, where $M\in \mathrm{mod}(\Bbbk Q),\ i\in \mathbb{Z}$,
subject to the following relations

\begin{equation*} 
\begin{split}
        Z_{M,i}Z_{N,i} &= v^{\langle M,N\rangle } \sum_{L} g_{M,N}^L Z_{L,i},\ \ \ \ M,N\in \mathrm{mod}(\Bbbk Q), i\in\mathbb{Z};\\
        Z_{M,i} Z_{N,i+1} &= v^{-\langle N,M\rangle } \sum_{V,W} u^{-\langle V,W\rangle } \gamma_{N,M}^{V,W} Z_{W,i+1} Z_{V,i}, \ \ \ \ M,N\in \mathrm{mod}(\Bbbk Q), i\in\mathbb{Z};\\
        Z_{M,i}Z_{N,j} &= v^{(-1)^{i-j}(M,N)}Z_{N,j}Z_{M,i},\ \ \ \ M,N\in \mathrm{mod}(\Bbbk Q),\ j>i+1\in\mathbb{Z}.
\end{split}
\end{equation*}

\end{teo}
It follows from Theorem~\ref{teo_gorsky} that there is a surjective map \begin{align*}
\pi^{\mathcal{H}}: \mathcal{SDH}_{tw}(C^b(\mathcal{P}))&\twoheadrightarrow \mathcal{DH}_{tw}(Q), \\
E_{M,i} &\mapsto Z_{M,i},\\
K_{\alpha,i} &\mapsto 1,
\end{align*}
whose kernel is the two-sided ideal generated by $\{K-1\ |\ K \in \mathcal{SDH}^0_{tw}(C^b(\mathcal{P}))\}$.

In the sequel, the basis given in the following proposition will be of great use.

\begin{prop}[{\cite[Prop.~5.4]{Zhang2022}}]
\label{prop_basis_semiderived}
    The elements 
    \[      E_{M_s,s}E_{M_{s-1},s-1}\dots E_{M_r,r}K_{\alpha_s,s} K_{\alpha_{s-1},s-1}\dots K_{\alpha_r,r},
    \]
    for $r\leq s\in\mathbb{Z},\ \alpha_i \in \mathcal{K}(\mathrm{mod}(\Bbbk Q)) \text{ and } M_i\in \mathrm{mod}(\Bbbk Q)$ for $r\leq i\leq s$,
    form a $\mathbb{Q}(v^{1/2})$-basis of the twisted semi-derived Hall algebra $\mathcal{SDH}_{tw}(C^b(\mathcal{P}))$.   
\end{prop}

\section{On the grading of the (semi-)derived Hall algebra}
Let $V$ be an object of $D^b(Q)$ and consider a decomposition $V=\bigoplus_{k} M_k[i_k]$, with $M_k \in \mathrm{mod}(\Bbbk Q)$ and $i_k\in \mathbb{Z}$ for any $k$. We write $C_V$ for the complex

\[C_V = \bigoplus_{k} C_{M_k}[i_k] \in C^b(\mathcal{P}).\]
Moreover, we define $Z_V\in \mathcal{DH}_{tw}(Q)$ and \label{symb_E_V}\label{symb_K_V} $E_V, K_V\in \mathcal{SDH}_{tw}(C^b(\mathcal{P}))$ by 
\begin{align*}
Z_V &=  \prod^\curvearrowright_k Z_{M_k,i_k}, \\
E_V &=  \prod^\curvearrowright_k E_{M_k,i_k},\\
K_V &=  \prod_k K_{\overline{M_k},i_k},
\end{align*} 
where in the first two products the factors are ordered so that 
\[\mathrm{Ext}^1_{\mathcal{D}^b(Q)}(M_{k'}[i_{k'}],M_k[i_k])=0\]
if $Z_{M_k,i_k}$ precedes $Z_{M_{k'},i_{k'}}$.

Notice that there is a natural surjective map
\[ \mathcal{K}(C^b(\mathcal{P})) \twoheadrightarrow \mathcal{K}(K^b(\mathcal{P})) \cong \mathcal{K}(D^b(Q)).\]
Its kernel is the subgroup of $\mathcal{K}(C^b(\mathcal{P}))$ generated by the classes of contractible complexes.
Recall from Remark \ref{rem_grading} that $\mathcal{SDH}_{tw}(C^b(\mathcal{P}))$ is graded by the Grothendieck group $\mathcal{K}(C^b(\mathcal{P}))$. Up to some minor adjustements due to the different settings, the next lemma can also be seen as a consequence of \cite[Lem.~4.5]{Gorsky2013semiderived}. We provide an alternative proof, more coherent with our treatment.

\begin{lem}
\label{lem_uniqueness_K_degree}
Let $V,W\in D^b(Q)$ such that $K_V$ and $K_W$ have the same degree in $\mathcal{K}(C^b(\mathcal{P}))$. Then $K_V=K_W$.
\end{lem}

\begin{proof}
By definition, we can write $K_V=[K_{P^0_V}][K_{P^1_V}]^{-1}$ and $K_W=[K_{P^0_W}][K_{P^1_W}]^{-1}$, such that $P^i_V$ and $P^i_W\ (i=0,1)$ are direct sums of the form

\[ P^i_V=\bigoplus_{1\leq k\leq N} P^i_{V,k}[i_k],\ \ P^i_W=\bigoplus_{1\leq k\leq N} P^i_{W,k}[i_k],  \]
where $N\in \mathbb{N}$, $i_k< i_{k+1}\in \mathbb{Z}$ for any $1\leq k\leq N-1$ and each direct summand $P^i_{V,k}, P^i_{W,k}$ is a module in $\mathcal{P}$ or the zero object.

By hypothesis, 

\[\mathrm{deg}(K_V)=\sum_{1\leq N}(\overline{K_{P^{0}_{V,k}}}[i_k]-\overline{K_{P^{1}_{V,k}}}[i_k])=\sum_{1\leq N}(\overline{K_{P^{0}_{W,k}}}[i_k]-\overline{K_{P^{1}_{W,k}}}[i_k])=\mathrm{deg}(K_W).\]
Consider the projection on the $i_1$-summand

\[ \mathcal{K}(C^b(\mathcal{P})) \cong \coprod_{\mathbb{Z}} \mathcal{K}(\mathcal{P}) \xrightarrow{\pi_{i_{1}}} \mathcal{K}(\mathcal{P}). \]
It maps 
\begin{align*}
    \mathrm{deg}(K_V) \mapsto \overline{P^0_{V,1}} - \overline{P^1_{V,1}} ;
    \mathrm{deg}(K_W) \mapsto \overline{P^0_{W,1}} - \overline{P^1_{W,1}} ;
\end{align*}
Therefore, we have 
\[\overline{P^0_{V,1}} - \overline{P^1_{V,1}} = \overline{P^0_{W,1}} - \overline{P^1_{W,1}}\ \text{ in } \mathcal{K}(\mathcal{P}),\]

from which we deduce 
\[ [K_{P^0_{V,1}[i_1]}][K_{P^1_{V,1}[i_1]}]^{-1}= [K_{P^0_{W,1}[i_1]}][K_{P^1_{W,1}[i_1]}]^{-1}.\]

The results follows by recursion on the indices $i_k$: 
if we consider the projection $\pi_{i_k}$ on the $k$-th summand, with $k>1$, we get 
\[\overline{P^0_{V,k-1}} - \overline{P^1_{V,k-1}}+
\overline{P^0_{V,k}} - \overline{P^1_{V,k}} = \overline{P^0_{W,k-1}} - \overline{P^1_{W,k-1}}+\overline{P^0_{W,k}} - \overline{P^1_{W,k}}\ \text{ if }i_k=i_{k-1}+1 ;\]
\[\overline{P^0_{V,k}} - \overline{P^1_{V,k}} = \overline{P^0_{W,k}} - \overline{P^1_{W,k}}\ \text{ otherwise }.\]
By recursion, we can deduce $\overline{P^0_{V,k}} - \overline{P^1_{V,k}} = \overline{P^0_{W,k}} - \overline{P^1_{W,k}}$ even in the case $i_k=i_{k-1}+1$.
Therefore, for any $1\leq k\leq N$, we have
\[ [K_{P^0_{V,k}[i_k]}][K_{P^1_{V,k}[i_k]}]^{-1}= [K_{P^0_{W,k}[i_k]}][K_{P^1_{W,k}[i_k]}]^{-1},\]
from which it follows that 
\[K_V=\prod_{1\leq k\leq N }[K_{P^0_{V,k}[i_k]}][K_{P^1_{V,k}[i_k]}]^{-1}=\prod_{1\leq k\leq N }[K_{P^0_{W,k}[i_k]}][K_{P^1_{W,k}[i_k]}]^{-1}=K_W\]
\end{proof}

\begin{cor}
Let $V_1,V_2,W_1,W_2 \in D^b(Q)$ such that $\mathrm{deg}(K_{V_1}K_{V_2}^{-1})=\mathrm{deg}(K_{W_1}K_{W_2}^{-1})$. Then 
\[K_{V_1}K_{V_2}^{-1}=K_{W_1}K_{W_2}^{-1}.\]
\end{cor}

Recall that the \emph{degeneration order} on the objects of the bounded derived category $D^b(Q)$, introduced in \cite{Jense_Su_Zimmermann_2005}, is the partial order defined as follows:
\[ \label{symb_leq_deg_order}V_1 \leq V_2 \text{ if there exists } Z\in D^b(Q) \text{ and a triangle } V_2\rightarrow V_1\oplus Z \rightarrow Z \rightarrow  V_2[1].\]
Notice that we are using the opposite order with respect to the one defined in \cite{Jense_Su_Zimmermann_2005}.

\begin{lem}
\label{lem_find_a_K}
 Let $M$ and $N$ be objects in $D^b(Q)$ such that  $M<N$. Then there exists a unique element $K_{M,N}\in \mathcal{SDH}^0_{tw}(C^b(\mathcal{P}))$ such that the following two conditions are satisfied:
 \begin{enumerate}
 \setlength\itemsep{1.2em}
 \item[$\mathrm{(i)}$] it is of the form
 \[  K_{M,N}=\prod_{1\leq i\leq N} K_{\overline{P}_i-\overline{Q}_i,z_i},   \]
 for some $P_i,Q_i \in \mathcal{P}$ and $z_i\in\mathbb{Z}$.
 \item[$\mathrm{(ii)}$] $\mathrm{deg}(E_M)=\mathrm{deg}(E_{N}K_{M,N}) \ \text{ in }\ \mathcal{K}(C^b(\mathcal{P})).$
 \end{enumerate} 
\end{lem}

\begin{proof}
Since $M<N$, we have $\overline{M}=\overline{N}$ in $\mathcal{K}(D^b(Q))$. Since the Grothendieck group $\mathcal{K}(D^b(Q))$ is a quotient of the Grothendieck group $\mathcal{K}(C^b(\mathcal{P}))$ by the ideal generated by the contractible complexes, we have \[\overline{C_M}=\overline{C_N}+\sum_{1\leq k\leq N_1} \overline{K_{P_k}[i_{k}}]-\sum_{1\leq k'\leq N_2} \overline{K_{P_k'}[i_{k'}'}]\ \text{  in }
\mathcal{K}(C^b(\mathcal{P})),\]
for some $N_1,N_2\in\mathbb{N}, i_{k}, i_{k'}'\in \mathbb{Z}, P_{k},P_{k'} \in\mathcal{P}$. Without loss of generality, we can assume $N_1=N_2$ and $i_k=i_k'$ for any $1\leq k\leq N_1$. 

It follows that 

\begin{align*}
 \mathrm{deg}(E_M)&=\mathrm{deg}([C_N]K_{P_M}^{-1}\prod_{1\leq k\leq N_1} K_{\overline{P_k}-\overline{P_k'},i_k})\\[1em]
 &=\mathrm{deg}(E_NK_{P_M}^{-1}K_{P_N}\prod_{1\leq k\leq N_1} K_{\overline{P_k}-\overline{P_k'},i_k}).  
\end{align*}

Therefore the elements
\[K_{M,N}= K_{P_M}^{-1}K_{P_N}\prod_{1\leq k\leq N_1} K_{\overline{P_k}-\overline{P_k'},i_k}\]
satisfy the conditions of the statement of the lemma. Uniqueness follows from Lemma \ref{lem_uniqueness_K_degree}.
\end{proof}

\section{The Hernandez--Leclerc isomorphism}
\label{sec_Hernadez_Leclerc_iso}
In the following, we freely use notations and terminology from section 2.
In particular, recall the finite-dimensional simple complex Lie algebra  $\mathfrak{g}$ of ADE type, the parity function $\Tilde{\varepsilon}$ and the height function $\varepsilon$ on the index set $I_\mathfrak{g}$ of the Dynkin diagram of $\mathfrak{g}$, the reduced expression $\underline{w}_0$ of the longest element of $W_\mathfrak{g}$ adapted to $\varepsilon$ and the admissible sequence $\mathfrak{s}$ associated to the pair $(\varepsilon,\underline{w}_0)$.
Recall the set
\[\widehat{I}_{\mathfrak{g}}=\{(i,p)\in I_\mathfrak{g}\times\mathbb{Z}\ |\ p-\Tilde{\varepsilon}_i\in 2\mathbb{Z}\}\]
is the set of vertices of the repetition quiver $\widehat{Q}$.

\begin{de}
The \emph{mesh category} $\Bbbk(\widehat{Q})$ associated to the repetition quiver $\widehat{Q}$ is the $\Bbbk$-category whose
\begin{itemize}
\item objects are identified with the vertices of $\widehat{Q}$;
\item for any two vertices $A,B\in\widehat{Q}$, the space of morphisms is the space of linear combinations of paths from $A$ to $B$, modulo the so called \emph{mesh relations}: for each $(i,p)$ in $\widehat{I}_{\mathfrak{g}}$, the sum of all paths from $(i,p)$ to $(i,p+2)$ vanishes;
\item the composition of morphisms is induced by the usual composition of paths.
\end{itemize}
\end{de}

Let $Q$ be an orientation of the Dynkin diagram of $\mathfrak{g}$ and let $\widehat{\varepsilon}$ be an height function on the Dynkin diagram of $\mathfrak{g}$ such that $Q=Q_{\widehat{\varepsilon}}$ and $\widehat{\varepsilon}_i=\tilde{\varepsilon}_i\ \mathrm{mod}\ 2$, for any $i$ in $I_\mathfrak{g}$.

We recall a fundamental result due to Happel. \label{symb_V(i,p)}
\begin{teo}[{\cite[Prop.~4.6]{Happel_articolo_1987}}]
There is an equivalence of $\Bbbk$-categories
\[V:\Bbbk(\widehat{Q})\xrightarrow{\sim} \mathrm{ind}(D^b(Q))\]
such that, for any $(i,p)$ in $\widehat{I}_{\mathfrak{g}}$, we have
\[V(i,p)=\tau^{(\widehat{\varepsilon}_i-p)/2}(I_i).\]
\end{teo}

By considering $\mathbb{Q}(v^{1/2})$ as a $\mathbb{Z}[t^{\pm 1/2}]$-module through the specialization $t^{1/2}\mapsto v^{1/2}$, we define the $\mathbb{Q}(v^{1/2})$-algebra
$$\mathcal{K}_v(\mathscr{C}_\mathfrak{g}^{\mathbb{Z}})=\mathbb{Q}(v^{1/2})\otimes_{\mathbb{Z}[t^{\pm 1/2}]}\mathcal{K}_t(\mathscr{C}_\mathfrak{g}^{\mathbb{Z}}).$$
By abuse of terminology, we still refer to the image of $M_t(m)$ in $\mathcal{K}_v(\mathscr{C}_\mathfrak{g}^{\mathbb{Z}})$ as the $(q,t)$-character of the standard module $M(m)$,  and similarly for $L_t(m)$.

\begin{teo}[{\cite[Thm.~8.2]{HL_quantum_Groth_rings_derived_Hall}}]
\label{teo_isoHL}
There is a $\mathbb{Q}(v^{1/2})$-algebra isomorphism \[\Phi:\mathcal{K}_v(\mathscr{C}_\mathfrak{g}^{\mathbb{Z}})\xrightarrow{\sim} \mathcal{DH}_{tw}(Q)\] satisfying the following:
\begin{enumerate}
\item[$(i)$] for $(i,p)\in\widehat{I}_{\mathfrak{g}}$, the $(q,t)$-character of the simple module $L(Y_{i,p})$ is mapped by $\Phi$ to a scalar multiple of $Z_{V(i,p)}$;
\item[$(ii)$] the basis of $(q,t)$-characters of standard modules is mapped by $\Phi$ to a rescaling of the natural basis of $\mathcal{DH}_{tw}(Q)$ labelled by the isoclasses of all objects of $D^b( Q)$.
\end{enumerate}
\end{teo}

To each dominant monomial $m=\bigoplus_{(i,p)\in \widehat{I}_{\mathfrak{g}}} y_{i,p}^{u_{i,p}}$, we associate the object $V(m)$ of $\mathcal{D}^b(Q)$ defined by 
\[ V(m) = \bigoplus_{(i,p)\in \widehat{I}_{\mathfrak{g}}} V(i,p)^{u_{i,p}}.\]
It follows from Theorem \ref{teo_isoHL} and the definition of the $(q,t)$-characters of the simple modules that the images under the isomorphism $\Phi$ of the $(q,t)$-characters of the standard and simple modules associated to the dominant monomial $m$ are respectively of the form

\begin{equation}
\label{formula_image_qt_standard}
\Phi (M_t(m))= a(m) Z_{V(m)}\ \text{ and }
\end{equation}

\begin{equation}
\label{formula_image_qt_simple}
\Phi (L_t(m))= a(m) Z_{V(m)} +\sum_{m'<m} a(m,m') Z_{V(m')},
\end{equation}
where $a(m)$ and $a(m,m')$ are certain coeffcients in $\mathbb{Q}(v^{1/2})$.

\begin{rem}[Comparing orders]\label{order_comparison}
Through the assignment $m\mapsto V(m)$, Nakajima's order on dominant monomials is linked to the degeneration order in the derived category: Let $m,m'\in \mathcal{M}^+$ be dominant monomials. We claim that
\[ \text{if } m' \leq m \text{ and } a(m,m')\neq 0\text{ then } V(m')\leq V(m).\]
To see why, we start by defining the set
\[ \widehat{I'} = \{ (i,p)\in I\times \mathbb{Z} \ |\ (i,p-1) \in \widehat{I}_{\mathfrak{g}}\}. \]
If $m$ and $m'$ are two dominant monomials such that $m'\leq m$, we can write them as
\[ m = \prod_{(i,p) \in \widehat{I}_{\mathfrak{g}}} Y_{(i,p)}^{u_{i,p}(m)}, \ \ \ m' = \prod_{(i,p) \in \widehat{I}_{\mathfrak{g}}} Y_{(i,p)}^{u_{i,p}(m)}\prod_{(i,p)\in \widehat{I'}} A_{i,p}^{-u_{i,p}(m,m')}. \]
Define the graded $\mathbb{C}$-vector spaces
\[ W= \oplus_{(i,p)\in \widehat{I}_{\mathfrak{g}}}\  \mathbb{C}^{u_{i,p}(m)} \text{ and } V'=\oplus_{(i,p)\in \widehat{I'}} \mathbb{C}^{u_{i,p}(m,m')}.\]
It follows that (see \cite[Rem.~3.16]{Leclerc_Plamondon_2013} or \cite[Thm.~4.10]{Fujita2022} based on \cite[\S~3.3]{Nakajima2001}), if we consider the associated strata of  Nakajima's (regular) graded quiver varieties, we have the inclusion
\[ \mathfrak{M}_0^{\bullet\mathrm{reg}}(0,W)\subset \overline{\mathfrak{M}_0^{\bullet\mathrm{reg}}(V',W).}\]
Therefore, we can conclude from \cite[Thm.~2.8]{Keller_Scherotzke_2016} that 
\[ V(m') \leq V(m). \]
\end{rem}

\section{A lift of the exchange relations to the semi-derived Hall algebra}
\label{sec_lift}

We define the \emph{semiderived $(q,t)$-characters} of the standard and simple modules associated to the dominant monomial $m$ as

\begin{equation}
    \mathcal{M}_v(m)= a(m) E_{V(m)}\ \text{ and }
\end{equation} 

\begin{equation}
\label{eq_semiderived_Lv(m)}
    \mathcal{L}_v(m)= a(m) E_{V(m)} +\sum_{m'<m,\ a(m,m')\neq 0} a(m,m') E_{V(m')}K(m,m'),
\end{equation}

where $a(m)$ and $a(m,m')$ are the coefficients defined in (\ref{formula_image_qt_standard}) and (\ref{formula_image_qt_simple}) and $K(m,m')$ is an abbreviation for the element $K_{V(m),V(m')} \in \mathcal{SDH}^0_{tw}(C^b(\mathcal{P}))$ introduced in \ref{lem_find_a_K}.

\begin{ex}
    Let $Q$ be the quiver of type $A_1$, with height function $\widehat{\varepsilon}_1=1$.
    A graphical representation of the bijection induced by Happel's theorem is given by

    \[
 \begin{tikzcd}
  \dots &  (1,-1) & (1,1)\arrow[blue, dotted]{d} & (1,3) &\dots\\  
  \dots & S[-1]  & S &  S[1]  & \dots
 \end{tikzcd}
 \]

 By Theorem (\ref{teo_isoHL}), we have
 \[\Phi (L_{t}(y_{1,1}y_{1,3})) = a(y_{1,1}y_{1,3})Z_{S[1]\oplus S}+a(y_{1,1}y_{1,3},1),\]
 
and, by our definition, we have

 \[\mathcal{L}_v(y_{1,1}y_{1,3}) = a(y_{1,1}y_{1,3})E_{S[1]\oplus S}+a(y_{1,1}y_{1,3},1)K_S.\]
 \end{ex}

\begin{ex}
    Let $Q$ be the quiver $2\rightarrow 1$ of type $A_2$, with height function $\widehat{\varepsilon}_1=0, \widehat{\varepsilon}_2=1.$
    A graphical representation of the bijection induced by Happel's theorem is given by

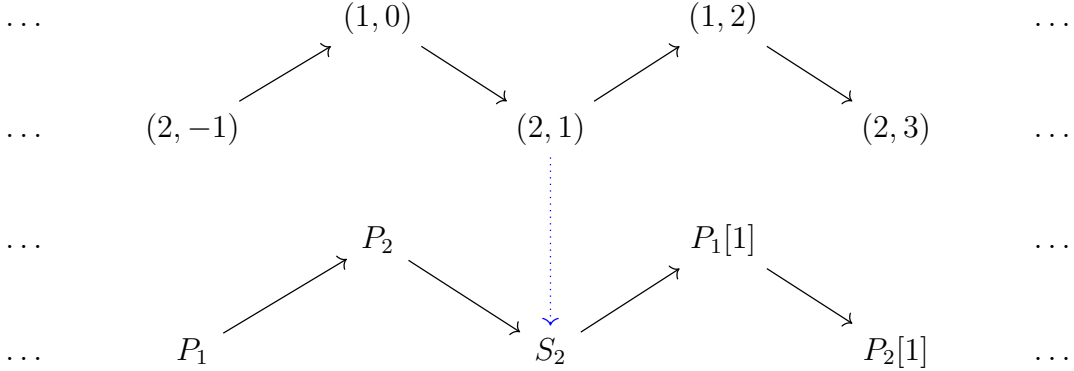
\[
 \begin{tikzcd}
  \dots &   & (1,0) \arrow{rd} & & (1,2) \arrow{rd}&  &\dots\\
   \dots & (2,-1) \arrow{ru}& & (2,1)\arrow[blue, dotted]{dd} \arrow{ru}& & (2,3) & \dots \\  
    \dots & & P_2 \arrow{rd} & & P_1[1] \arrow{rd}& & \dots  \\
  \dots & P_1 \arrow{ru}& & S_2 \arrow{ru}& & P_2[1]  
 & \dots 
 \end{tikzcd}
 \]

 By Theorem (\ref{teo_isoHL}), we have
 \[\Phi (L_{t}(y_{1,0}y_{1,2})) = a(y_{1,0}y_{1,2})Z_{P_1[1]\oplus P_2}+a(y_{1,0}y_{1,2},y_{2,1}) Z_{S_2},\]
 
and, by our definition, we have

 \[ \mathcal{L}_v(y_{1,0}y_{1,2}) = a(y_{1,0}y_{1,2})E_{P_1[1]\oplus P_2}+a(y_{1,0}y_{1,2},y_{2,1}) E_{S_2}K_{\overline{S}_1}.\]

\end{ex}

Notice that since $E_{S_2}=[P_1\rightarrow P_2] K_{P_1}^{-1}$, both terms on the right are of degree $\overline{P}_2+\overline{P_1[1]}$ in $\mathcal{K}(C^b(\mathcal{P}))$. We can also deduce this from the defining relations of the semi-derived Hall algebra.

\begin{rem}
\label{rem_form_exchange_relation_monomials}
    Consider the quantum cluster algebra structure of $\mathcal{K}_t(\mathscr{C}_{\mathfrak{g}}^{\mathbb{Z}})$, cf. Theorem~\ref{teo_quant_clust_Groth_ring}. Let us fix a quantum seed of $\mathcal{K}_t(\mathscr{C}_{\mathfrak{g}}^{\mathbb{Z}})$. Let $k$ be a vertex of the corresponding quiver. It follows from \cite[Thm.~6.15]{HFOO_iso_quant_groth_ring_clust_alg} that, in the notation of section \ref{sec_def_Groth_rings}, the exchange relation associated to a mutation at $k$ is of the form
\begin{equation}
\label{eq_exchange_relation_qt_characters}
L_t(m_k)L_t(m_k') = a(t)\overrightarrow{\prod_{i\rightarrow k}} L_t(m_i) +   b(t)\overrightarrow{\prod_{j\leftarrow k}} L_t(m_j), \end{equation} 

where the coefficients $a(t),b(t)$ belong to $\mathbb{Z}[t^{\pm 1/2}]$, $( L_t(m_k),L_t(m_k'))$ is the exchange pair and, for each arrow $i\rightarrow k$ (resp. $j \leftarrow k$), $L_t(m_i)$ (resp. $L_t(m_j)$) is the cluster variable associated to the vertex $i$ (resp. j).
It follows that one of the following holds:
\[ \begin{cases}
    m_km_k'= \prod_{i\rightarrow k} m_i; \\
    m_km_k'\geq \prod_{j\leftarrow k} m_j;
    \end{cases}
    \text{  or  }\ \begin{cases} m_km_k'= \prod_{j\leftarrow k} m_j;\\
    m_km_k'\geq \prod_{i\rightarrow k} m_i.
    \end{cases}\]
\end{rem}

In the following proposition, we keep the notation of Remark \ref{rem_form_exchange_relation_monomials}.

Denote $a(v)$ and $b(v)$ the specialization $t^{1/2}\mapsto v^{1/2}$ of the corresponding coefficients $a(t)$ and $b(t)$ of (\ref{eq_exchange_relation_qt_characters}).
\begin{prop} 
\label{prop_semiderived_exchange_rel}
    If $m_km_k'= \prod_{i\rightarrow k} m_i$, then, in the twisted semi-derived Hall algebra $\mathcal{SDH}_{tw}(C^b(\mathcal{P}))$, we have 

\begin{equation}
\label{eq_semiderived_exchange_relation}
    \mathcal{L}_v(m_k)\mathcal{L}_v(m_k') = a(v)\overrightarrow{\prod_{i\rightarrow k}} \mathcal{L}_v(m_i) +   b(v)K(m_km_k', \prod_{j\leftarrow k} m_j)\overrightarrow{\prod_{j\leftarrow k}} \mathcal{L}_v(m_j). 
\end{equation} 

Otherwise, 

\begin{equation}
\label{eq_semiderived_exchange_relation2}
\mathcal{L}_v(m)\mathcal{L}_v(m') = a(v)K(mm', \prod_{i\rightarrow k} m_i)\overrightarrow{\prod_{i\rightarrow k}} \mathcal{L}_v(m_i)+   b(v)\overrightarrow{\prod_{j\leftarrow k}} \mathcal{L}_v(m_j). 
\end{equation} 

\end{prop}

\begin{proof}

Assume that $m_km_k'= \prod_{i\rightarrow k} m_i$ (the other case is treated in the same way). Expanding both sides of equation 
(\ref{eq_semiderived_exchange_relation}) along the  basis of the semi-derived algebra of Proposition \ref{prop_basis_semiderived}, we find two xspressions:

\[ \sum_{1\leq i\leq N} c_i E_{M_i}K(i)\ \text{ and }\ \sum_{1\leq j\leq N'} c_j' E_{M'_j}K(j)', \]
where $N,N'\in \mathbb{N}$, $c_j,c_j'\in \mathbb{Q}(v^{1/2}),\  $$M_i, M'_j \in D^b(Q)$ and $K(i)$, $K(j)'$ are some monomials in the generators $K_{\alpha,k}$ $(\alpha \in \mathcal{K}(\mathrm{mod}(\Bbbk Q)))$.
It follows from the equality (\ref{eq_exchange_relation_qt_characters}) that the images in $\mathcal{DH}_{\mathrm{tw}}(Q)$ of these two expression coincide, so that we have

\[ \sum_{1\leq i\leq N} c_i Z_{M_i} =  \sum_{1\leq j\leq N'} c_j' Z_{M'_j}, \]
Let $i,j$ be indexes such that $M_i=M_j'$.
Since, by construction, the expressions on both sides of (\ref{eq_semiderived_exchange_relation}) are homogeneous and of same the $\mathcal{K}(C^b(\mathcal{P}))$-degree, it follows from Lemma \ref{lem_uniqueness_K_degree} that $K(i)=K(j)'$ and, therefore, $E_{M_i}K(i)=E_{M_j'}K(j)'$. We deduce that $N=N'$ and, up to reordering the indices, $c_i=c_i'$ and $E_{M_i}K(i)=E_{M_j'}K(j)'$.
\end{proof}

As a corollary, we get that the semiderived $(q,t)$-characters $\mathcal{L}_v(m)$ of the Kirillov--Reshetikhin modules satisfy a lift of the $T$-system to the semi-derived Hall algebra.
For  $(i,p)\in \widehat{I}_{\mathfrak{g}}$ and an integer $k\geq 0$, we define $m_{k,p}^i$ to be the dominant monomial
\[ m_{k,p}^i=Y_{i,p}Y_{i,p+2d_i}\dots Y_{i,p+2(k-1)}.\]
We write $\mathcal{W}_{k,p}^i$ for the semiderived analogue of the Kirillov--Reshetikhin module $W_{k,p}^i$ :

\[ \mathcal{W}_{k,p}^i = \mathcal{L}_v(m_{k,p}^i).\]
Notice that, for any $(i,p)\in \widehat{I}_{\mathfrak{g}}$ and $k\in \mathbb{N}$, we have $m_{k,p}^im_{k,p+2}^i>\prod_{i\sim j} m_{k,p+1}^j$. Therefore, the element $K(m_{k,p}^im_{k,p+2}^i,\prod_{i\sim j} m_{k,p+1}^j)$ is well defined. 

\begin{cor}

Let $(i,p)\in \widehat{I}_{\mathfrak{g}}$ and $k\in \mathbb{N}$. Then the following equality holds in the twisted semiderived Hall algebra

\begin{equation}
    \mathcal{W}_{k,p}^i \mathcal{W}_{k,p+2}^i = t^a\mathcal{W}_{k-1,p+2}^i \mathcal{W}_{k+1,p}^i + t^bK(m_{k,p}^im_{k,p+2}^i,\prod_{i\sim j}m_{k,p+1}^j)\prod_{i\sim j} \mathcal{W}_{k,p+1}^j,
\end{equation}
where $t^a$ and $t^b$ are as in Theorem \ref{teo_quantum_T_system}.
\end{cor}

\section{A quantum cluster algebra structure on the semiderived Hall algebra}
\label{sec_quantu_clust_semider}
Let $[a,b]\subset \mathbb{Z}$ be an interval, with $a\in \{-\infty\}\sqcup \mathbb{Z}$, $b\in \mathbb{Z}$. Recall from Section \ref{sec_Hernadez_Leclerc_iso} that we have fixed an height function $\varepsilon$, a reduced expression $\underline{w}_0$ adapted to $\varepsilon$ of the longest element of the Weyl group of $\mathfrak{g}$ and the admissible sequence $\mathfrak{s}$ associated to the pair $(\varepsilon, \underline{w}_0)$.
Consider the cluster algebra structure on $\mathcal{K}(\mathscr{C}_{\mathfrak{g}}^{[a,b], \mathfrak{s}})$ 
\[ \varphi(\mathfrak{C}^{[a,b]}_-): \mathcal{A}(\widehat{Q}^{[a,b],\mathfrak{s}}) \xrightarrow{\sim} \mathcal{K}(\mathscr{C}_\mathfrak{g}^{[a,b],\mathfrak{s}}),\]
cf.~Theorem~\ref{teo_KKOP_monoidal_cat}.
Denote by $J$ the vertex set of $\widehat{Q}^{[a,b],\mathfrak{s}}$.
For any $(i,p)\in J$, we define the dominant monomial 
\[m_{(i,p)}=m^i_{(p^i_{\mathrm{max}}-p)/2+1,p}=Y_{i,p}Y_{i,p+2}\dots Y_{i,p^i_{\mathrm{max}}},\]
where $p^i_{\mathrm{max}}$ is the greatest integer $p'$ such that $(i,p')$ belongs to the vertex set of $\widehat{Q}^{[a,b],\mathfrak{s}}$.
It follows from Theorem~\ref{teo_KKOP_monoidal_cat} that the initial cluster is 
\[ (M(\mathfrak{c})_{\mathfrak{c}\in \mathfrak{C}_-^{[a,b]}})=([L(m_{(i,p)})])_{(i,p)\in \widehat{Q}^{[a,b],\mathfrak{s}}}.\]
The set of frozen vertices of
$\widehat{Q}^{[a,b],\mathfrak{s}}$ is
\[J^{fr}= \{(i,p) \in \widehat{Q}^{[a,b],\mathfrak{s}}\ |\ p\leq p' \text{ for any } (i,p') \text{ in } \widehat{Q}^{[a,b],\mathfrak{s}} \}.\]
We write $J^\mu=J\backslash J^{fr}$ for the set of mutable vertices.
Consider the parametrization $(x(e),Q(e))$ of the seeds of $\mathcal{K}(\mathscr{C}_\mathfrak{g}^{[a,b],\mathfrak{s}})$ by the vertices $e$ of the regular $J^{\mu}$-tree $\mathbb{T}$ rooted at $e_0$. 
In particular, 
\[(x(e_0),Q(e_0)) = (([L(m^i_{(p^i_{\mathrm{max}}-p)/2+1,p}])_{(i,p)}, \widehat{Q}^{[a,b],\mathfrak{s}}).\]
For any vertex $e$ of $\mathbb{T}$ and any vertex $i$ of $\widehat{Q}^{[a,b],\mathfrak{s}}$, we write $m_i(e)$ for the dominant monomial such that $x_i(e)=[L(m_i(e))]$. Notice that $m_{(i,p)}(e_0)=m_{(i,p)}$.

\begin{lem}
\label{lem_mutation_m_greenred}
Let $e$ and $e'$ be vertices of $\mathbb{T}$ linked by an edge labelled by the mutable vertex $k$ of $Q(e)$. The vertex k is green if and only if the product of the dominant monomials $m_k(e)$ and $m_k(e')$ is equal to the product of the monomials associated to the vertices receiving an arrow from $k$:
\[m_k(e)m_k(e')= \prod_{k\rightarrow i} m_j(e). \]
\end{lem}

\begin{proof}
By the mutation rule for $g$-vectors (\ref{eq_g_vector}), we have

\[
g_k(e')=\begin{cases}
    -g_k(e)+\sum_{k\rightarrow k'} g_{k'}(e) & \text{if $k$ is green in $Q(e)$ with respect to $e_0$},\\
    -g_k(e)+\sum_{k'\rightarrow k} g_{k'}(e) & \text{if $k$ is red in $Q(e)$ with respect to $e_0$ }.
\end{cases}
\]
As shown in the proof of \cite[Prop.~4.16]{HL_Clust_alg_approach_q_char}, for this choice of the inital seed, for any vertex $(i,p)\in \widehat{Q}^{[a,b],\mathfrak{s}}$, we have

\[ m_{(i,p)}(e) = \prod_{j,s}(\prod_{k\geq 0, r+2k\leq p^i_{\mathrm{max}}} Y_{i,r+2k})^{g_{(i,p)}(e;j,s)}, \]
where $g_{(i,p)}(e;j,s)$ is the $(j,s)$-component of $g_{(i,p)}$.

It follows that 
\[
m_k(e')=\begin{cases}
    m_k(e)^{-1}\prod_{k\rightarrow k'} m_{k'}(e) & \text{if $k$ is green in $Q(e)$ with respect to $e_0$},\\
    m_k(e)^{-1}\prod_{k'\rightarrow k} m_{k'}(e) & \text{if $k$ is red in $Q(e)$ with respect to $e_0$ }.
\end{cases}
\]
\end{proof}

\begin{con}
\label{con_QSDH_quiver}
For any mutable vertex $(i,p)$ in $\widehat{Q}^{[a,b],\mathfrak{s}}$, using the notation of Section \ref{sec_Hernadez_Leclerc_iso}, we define the modules

\[V_{(i,p)}^\text{in}=\bigoplus_{(j,s)\rightarrow (i,p) }V(m_{(j,s)})=(\bigoplus_{p+2\leq p+2+2k\leq p^i_{\mathrm{max}}} V(i,p+2+2k))\oplus \bigoplus_{i\sim j} V(j,p-1)\oplus\dots\oplus V(j,p^j_{\mathrm{max}})\]
and 

\[V_{(i,p)}^\text{out}=\bigoplus_{(j,s)\leftarrow (i,p) }V(m_{(j,s)})=V(i,p-2)\oplus\dots \oplus V(i,p^i_{\mathrm{max}})\bigoplus_{i\sim j} (\bigoplus_{p+1\leq p+1+2k\leq p^j_{\mathrm{max}}} V(j,p+1+2k)),\]
where, by convention, we omit the terms $V(i',p')$ if $(i',p')$ is not a vertex of $\widehat{Q}^{[a,b],\mathfrak{s}}$.
Notice that $V_{(i,p)}^\text{in}$ is smaller in the degeneration order than $V_{(i,p)}^\text{out}$ since $\bigoplus_{i\sim j} V(j,p-1)$
can be obtained as an extension of $V(i,p)$ by $V(i,p-2)$. In particular, it follows that if $V(i,p-2)$ and $V(i,p)$ belong to the same shift of the canonical heart of $D^b(Q)$, then the corresponding elements of the semi-derived Hall algebra have the same degree:
\begin{equation}
\label{eq_deg_Vin_Vout_1}
    \mathrm{deg}(E_{V_{(i,p)}^\text{out}}) = \mathrm{deg}(E_{V_{(i,p)}^\text{in}}),
\end{equation}
cf. Theorem \ref{teo_relations_semiderived}.
Otherwise, we have 
\[ V(i,p) = P_j[z+1] \text{ and } V(i,p-2) = I_j[z], \]
for an integer $z$ and $j=i$ or $i^*$. Consider the exact sequence
\[ 0\rightarrow  \mathrm{rad}(P_j) \rightarrow P_j \rightarrow I_j \rightarrow I_j/S_j \rightarrow 0.\]
It follows from the relations (Theorem \ref{teo_relations_semiderived}) of the twisted semi-derived Hall algebra that we have 
\begin{equation}
\label{eq_deg_Vin_Vout_2} \mathrm{deg}(E_{V_{(i,p)}^\text{out}}) = \mathrm{deg}(E_{V_{(i,p)}^\text{in}}K_{\overline{S}_j,z}).\end{equation}
 We extend $\widehat{Q}^{[a,b],\mathfrak{s}}$ to a quiver $\widetilde{Q}^{[a,b],\mathfrak{s}}$ in the following way:
\begin{itemize}
    \item for each non-frozen vertex $(i,p)$ such that $V(i,p)=P_j[z+1]$ for $j=i$ or $j=i^*$ and $z\in \mathbb{Z}$, add a frozen vertex $\widetilde{S}_{j,z}$ and a frozen arrow $\widetilde{S}_{j,z}\rightarrow (i,p)$.
\end{itemize}
\end{con}

Let $\widetilde{J}$ be the set of vertices of $\widetilde{Q}^{[a,b],\mathfrak{s}}$. We write $J^{fr}_K$ for the set of frozen vertices of the form $\widetilde{S}_{j,z}$. Therefore, we have
\begin{equation}
\label{eq_subsets_J}
 \widetilde{J}=J\sqcup J_K^{fr}=J^\mu\sqcup J^{fr} \sqcup J^{fr}_K,
 \end{equation}
where the union $J^{fr} \sqcup J^{fr}_K,$ is the set of frozen vertices of $\widetilde{Q}^{[a,b],\mathfrak{s}}$.

\begin{ex}
    Let $\mathfrak{g}$ be of type $A_3$. Let $\varepsilon$ be the height function defined by
    \[ \varepsilon_1=1,\ \varepsilon_2=0,\ \varepsilon_3=1\]
    and fix the expression $\underline{w}_0=s_2s_1s_3s_2s_1s_3$ adapted to $\varepsilon$. The quiver $\widetilde{Q}^{[-\infty,3],\mathfrak{s}}$ is represented in Figure~\ref{fig_quiver_semiderived_-_A3}, where we have labelled each vertex $(i,p)$ by $V(i,p)$.
\end{ex}

\begin{figure}
\centering

\begin{tikzpicture}[xscale= 2.2, yscale=2.2]

\node[blue, scale=0.8] (K11) at (6.5, 4.6) {$\widetilde{S}_{1,-1}$};
\node[blue, scale=0.8] (K12) at (4.3, 4.6) {$\widetilde{S}_{3,-2}$};
\node[blue, scale=0.8] (K13) at (2.1, 4.6) {$\widetilde{S}_{1,-3}$};

\node[blue, scale=0.8] (K21) at (6, 3.6) {$\widetilde{S}_{2,-1}$};
\node[blue, scale=0.8] (K22) at (3.8, 3.6) {$\widetilde{S}_{2,-2}$};
\node[blue, scale=0.8] (K23) at (1.6, 3.6) {$\widetilde{S}_{2,-3}$};

\node[blue, scale=0.8] (K31) at (6.5, 2.6) {$\widetilde{S}_{3,-1}$};
\node[blue, scale=0.8] (K32) at (4.3, 2.6) {$\widetilde{S}_{1,-2}$};
\node[blue, scale=0.8] (K33) at (2.1, 2.6) {$\widetilde{S}_{3,-3}$};

    \node (N10) at (7.6,4) {$I_3$};
    \node (N11) at (6.5,4) {$P_1$};
    \node (N12) at (5.4,4) {$I_1[-1]$};
    \node (N13) at (4.3,4) {$P_3[-1]$};
    \node (N14) at (3.2,4) {$I_3[-2]$};
    \node (N15) at (2.1,4) {$P_1[-2]$};
    \node (N16) at (1,4) {$I_1[-3]$};

    \node (N20) at (7.1,3) {$I_2$};
    \node (N21) at (6,3) {$P_2$};
    \node (N22) at (4.9,3) {$I_2[-1]$};
    \node (N23) at (3.8,3) {$P_2[-1]$};
    \node (N24) at (2.7,3) {$I_2[-2]$};
    \node (N25) at (1.6,3) {$P_2[-2]$};
    \node (N26) at (0.5,3) {$I_2[-3]$};

    \node (N30) at (7.6,2) {$I_1$};
    \node (N31) at (6.5,2) {$P_3$};
    \node (N32) at (5.4,2) {$I_3[-1]$};
    \node (N33) at (4.3,2) {$P_1[-1]$};
    \node (N34) at (3.2,2) {$I_1[-2]$};
    \node (N35) at (2.1,2) {$P_3[-2]$};
    \node (N36) at (1,2) {$I_3[-3]$};

    \node (d1) at (0.5,4) {$\dots$};
    \node (d2) at (0,3) {$\dots$};
    \node (d3) at (0.5,2) {$\dots$};

    \draw[->] (N10) -- (N11);
    \draw[->] (N11) -- (N12);
    \draw[->] (N12) -- (N13);
    \draw[->] (N13) -- (N14);
    \draw[->] (N14) -- (N15);
    \draw[->] (N15) -- (N16);

    \draw[->] (N11) -- (N20);
     \draw[->] (N12) -- (N21);
      \draw[->] (N13) -- (N22);
       \draw[->] (N14) -- (N23);
        \draw[->] (N15) -- (N24);
         \draw[->] (N16) -- (N25);

    \draw[->] (N20) -- (N21);
    \draw[->] (N21) -- (N22);
    \draw[->] (N22) -- (N23);
    \draw[->] (N23) -- (N24);
    \draw[->] (N24) -- (N25);
    \draw[->] (N25) -- (N26);
    \

    \draw[->] (N20) -- (N10);
    \draw[->] (N21) -- (N11);
    \draw[->] (N22) -- (N12);
    \draw[->] (N23) -- (N13);
    \draw[->] (N24) -- (N14);
    \draw[->] (N25) -- (N15);
    \draw[->] (N26) -- (N16);

    \draw[->] (N20) -- (N30);
    \draw[->] (N21) -- (N31);
    \draw[->] (N22) -- (N32);
    \draw[->] (N23) -- (N33);
    \draw[->] (N24) -- (N34);
    \draw[->] (N25) -- (N35);
    \draw[->] (N26) -- (N36);

   \draw[->] (N30) -- (N31);
   \draw[->] (N31) -- (N32);
    \draw[->] (N32) -- (N33);
    \draw[->] (N33) -- (N34);
    \draw[->] (N34) -- (N35);
    \draw[->] (N35) -- (N36);

    \draw[->] (N31) -- (N20);
     \draw[->] (N32) -- (N21);
      \draw[->] (N33) -- (N22);
       \draw[->] (N34) -- (N23);
        \draw[->] (N35) -- (N24);
         \draw[->] (N36) -- (N25);

    \draw[->, blue] (K11) -- (N11);
    \draw[->, blue] (K12) -- (N13);
    \draw[->, blue] (K13) -- (N15);

    \draw[->, blue] (K21) -- (N21);
    \draw[->, blue] (K22) -- (N23);
    \draw[->, blue] (K23) -- (N25);

   \draw[->, blue] (K31) -- (N31);
    \draw[->, blue] (K32) -- (N33);
    \draw[->, blue] (K33) -- (N35);
    
\end{tikzpicture}
\caption{The quiver $\widetilde{Q}^{[-\infty,3],\mathfrak{s}}$ for $\mathfrak{g}$ of type $A_3$}
  \label{fig_quiver_semiderived_-_A3}
\end{figure}
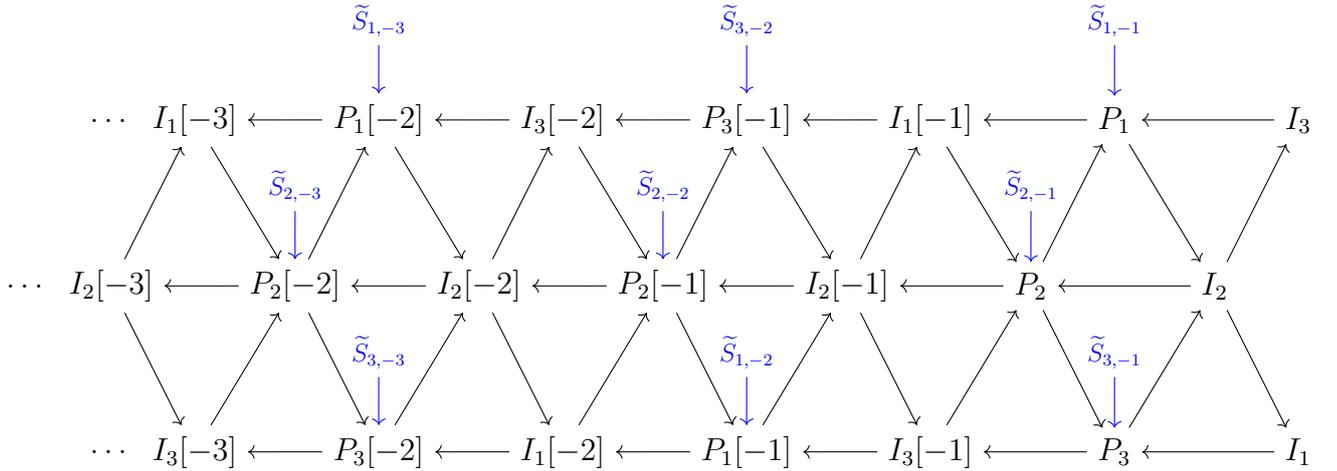

We define the $J\times J$-antisymmetric matrix $\Lambda^{[a,b],\mathfrak{s}}=(\lambda_{(i,p),(j,s)})$ by
\[ \lambda_{(i,p),(j,s)} = \Lambda(L(m_{(i,p)}), L(m_{(j,s)})),\]
where the expression on the right-hand side is the invariant defined by Kashiwara--Kim--Oh--Park, cf Remark \ref{rem_invariant_Lambda}.
Let $\widehat{B}^{[a,b],\mathfrak{s}}$ be the exchange matrix of the quiver $\widehat{Q}^{[a,b],\mathfrak{s}}$.
Up to reordering the indices, the matrix $\Lambda^{[a,b],\mathfrak{s}}=(\lambda_{(i,p),(j,s)})$ coincides with the matrix $\Lambda(\mathfrak{C}^{[a,b]}_-)$. In particular, it follows from \ref{teo_quantum_analogue_monoidal_cat} that $(\Lambda^{[a,b],\mathfrak{s}},\widehat{B}^{[a,b],\mathfrak{s}})$ is a compatible pair. Let $\mathcal{A}_t(\widehat{Q}^{[a,b],\mathfrak{s}})$ be the associated quantum cluster algebra.
Let us extend $\Lambda^{[a,b],\mathfrak{s}}$ to a $\widetilde{J}\times \widetilde{J}$-antisymmetric matrix $\widetilde{\Lambda}^{[a,b],\mathfrak{s}}$ by setting each new entry to 0. 
Notice that $(\widetilde{\Lambda}^{[a,b],\mathfrak{s}},\widetilde{B}^{[a,b],\mathfrak{s}})$ is still a compatible pair. We write $\mathcal{A}_t(\widetilde{Q}^{[a,b],\mathfrak{s}})$ for the associated quantum cluster algebra \emph{localized} at the frozen quantum cluster variables associated to the vertices of $J_K^{fr}$. We denote by $\pi$ the surjection
\[\pi: \mathcal{A}_t(\widetilde{Q}^{[a,b],\mathfrak{s}}) \rightarrow \mathcal{A}_t(\widehat{Q}^{[a,b],\mathfrak{s}}),\]
which sends each frozen quantum cluster variable associated to a vertex of $J_K^{fr}$ to 1.

Recall that (cf.~Section~\ref{sec_semiderived_Hall_complexes}), for any $i\in I_\mathfrak{g}$, we have fixed a minimal projective resolution
\[ 0\rightarrow P_{S_i} \rightarrow Q_{S_i} \rightarrow S_i.\]
Similarly to Lemma \ref{lem_uniqueness_K_degree}, we can prove the following
\begin{lem}
\label{lem_invertible_frozen_to_Groth_projectives}
Let $\mathcal{A}_t(\widetilde{Q}^{[a,b],\mathfrak{s}})_K^{\times} $ be the subgroup of the group of invertible elements of $\mathcal{A}_t(\widetilde{Q}^{[a,b],\mathfrak{s}})$ generated by the frozen quantum cluster variables $X_{\widetilde{S}_{j,z}}$, for $\widetilde{S}_{j,z}$ in $J_K^{fr}$. The map 
\[\mathcal{A}^{\times}_t(\widetilde{Q}^{[a,b],\mathfrak{s}})_K  \rightarrow \mathcal{K}(C^b(\mathcal{P})),\ X_{\widetilde{S}_{j,z}} \mapsto K_{\overline{S}_{j},z}=[K_{Q_{S_j}}[z]][K_{P_{S_j}}[z]]^{-1}\]is injective.
\end{lem}

By considering $\mathbb{Q}(v^{1/2})$ as a $\mathbb{Z}[t^{\pm 1/2}]$-algebra through the specialization $t^{1/2}\mapsto v^{1/2}$, we define the $\mathbb{Q}(v^{1/2})$-algebras

\begin{align*}
\mathcal{A}_v(\widehat{Q}^{[a,b],\mathfrak{s}}) &= \mathbb{Q}(v^{1/2})\otimes_{\mathbb{Z}[t^{\pm 1/2}]} \mathcal{A}_t(\widehat{Q}^{[a,b],\mathfrak{s}}) \text{ and } \\\mathcal{A}_v(\widetilde{Q}^{[a,b],\mathfrak{s}}) &= \mathbb{Q}(v^{1/2})\otimes_{\mathbb{Z}[t^{\pm 1/2}]} \mathcal{A}_t(\widetilde{Q}^{[a,b],\mathfrak{s}}).
\end{align*}

We assign a $\mathcal{K}(C^b(\mathcal{P}))$-degree to each initial quantum cluster variable of $\mathcal{A}_v(\widetilde{Q}^{[a,b],\mathfrak{s}})$ as follows:

\begin{itemize}
    \item for any vertex of $ \widetilde{Q}^{[a,b],\mathfrak{s}}$ of the form $(i,p)$, we set
    \[ \mathrm{deg}(X_{(i,p)}) = \mathrm{deg}(\mathcal{L}_v(m_{(i,p)})); \]
    \item for any vertex of $ \widetilde{Q}^{[a,b],\mathfrak{s}}$ of the form $\widetilde{S}_{j,z}$, we set 
    \[\mathrm{deg}(X_{\widetilde{S}_{j,z}})=\mathrm{deg}(K_{\overline{S}_j,z}).\]
\end{itemize}

Consider the parametrization $(X(e),Q(e))$ of the seeds of $\mathcal{A}_v(\widetilde{Q}^{[a,b],\mathfrak{s}})$ by the vertices of the regular $J^{\mu}$-tree $\mathbb{T}$ rooted at $e_0$. Write $b_{ij}$ for the $ij$-coefficient of the matrix $\widetilde{B}^{[a,b],\mathfrak{s}}$.
It follows from the construction of the quiver $\widetilde{Q}^{[a,b],\mathfrak{s}}=(b_{ij})_{i,j\in \widetilde{J}}$ that, for each mutable vertex $k$ of $\widetilde{Q}^{[a,b],\mathfrak{s}}$, each exchange relation is homogeneous, that is
\[ \sum_{i\rightarrow k} b_{ik}\mathrm{deg}(X_i(e_0)) = \sum_{j\rightarrow i} -b_{jk}\mathrm{deg}(X_j(e_0)).\]
As shown in \cite[\S~3]{Grabowsky_Launois_2014}, this allows to recursively define the degree of each quantum cluster variable, making $\mathcal{A}_v(\widetilde{Q}^{[a,b],\mathfrak{s}})$ into a $\mathcal{K}(C^b(\mathcal{P}))$-graded algebra.

\begin{rem}
\label{rem_subquiver_framed_quiver}
Notice that we can obtain the quiver $\widetilde{Q}^{[a,b],\mathfrak{s}}$ by removing from the framed quiver of $\widehat{Q}^{[a,b],\mathfrak{s}}$ some vertices of the type $i'$, for $i \in J^\mu$ (cf.~subsection~\ref{subsect_green_red_vertices}). It follows that the vertex $i$ of $Q(e)$ is green (resp. red) with respect to $e_0$ if and only if then there are no arrows $i\rightarrow \widetilde{S}_{j,z}$ (resp. $\widetilde{S}_{j,z} \rightarrow i$), for a vertex $\widetilde{S}_{j,z} \in J^{fr}_{K}$ .
\end{rem}

Let $\mathcal{DH}_{tw}(Q)^{[a,b],\mathfrak{s}}$ be the subalgebra of $\mathcal{DH}_{tw}(Q)$ generated by the $Z_{V(i,p)}$, for $(i,p)\in J$.
Let $\mathcal{SDH}_{tw}(C^b(\mathcal{P}))^{[a,b],\mathfrak{s}}$ be the subalgebra of $\mathcal{SDH}_{tw}(C^b(\mathcal{P}))$ generated by the $E_{V(i,p)}$, for $(i,p)\in J$, and by the $K_{\overline{S}_j,z}$, for each frozen vertex $\widetilde{S}_{j,z}$ of $J_K^{fr}$. Notice that this subalgebra is generated by homogeneous elements and thus inherits the $\mathcal{K}(C^b(\mathcal{P}))$-grading of the twisted semi-derived Hall algebra.

\begin{prop}
\label{prop_l'image_est_la_bonne}
For any dominant monomial $m$ in the variables $(Y_{i,p})_{(i,p)\in J}$, the semi-derived $(q,t)$-character $\mathcal{L}_v(m)$ is contained in $\mathcal{SDH}_{tw}(C^b(\mathcal{P}))^{[a,b],\mathfrak{s}}$.
\end{prop}

\begin{proof}
Recall the definition of $\mathcal{L}_v(m)$ given in (\ref{eq_semiderived_Lv(m)}). 
Let $m'\leq m$ be a dominant monomial. Therefore, there is a positive integer $N$ and a sequence of indices $(i_k,p_k)$ in $\widehat{I}_{\mathfrak{g}}$ such that
\[ m = m' \prod_{1\leq k\leq N} A_{i_k,p_k+1}.\]
It follows from the fact that $\mathfrak{s}$ is an admissible sequence that
\begin{itemize}
    \item[(i)] $m'$ is a monomial in the variables $Y_{i,p}$, $(i,p)\in J$;
    \item[(ii)] for any $k$, the vertices $(i_k,p_k), (i_k,p_k+2)$ and $(j,p_k+1)$, for $j\sim i_k$, are in $J$.
\end{itemize}
 It follows from (i) that $E_{V(m')}$ is an ordered product of some elements $E_{V(i,p)}$, with $(i,p)\in J$. Assign an element of $\mathcal{K}(C^b(\mathcal{P})$ to each monomial $\widetilde{m}$ in the variables $(Y_{i,p}^{\pm})_{(i,p)\in\widehat{I}_{\mathfrak{g}}}$, by extending linearly the assignment 
 \[ \mathrm{deg}(Y_{i,p}^{\pm} ) = \pm \mathrm{deg}(E_{V(i,p)}).\]
 Therefore, for any dominant monomial $m$, we have that $\mathrm{deg}(m)=\mathrm{deg}(E_{V(m)})$. Moreover, similarly to how we have deduced (\ref{eq_deg_Vin_Vout_1}) and (\ref{eq_deg_Vin_Vout_2}), we can deduce from (ii) and the relations of the twisted semi-derived Hall algebra (cf.~Theorem~\ref{teo_relations_semiderived}) that, for any $1\leq k\leq N$, we have
 \[
 \mathrm{deg}(A_{i_k,p_k+1})=\begin{cases}
     \mathrm{deg}(K_{\overline{S}_j,z}) &\text{if  $V(i_k,p_k+2)=P_j[z+1]$, for $j\in \{i,i^*\}$ and $z\in \mathbb{Z}$},\\
     0 & \text{otherwise.}
 \end{cases}
 \]
 Therefore, we have that 
 \[ \mathrm{deg}(E_{V(m)}) = \mathrm{deg}(E_{V(m')}\prod_{k'} K_{\overline{S}_{j_{k'}},z_{k'}}),\]
 for certain $\widetilde{S}_{j_{k'},z_{k'}}$ in $J_{K}^{fr}$. It follows from Lemma \ref{lem_find_a_K} that 
 \[  K(m,m')=  \prod_{k'} K_{\overline{S}_{j_{k'}},z_{k'}}.\]
 In conclusion, since $E_{V(m')}K(m,m')$ is contained in $\mathcal{SDH}_{tw}(C^b(\mathcal{P}))^{[a,b],\mathfrak{s}}$ for any $m'\leq m$, the same is true for $\mathcal{L}_v(m)$.
\end{proof}

Let $\mathcal{K}_v(\mathscr{C}_\mathfrak{g}^{[a,b],\mathfrak{s}})$ be the $\mathbb{Q}(v^{1/2})$-algebra
\[ \mathcal{K}_v(\mathscr{C}_\mathfrak{g}^{[a,b],\mathfrak{s}})=\mathbb{Q}(v^{1/2})\otimes_{\mathbb{Z}[t^{\pm 1/2}]}\mathcal{K}_t(\mathscr{C}_\mathfrak{g}^{[a,b],\mathfrak{s}}).\]
We write $\varphi_v^{[a,b],\mathfrak{s}}$ for  the $\mathbb{Q}(v^{1/2})$-isomorphism 
\[ \mathcal{A}_v(\widehat{Q}^{[a,b],\mathfrak{s}}) \rightarrow \mathcal{K}_v(\mathscr{C}_\mathfrak{g}^{[a,b],\mathfrak{s}})
\]
induced by $\varphi_t(\mathfrak{C}_-^{[a,b]})$, cf.~Theorem~\ref{teo_quantum_analogue_monoidal_cat}.

\begin{teo}
\label{teo_quantum_clust_semiderived_[a,b]}
    There is an isomorphism of $\mathcal{K}(C^b(\mathcal{P}))$-graded $\mathbb{Q}(v^{1/2})$-algebras 
    \[\widetilde{\theta}: \mathcal{A}_v(\widetilde{Q}^{[a,b],\mathfrak{s}}) \xrightarrow{\sim} \mathcal{SDH}_{tw}(C^b(\mathcal{P}))^{[a,b],\mathfrak{s}} \]
which makes the following diagram commute:

\[
\begin{tikzcd} \mathcal{A}_v(\widetilde{Q}^{[a,b],\mathfrak{s}}) \arrow[d, "\pi", twoheadrightarrow] \arrow[rr, dashed, "\Tilde{\theta}"] & &\mathcal{SDH}_{tw}(C^b(\mathcal{P}))^{[a,b],\mathfrak{s}} \arrow[d, "\pi^{\mathcal{H}}", twoheadrightarrow] \\
    \mathcal{A}_v(\widehat{Q}^{[a,b],\mathfrak{s}}) \arrow[r, "\sim", "\varphi_v^{[a,b],\mathfrak{s}}"'] & \mathcal{K}_v(\mathscr{C}_\mathfrak{g}^{[a,b],\mathfrak{s}}) \arrow[r, "\Phi"', "\sim"] &\mathcal{DH}_{tw}(Q)^{[a,b],\mathfrak{s}}.
\end{tikzcd}
\]
\end{teo}

\begin{proof}
Let $\theta$ be the composition $\theta=\Phi\circ \varphi_v^{[a,b],\mathfrak{s}}$.
We start by defining a map from the set of quantum cluster variables of $\mathcal{A}_v(\widetilde{Q}^{[a,b],\mathfrak{s}})$ to $\mathcal{SDH}_{tw}(C^b(\mathcal{P}))^{[a,b],\mathfrak{s}}$. Let $X$ be a non-frozen quantum cluster variable of $\mathcal{A}_v(\widetilde{Q}^{[a,b],\mathfrak{s}})$. Let $m\in \mathcal{M}^+$ be the dominant monomial such that $\theta(\pi(X))=\Phi(L_t(m))$. We set
\[ \widetilde{\theta}: X \mapsto \mathcal{L}_v(m).\]
For any frozen vertex $\widetilde{S}_{j,z}$ in $J_K^{fr}$, we set
\[ \widetilde{\theta}: X_{\widetilde{S}_{j,z}} \mapsto K_{\overline{S}_j,z}.\]
It follows from Proposition \ref{prop_l'image_est_la_bonne} that, under this assignment, the image of each quantum cluster variable is contained in $\mathcal{SDH}_{tw}(C^b(\mathcal{P}))^{[a,b],\mathfrak{s}}$.
We want to show that the we can extend the map $\widetilde{\theta}$ to a graded-algebra homomorphsim. \\
First of all, we show by induction that, if $X$ is a quantum cluster variable such that $\Tilde{\theta}(X)=\mathcal{L}_v(m)$, then 
\begin{equation}
\label{eq_degree_cluster_Ltm}
\mathrm{deg}(X)=\mathrm{deg}(\mathcal{L}_v(m)).
\end{equation}
We proceed by recursion on the vertices of the regular $J^{\mu}$-tree $\mathbb{T}$.
For the initial cluster it is true by the definition of the $\mathcal{K}(C^b(\mathcal{P}))$-degree on the quantum cluster algebra $\mathcal{A}_v(\widetilde{Q}^{[a,b],\mathfrak{s}})$. Let $e$ and $e'$ be vertices of $\mathbb{T}$ linked by an edge labelled by $k\in J^{\mu}$ such that (\ref{eq_degree_cluster_Ltm}) holds for any $X_i(e)\in X(e)$. We need to show that (\ref{eq_degree_cluster_Ltm}) holds for $X_k(e')$. Suppose that the vertex $k$ is green in $Q(e)$ with respect to $e_0$ (if it is red, we can proceed in a similar way). The exchange relation associated to the mutation of $X(e)$ at $k$ is of the form
\begin{equation}
\label{eq_exchange_rel_X_k(e)}
X_k(e)X_k(e')=a(v)\overrightarrow{\prod_{i\rightarrow k}} X_i(e) + b(v)\Bigl(\overrightarrow{\prod_{k\rightarrow i'}} X_{i'}(e)\Bigr) \prod_{\widetilde{S}_{j,z}\rightarrow k}X_{\widetilde{S}_{j,z}},
\end{equation}
where the coefficients $a(v)$ and $b(v)$ in $\mathbb{Q}(v^{1/2})$ are the images under the specialization $t^{1/2}\rightarrow v^{1/2}$ of the corresponding coefficients $a(t)$ and $b(t)$ of (\ref{eq_exchange_relation_qt_characters}). Indeed, the image in $\mathcal{K}_v(\mathscr{C}_\mathfrak{g}^{[a,b],\mathfrak{s}})$ of (\ref{eq_exchange_rel_X_k(e)}) via $\varphi_v^{[a,b],\mathfrak{s}}\circ \pi$ gives the exchange relation 
\[L_t(m_k(e))L_t(m_k(e'))=a(v)\overrightarrow{\prod_{i\rightarrow k}} L_t(m_i(e)) + b(v)\overrightarrow{\prod_{k\rightarrow i'}} L_t(m_{i'}(e)).\]
Since $k$ is green in $Q(e)$ with respect to $e_0$, it follows from Lemma \ref{lem_mutation_m_greenred} that we have
\begin{equation}
\label{eq_prod_m_in_proof}
m_k(e)m_k(e')= \prod_{i\rightarrow k}m_i(e).
\end{equation}
Moreover, by induction, we have that
\[
\mathrm{deg}\bigr(\overrightarrow{\prod_{i\rightarrow k}} \mathcal{L}_v(m_i(e))\bigl)=\mathrm{deg}(\overrightarrow{\prod_{i\rightarrow k}} X_i(e))\ \text{ and }\ \mathrm{deg}\bigl(\mathcal{L}_v(m_k(e))\bigr)=\mathrm{deg}(X_k(e)).
\]
We can deduce that
\begin{align*}
\mathrm{deg}\bigl(\mathcal{L}_v(m_k(e'))\bigr)&=\mathrm{deg}  \bigl(\overrightarrow{\prod_{i\rightarrow k}} \mathcal{L}_v(m_i(e))\bigr)-\mathrm{deg}\bigl(\mathcal{L}_v(m_k(e))\bigl)\\
&=\mathrm{deg}(\overrightarrow{\prod_{i\rightarrow k}} X_{i}(e))-\mathrm{deg}(X_k(e))=\mathrm{deg}(X_k(e')),
\end{align*}
where the first equality follows from (\ref{eq_prod_m_in_proof}) and Proposition \ref{prop_semiderived_exchange_rel}.

Let $N$ be a positive integer and let $f$ be an homogeneous polynomial expression degree in positive powers of certain quantum cluster variables $X_1,\dots,X_N$ and possibly negative powers of the frozen quantum cluster variables associated to the vertices of $J_K^{fr}$. Suppose that $f$ equals $0$ in $\mathcal{A}_v(\widetilde{Q}^{[a,b],\mathfrak{s}})$.
We need to show that  $\widetilde{\theta}(f) = 0$ in $\mathcal{SDH}_{tw}(C^b(\mathcal{P}))^{[a,b],\mathfrak{s}}$.
Consider the expansion of $\widetilde{\theta}(f)$ along the  basis of the semi-derived algebra of Proposition \ref{prop_basis_semiderived}:

\[ \widetilde{\theta}(f)= \sum_{1\leq s\leq N'} b_s E_{M_s}K(s)\ \]
where $N'\in \mathbb{N}$, $b_s\in \mathbb{Q}(v^{1/2}),\ M_s\in D^b(Q)$ and $K(s)$ are some monomials in the generators $K_{\alpha,k}$ $(\alpha \in \mathcal{K}(\mathrm{mod}(\Bbbk Q)))$.
Since $f$ is homogeneous, it follows from (\ref{eq_degree_cluster_Ltm}) that we have that $\mathrm{deg}(E_{M_s}K(s))=\mathrm{deg}(E_{M_{s'}}K(s'))$ for any $1\leq s,s'\leq N'$.

Consider that
\[ \theta(\pi(f)) = \pi^{\mathcal{H}}(\widetilde{\theta}(f))= \sum_{1\leq s\leq N'} b_s Z_{M_s} = 0 .\]

Suppose that there exist $s\neq s'$ such that $Z_{M_s}=Z_{M_{s'}}$. Then $E_{M_s}=E_{M_{s'}}$, and, since $\mathrm{deg}(E_{M_s}K(s))=\mathrm{deg}(E_{M_{s'}}K(s'))$ we would have, by Lemma \ref{lem_find_a_K}, $K(s)=K(s')$, too. Therefore $E_{M_s}K(s)=E_{M_{s'}}K(s')$. Since this is not possible, all the $Z_{M_s}$ are distinct. Therefore, $b_s=0$ for any $s$. Therefore, the map $\widetilde{\theta}$ extends to a graded-algebra homomorphsim.

For the injectivity of $\widetilde{\theta}$, let $g$ be an homogeneous element of $\mathcal{A}_v(\widetilde{Q}^{[a,b],\mathfrak{s}})$ such that $\widetilde{\theta}(g)=0$. 
Write $g$ as a Laurent polynomial in the initial cluster variables:

\[ g = \sum_{s=1}^{N''} \Bigl( a_s \bigr(\overrightarrow{\prod_{i\in J}} X_{i}^{c(s;i)}\bigr) \prod_{\alpha \in J^{fr}_K} X_\alpha^{c(s;\alpha)'} \Bigr) ,\]
for some $a_s \in \mathbb{Q}(v^{1/2})$ and  $c(s;j), c(s;\alpha)' \in \mathbb{Z}$.

Therefore, we also have that \[\pi(g)=\sum_{s=1}^{N''} a_s \bigr(\overrightarrow{\prod_{i\in J}} X_{i}^{c(s;j)}\bigr)  =0.\] 
Let $s\neq s'$ such that 
\[ \bigr(\overrightarrow{\prod_{i\in J}} X_{i}^{c(s;j)}\bigr) = \bigr(\overrightarrow{\prod_{i\in J}} X_{i}^{c(s';j)}\bigr),\]
which implies
\[ \mathrm{deg(\prod_{\alpha \in J^{fr}_K} X_\alpha^{c(s;\alpha)'}) }= \mathrm{deg}(\prod_{\alpha \in J^{fr}_K} X_\alpha^{c(s';\alpha)'}).\]
Since it follows from Lemma \ref{lem_invertible_frozen_to_Groth_projectives} that
\[\prod_{\alpha \in J^{fr}_K} X_\alpha^{c(s;\alpha)'}=\prod_{\alpha \in J^{fr}_K} X_\alpha^{c(s';\alpha)'},\]
we can conclude that $g=0$.

For the surjectivity of $\widetilde{\theta}$, let $(i,p)$ be in $J$ and let $c\in [a,b]$ be the vertex of $Q(\underline{w}_0)$ corresponding to $(i,p)$ under the bijection of section \ref{sec_Hernadez_Leclerc_iso}. Let $\mathfrak{C}=(\mathfrak{c}_s)_s$ be a chain of $i$-boxes of range $[a,b]$ such that $\mathfrak{c}_1=\{c\}$. The canonical chain transformation $\mathfrak{C}\mapsto \mathfrak{C}_{-}^{[a,b]} $ (cf.~Remark~\ref{rem_sequence_canonical_chain_transform}) induces an isomorphism of quantum cluster algebras $\mathcal{A}(Q(\mathfrak{C}),\Lambda(\mathfrak{C}))\cong\mathcal{A}_t(\widehat{Q}^{[a,b],\mathfrak{s}})$ such that the composition
\[ \mathcal{A}(Q(\mathfrak{C}),\Lambda(\mathfrak{C}))\cong\mathcal{A}_t(\widehat{Q}^{[a,b],\mathfrak{s}})\xrightarrow{\varphi_t(\mathfrak{C}_-^{[a,b]})}  \mathcal{K}_t(\mathscr{C}_\mathfrak{g}^{[a,b],\mathfrak{s}}) \]
is the isomorphism $\varphi_t(\mathfrak{C})$ of Theorem \ref{teo_quantum_analogue_monoidal_cat}. Let $X_1'$ be the initial quantum cluster variable of $\mathcal{A}(Q(\mathfrak{C}),\Lambda(\mathfrak{C}))$ associated to the vertex $1\in K(\mathfrak{C})$. Since by Theorem \ref{teo_quantum_analogue_monoidal_cat} we have $\varphi_t(\mathfrak{C})(X_1')=L_t(\mathfrak{c}_1)=L_t(Y_{i,p}),$
we can conclude that there exists a quantum cluster variable of 
$\mathcal{A}_t(\widehat{Q}^{[a,b],\mathfrak{s}})$ whose image under $\varphi_t(\mathfrak{C}_-^{[a,b]})$ is $L_t(Y_{i,p})$ (see \cite[Lem.~5.17]{HFOO_iso_quant_groth_ring_clust_alg} for a proof, in a different setting, of the same fact). Therefore, we deduce that there exists a quantum cluster variable $X$ of $\mathcal{A}_v(\widetilde{Q}^{[a,b],\mathfrak{s}})$ whose image under $\widetilde{\theta}$ is $\mathcal{L}_v(Y_{i,p})$. By definition (cf. section \ref{sec_lift}), we have $\mathcal{L}_v(Y_{i,p})=a(Y_{i,p})E_{V(i,p)}$, with a non-zero $a(Y_{i,p})\in \mathbb{Q}(v^{1/2})$. Therefore, for any $(i,p)\in J$  and $\widetilde{S}_{j,z}$ in $J_K^{fr}$, the elements $E_{V(i,p)}$ and $K_{\overline{S}_{j},z}^{\pm 1}$ belong to the image of $\widetilde{\theta}$. Since these elements generate $\mathcal{SDH}_{tw}(C^b(\mathcal{P}))^{[a,b],\mathfrak{s}}$ as a $\mathbb{Q}(v^{1/2})$-algebra, the homomorphism $\widetilde{\theta}$ is surjective.
\end{proof}

\begin{rem}
\label{rem_exchange_relation_semiderived_proof}
With the notation of the proof of Theorem \ref{teo_quantum_clust_semiderived_[a,b]}, we have that 
\begin{align*}
\mathrm{deg}\bigl(\mathcal{L}_v(m_k(e))\mathcal{L}_v(m_k(e'))\bigr) &=\mathrm{deg}(\Bigl(\overrightarrow{\prod_{k\rightarrow i'}} X_{i'}(e)\Bigr) \prod_{\widetilde{S}_{j,z}\rightarrow k}X_{\widetilde{S}_{j,z}}) \\&=\mathrm{deg}\bigr(\prod_{j\rightarrow i} \mathcal{L}_v(m_j(e))\prod_{\widetilde{S}_{j,z}\rightarrow k}K_{\overline{S}_{j,z}}\bigl).
\end{align*}
 It follows from Lemma \ref{lem_find_a_K} that the image of the exchange relation (\ref{eq_exchange_rel_X_k(e)}) via the map $\widetilde{\theta}$ coincides with (\ref{eq_semiderived_exchange_relation}).
 In particular, any exchange relation of the cluster algebra structure of $\mathcal{SDH}_{tw}(C^b(\mathcal{P}))^{[a,b],\mathfrak{s}}
 $ is of the form (\ref{eq_semiderived_exchange_relation}) or (\ref{eq_semiderived_exchange_relation2}).
\end{rem}

Let $\mathfrak{C}=(\mathfrak{c}_k)_{k\geq 1}$ be a chain of $i$-boxes of range $[-\infty,\infty]$. Consider the sequence $(\mathfrak{C}_s)_{s\geq 1}$ of chains 
of $i$-boxes defined by $\mathfrak{C}_s=(\mathfrak{c}_k)_{1\leq k\leq s}$. For any $s\geq 1$, let $[a_s,b_s]$ be the range of $\mathfrak{C}_s$. We write $J_s,J^\mu_s J^{fr}_s,J^{fr}_{K,s},\widetilde{J}_s$ for the subsets of the vertex set of $\widetilde{Q}^{[a_s,b_s],\mathfrak{s}}$ corresponding to the subsets appearing in (\ref{eq_subsets_J}).
Let $\widetilde{Q}(\mathfrak{C}_s)$ be the quiver obtained from $\widetilde{Q}(\mathfrak{C}_-^{[a_s,b_s]})\cong \widetilde{Q}^{[a_s,b_s],\mathfrak{s}}$ using the mutation sequence associated to the canonical chain transformation $\mathfrak{C}_-^{[a_s,b_s]}\mapsto \mathfrak{C}_s$.
Let $\widetilde{\Lambda}(\mathfrak{C}_s)$ be the $\widetilde{K}(\mathfrak{C}_s)\times \widetilde{K}(\mathfrak{C}_s)$-matrix obtained from $\Lambda(\mathfrak{C}_s)$ by setting each new entry to 0.
Let $\mathcal{A}_t(\widetilde{Q}(\mathfrak{C}_s))$ be the associated quantum cluster algebra localized at the frozen quantum cluster variables associated to the vertices of $J_{K,s}^{fr}$. We write $\mathcal{A}_v(\widetilde{Q}(\mathfrak{C}_s))$ for the $\mathbb{Q}(v^{1/2})$-algebra

\[
\mathcal{A}_v(\widetilde{Q}(\mathfrak{C}_s)) = \mathbb{Q}(v^{1/2})\otimes_{\mathbb{Z}[t^{\pm 1/2}]} \mathcal{A}_t(\widetilde{Q}(\mathfrak{C}_s)).\]

Remark that, if we write $\widetilde{B}(\mathfrak{C}_s)$ for the exchange matrix associated to the quiver $\widetilde{Q}(\mathfrak{C}_s)$, we have that
\begin{align*}
&\widetilde{B}(\mathfrak{C}_s)_{|J_{K,s}^{fr}\times J_{K,s}^{fr}}=0 \text{ and };\\
&\widetilde{B}(\mathfrak{C}_s)_{|K(\mathfrak{C}_s)\times K(\mathfrak{C}_s)} = B(\mathfrak{C}_s).
\end{align*} 
Moreover, for any $s'\geq s$, by Remark \ref{rem_sequence_canonical_chain_transform}, we have that

\[B(\mathfrak{C}_{s'})_{|K(\mathfrak{C}_s)\times K^{\mu}(\mathfrak{C}_s)} = B(\mathfrak{C}_s)_{|K(\mathfrak{C}_s)\times K^{\mu}(\mathfrak{C}_s)}\text{  and  } B(\mathfrak{C}_{s'})_{|(K(\mathfrak{C})_{s'}\backslash K (\mathfrak{C}_s))\times K^{\mu}(\mathfrak{C}_s)}=0.\]
Write $\widetilde{K}(\mathfrak{C}_s)$ for the vertex set $K(\mathfrak{C}_s)\sqcup J_{K,s}^{fr}$ of $\widetilde{Q}(\mathfrak{C}_-^{[a_s,b_s]})$.
The next Lemma states that also the restriction to the subsets of indices $\widetilde{K}(\mathfrak{C}_s)\times K^{\mu}(\mathfrak{C}_s)$ stabilizes.

\begin{lem}
\label{lem_B_tilde_stabilizes}
    For any $s'\geq s$, we have that
\begin{equation}
\label{eq_restriction_sequence_B(C)_SDH}
\widetilde{B}(\mathfrak{C}_{s'})_{|\widetilde{K}(\mathfrak{C}_s)\times K^{\mu}(\mathfrak{C}_s)} = \widetilde{B}(\mathfrak{C}_s)_{|\widetilde{K}(\mathfrak{C}_s)\times K^{\mu}(\mathfrak{C}_s)}\text{  and  } \widetilde{B}(\mathfrak{C}_{s'})_{|(\widetilde{K}(\mathfrak{C})_{s'}\backslash \widetilde{K} (\mathfrak{C}_s))\times K^{\mu}(\mathfrak{C}_s)}=0.
\end{equation}
\end{lem}
\begin{proof}
Notice the $i$-boxes $(\mathfrak{c}_k)_{1\leq k\leq s}$ give alternative labels for the vertices of the quiver $Q(\mathfrak{C}_s)$.
Let $\mathfrak{c}\in \mathfrak{C}_s$ be an $i$-box associated to a mutable vertex. By Theorem \ref{teo_quantum_analogue_monoidal_cat}, there are isomorphisms
\[\varphi_t(\mathfrak{C}_s): \mathcal{A}(Q(\mathfrak{C}_s),\Lambda(\mathfrak{C_s})) \xrightarrow{\sim} \mathcal{K}_t(\mathscr{C}_\mathfrak{g}^{[a_s,b_s],\mathfrak{s}}),
\]
\[\varphi_t(\mathfrak{C}_{s'}): \mathcal{A}(Q(\mathfrak{C}_{s'}),\Lambda(\mathfrak{C}_{s'})) \xrightarrow{\sim} \mathcal{K}_t(\mathscr{C}_\mathfrak{g}^{[a_{s'},b_{s'}],\mathfrak{s}}),
\]
such that, for any $\mathfrak{c'}\in \mathfrak{C}_s$ , we have $\varphi_t(\mathfrak{C}_s)(X_{\mathfrak{c}'}) = \varphi_t(\mathfrak{C}_{s'})(X_{\mathfrak{c}'})=L_t(\mathfrak{c}')$.
Therefore, it follows from (\ref{eq_restriction_sequence_B(C)_SDH}) that the images under $\varphi_t(\mathfrak{C}_{s'})$ and $\varphi_t(\mathfrak{C}_{s})$ of the exchange relation for the mutation at $\mathfrak{c}$ 
coincide and provide an equation of the form

\begin{equation}
L_t(\mathfrak{c})L_t(m) = a(t)\overrightarrow{\prod_{\mathfrak{c}'\rightarrow \mathfrak{c}}} L_t(\mathfrak{c}') +   b(t)\overrightarrow{\prod_{\mathfrak{c}''\leftarrow \mathfrak{c}}} L_t(\mathfrak{c}''), 
\end{equation} 
for a certain dominant monomial $m$.

It follows from Remark \ref{rem_exchange_relation_semiderived_proof} that also the images of the exchange relation for the mutation at $\mathfrak{c}$ under the maps 

\begin{align*}
\widetilde{\theta}_s:\mathbb{Q}(v^{1/2})\otimes_{\mathbb{Z}[t^{\pm 1/2}]} \mathcal{A}(\widetilde{Q}(\mathfrak{C}_{s}),\widetilde{\Lambda}(\mathfrak{C}_{s}))\cong \mathcal{A}_v(\widetilde{Q}^{[a_s,b_s],\mathfrak{s}}) \xrightarrow{\sim} \mathcal{SDH}_{tw}(C^b(\mathcal{P}))^{[a_s,b_s],\mathfrak{s}} \\
\widetilde{\theta}_{s'}:\mathbb{Q}(v^{1/2})\otimes_{\mathbb{Z}[t^{\pm 1/2}]}\mathcal{A}(\widetilde{Q}(\mathfrak{C}_{s'}),\widetilde{\Lambda}(\mathfrak{C}_{s'}))    \cong\mathcal{A}_v(\widetilde{Q}^{[a_{s'},b_{s'}],\mathfrak{s}}) \xrightarrow{\sim} \mathcal{SDH}_{tw}(C^b(\mathcal{P}))^{[a_{s'},b_{s'}],\mathfrak{s}} 
\end{align*}
coincide and are of the form

\begin{equation}
  \label{eq_exchange_relation_proof_stabilized_frozen}  \mathcal{L}_v(\mathfrak{c})\mathcal{L}_v(m) = a(v)\overrightarrow{\prod_{\mathfrak{c}'\rightarrow \mathfrak{c}}} \mathcal{L}_v(\mathfrak{c}') +   b(v)\widetilde{K}\overrightarrow{\prod_{\mathfrak{c}''\leftarrow \mathfrak{c}}} \mathcal{L}_v(\mathfrak{c}'') 
\end{equation} 

or 

\begin{equation}
\mathcal{L}_v(\mathfrak{c})\mathcal{L}_v(m) = a(v)\widetilde{K}\overrightarrow{\prod_{\mathfrak{c}'\rightarrow \mathfrak{c}}} \mathcal{L}_v(\mathfrak{c}')+   b(v)\overrightarrow{\prod_{\mathfrak{c}''\leftarrow \mathfrak{c}}} \mathcal{L}_v(\mathfrak{c}''),
\end{equation} 
where $\widetilde{K}$ is a monomial in certain $K_{\alpha,z}$, $\alpha\in \mathcal{K}(\mathrm{mod}(\Bbbk Q)), z\in\mathbb{Z}$ and, for any $i$-box $\widetilde{c}$, we write $\mathcal{L}_v(\widetilde{\mathfrak{c}})$ for the semi-derived $(q,t)$-character of the simple module $M(\mathfrak{c})$. Let us assume that the exchange relation is of the form (\ref{eq_exchange_relation_proof_stabilized_frozen}). Then we have

\[ \widetilde{\theta}_s(\prod X_{\widetilde{S}_{j,z}})= \widetilde{K}=\widetilde{\theta}_{s'}(\prod X_{\widetilde{S}_{j',z'}}),\]
where the product on the left hand side is taken over the set of arrows $\mathfrak{c}\rightarrow \widetilde{S}_{j,z} \text{ in the quiver } \widetilde{Q}(\mathfrak{C}_{s})$ 
and the product on the right hand side over the set of arrows $\mathfrak{c}\rightarrow\widetilde{S}_{j',z'}  \text{ in the quiver } \widetilde{Q}(\mathfrak{C}_{s'})$

It follows from Lemma \ref{lem_invertible_frozen_to_Groth_projectives} that the number of arrows  $\mathfrak{c}\rightarrow\widetilde{S}_{j,z}  \text{ in } \widetilde{Q}(\mathfrak{C}_{s'})$ equals the number of arrows $\mathfrak{c}\rightarrow\widetilde{S}_{j,z} \text{ in } \widetilde{Q}(\mathfrak{C}_{s})$, which proves the Lemma.
\end{proof}

Let $J_{K,\mathbb{Z}}^{fr}$ be the set of indices 
\[ J_{K,\mathbb{Z}}^{fr} = \{ \widetilde{S}_{j,z} \ |\ j\in I_\mathfrak{g}, z\in \mathbb{Z}\} .\]
We write $\widetilde{K}(\mathfrak{C})$ for the union $\widetilde{K}(\mathfrak{C})=K(\mathfrak{C})\sqcup J_{K,\mathbb{Z}}^{fr}.$ Let $\widetilde{B}(\mathfrak{C})$ be the exchange matrix of dimension $\widetilde{K}(\mathfrak{C})\times\widetilde{K}(\mathfrak{C})$ defined as follows:

\begin{itemize}
    \item $\widetilde{B}(\mathfrak{C})_{|K(\mathfrak{C})\times K(\mathfrak{C})} = B(\mathfrak{C})$ ; 
    \item $\widetilde{B}(\mathfrak{C})_{|J_{K,\mathbb{Z}}^{fr}\times J_{K,\mathbb{Z}}^{fr}} = 0; $
    \item $\widetilde{B}(\mathfrak{C})_{\widetilde{S}_{j,z},k} = B(\mathfrak{C}_s)_{\widetilde{S}_{j,z},k}, \text{ for any } (\widetilde{S}_{j,z},k) \in J^{fr}_{K,\mathbb{Z}} \times K(\mathfrak{C}) \text{ and for any } s\geq 1$ such that $(\widetilde{S}_{j,z},k) \in J^{fr}_{K,s} \times K^\mu(\mathfrak{C}_s)$.    
\end{itemize}
It follows from Lemma \ref{lem_B_tilde_stabilizes} that the matrix $\widetilde{B}(\mathfrak{C})$ is well-defined. Let $\widetilde{Q}(\mathfrak{C})$ be the corresponding quiver and let $\widetilde{\Lambda}(\mathfrak{C})$ be the $\widetilde{K}(\mathfrak{C})\times \widetilde{K}(\mathfrak{C})$-matrix obtained from $\Lambda(\mathfrak{C})$ by setting each new entry to 0. Then $(\widetilde{B}(\mathfrak{C}),\widetilde{\Lambda}(\mathfrak{C}))$ is a compatible pair. Let $\mathcal{A}_t(\widetilde{Q}(\mathfrak{C}))$ be the associated quantum cluster algebra localized at the frozen quantum cluster variables associated to the vertices of $J_{K,\mathbb{Z}}^{fr}$. We write $\mathcal{A}_v(\widetilde{Q}(\mathfrak{C}))$ for the $\mathbb{Q}(v^{1/2})$-algebra

\[
\mathcal{A}_v(\widetilde{Q}(\mathfrak{C})) = \mathbb{Q}(v^{1/2})\otimes_{\mathbb{Z}[t^{\pm 1/2}]} \mathcal{A}_t(\widetilde{Q}(\mathfrak{C})).\]

We assign a $\mathcal{K}(C^b(\mathcal{P}))$-degree to each initial quantum cluster variable of $\mathcal{A}_v(\widetilde{Q}(\mathfrak{C}))$ as follows:

\begin{itemize}
    \item for any vertex $i$ of $K(\mathfrak{C})$, we set
    \[ \mathrm{deg}(X_{i}) = \mathrm{deg}(\mathcal{L}_v(m)); \]
    where $m$ is the dominant monomial such that $L_t(\mathfrak{c}_i)=L_t(m)$.
    \item for any vertex of $ \widetilde{Q}(\mathfrak{C})$ of the form $\widetilde{S}_{j,z}$, we set 
    \[\mathrm{deg}(X_{\widetilde{S}_{j,z}})=\mathrm{deg}(K_{\overline{S}_j,z}).\]
\end{itemize}
Notice that, by construction of the quiver $\widetilde{Q}(\mathfrak{C})$, for any vertex $i\in \widetilde{K}(\mathfrak{C}))$ and any $s\geq 1$ such that $i$ belongs to $\widetilde{K}(\mathfrak{C}_s)$, the degree of the initial quantum cluster variable $X_i$ of $\mathcal{A}_v(\widetilde{Q}(\mathfrak{C}))$ is the same as the degree the corresponding initial quantum cluster variable of \[\mathbb{Q}(v^{1/2})\otimes_{\mathbb{Z}[t^{\pm 1/2}]} \mathcal{A}(\widetilde{Q}(\mathfrak{C}_{s}),\widetilde{\Lambda}(\mathfrak{C}_{s}))\cong \mathcal{A}_v(\widetilde{Q}^{[a_s,b_s],\mathfrak{s}}).\] In particular, for each vertex $i\in K(\mathfrak{C})$, the associated exchange relation is homogeneous. Therefore, as we have done for $\mathcal{A}_v(\widetilde{Q}^{[a,b],\mathfrak{s}})$, we can make $\mathcal{A}_v(\widetilde{Q}(\mathfrak{C}))$ into a $\mathcal{K}(C^b(\mathcal{P}))$-graded algebra.

\begin{cor}  
\label{cor_clust_semiderived}
    There is a unique isomorphism of $\mathcal{K}(C^b(\mathcal{P}))$-graded $\mathbb{Q}(v^{1/2})$-algebras 
    \[\widetilde{\theta}: \mathcal{A}_v(\widetilde{Q}(\mathfrak{C})) \xrightarrow{\sim} \mathcal{SDH}_{tw}(C^b(\mathcal{P})) \]
such that, for any $\widetilde{S}_{j,z} \in J^{fr}_{K,\mathbb{Z}}$ and $i \in K(\mathfrak{C})$
\begin{align*}
 &\widetilde{\theta}:   X_{S_{j,z}} \mapsto K_{\overline{S}_j,z} \\
 &\widetilde{\theta}:  X_{i} \mapsto \mathcal{L}_v(m),
\end{align*}
where $m$ is the dominant monomial such that $L_t(\mathfrak{c}_i)=L_t(m)$.
Moreover, the isomorphism $\widetilde{\theta}$ makes the following diagram commute:

\[
\begin{tikzcd}
    \mathcal{A}_v(\widetilde{Q}(\mathfrak{C})) \arrow[d, "\pi", twoheadrightarrow] \arrow[rr, dashed, "\Tilde{\theta}"] & &\mathcal{SDH}_{tw}(C^b(\mathcal{P}))\arrow[d, "\pi^{\mathcal{H}}", twoheadrightarrow] \\
    \mathcal{A}_v(Q(\mathfrak{C})) \arrow[r, "\sim", "\varphi_v(\mathfrak{C})"'] & \mathcal{K}_v(\mathscr{C}_\mathfrak{g}^{\mathbb{Z}}) \arrow[r, "\Phi"', "\sim"] &\mathcal{DH}_{tw}(Q).
\end{tikzcd}
\]
\end{cor}

\begin{proof}
The algebras $\mathcal{DH}_{tw}(Q)$ and $\mathcal{SDH}_{tw}(C^b(\mathcal{P}))$  are, respectively, the colimits of the diagrams

\[
\mathcal{SDH}_{tw}(C^b(\mathcal{P}))^{[a_{1},b_{1}],\mathfrak{s}}  \hookrightarrow \mathcal{SDH}_{tw}(C^b(\mathcal{P}))^{[a_{2},b_{2}],\mathfrak{s}} \hookrightarrow \mathcal{SDH}_{tw}(C^b(\mathcal{P}))^{[a_{3},b_{3}],\mathfrak{s}} \hookrightarrow \dots \]
and 
\[
\mathcal{DH}_{tw}(Q)^{[a_{1},b_{1}]} \hookrightarrow \mathcal{DH}_{tw}(Q)^{[a_{2},b_{2}]} \hookrightarrow \mathcal{DH}_{tw}(Q)^{[a_{3},b_{3}]} \hookrightarrow \dots .
\]

It follows from Lemma \ref{lem_B_tilde_stabilizes} that the quantum cluster algebra $\mathcal{A}_v(\widetilde{B}(\mathfrak{C}))$ is the colimit of the diagram 

\[
\mathcal{A}_v(\widetilde{Q}(\mathfrak{C}_1))\hookrightarrow \mathcal{A}_v(\widetilde{Q}(\mathfrak{C}_2)) \hookrightarrow \mathcal{A}_v(\widetilde{Q}(\mathfrak{C}_3)) \hookrightarrow \dots .
\]

The statement of the theorem is a consequence of the fact that these colimits are compatible with the commutative diagrams  

\[
\begin{tikzcd}
    \mathcal{A}_v(\widetilde{Q}(\mathfrak{C})_s) \arrow[d, twoheadrightarrow] \arrow[r, "\sim"]  &\mathcal{SDH}_{tw}(C^b(\mathcal{P}))^{[a_s,b_s],\mathfrak{s}}\arrow[d, twoheadrightarrow] \\
    \mathcal{A}_v(Q(\mathfrak{C}_s)) \arrow[r, "\sim"] &\mathcal{DH}_{tw}(Q)^{[a_s,b_s],\mathfrak{s}}
\end{tikzcd}
\]
provided by Theorem \ref{teo_quantum_clust_semiderived_[a,b]}. 
\end{proof}

\begin{ex}
Let $\mathfrak{g}$ be of type $A_4$. Let $Q: 1\rightarrow 2 \leftarrow 3 \rightarrow 4$ be the quiver of type $A_4$ with height function
$\varepsilon$ defined by 
\[ \varepsilon_1 = 0 ,\ \varepsilon_2=-1, \varepsilon_3=0,\ \varepsilon_4=-1.\]
Fix the $\varepsilon$-adpated reduced expression $\underline{w}_0=s_2s_4s_1s_3s_2s_4s_1s_3s_2s_4$. Let $\mathfrak{s}$ be the admissible sequence associated to the pair $(\varepsilon, \underline{w}_0)$:
\[ \mathfrak{s}: \dots,\underset{\textcolor{blue}{-4}}{(1,-4)}, \underset{\textcolor{blue}{-3}}{(4,-3)}, \underset{\textcolor{blue}{-2}}{(2,-3)}, \underset{\textcolor{blue}{-1}}{(3,-1)}, \underset{\textcolor{blue}{0}}{(1,-2)}, \underset{\textcolor{blue}{1}}{(2,-1)}, \underset{\textcolor{blue}{2}}{(4,-1)}, \underset{\textcolor{blue}{3}}{(1,0)}, \underset{\textcolor{blue}{4}}{(3,0)},
\underset{\textcolor{blue}{5}}{(2,1)}, \dots\]
Consider the chain of $i$-boxes $\mathfrak{C}=(\mathfrak{c}_k)_{k\geq 1}$ associated to the rooted sequence of expansion operators 
\[ (3;(R, L, L, L, L, R, R, R, R, L, L, L, L, R, R, R, R, \dots).\]
More explicitly
\[\mathfrak{C} = ([3]_1,[4]_3,[1]_2,[0]_4,[0,3]_1,[-1,4]_3,[1,5]_2,[2,6]_4, [0,7]_1, [-1,8]_3, \dots ).\]

Then, by applying the methods Remark \ref{rem_sequence_canonical_chain_transform}, we can obtain the quiver $Q(\mathfrak{C})$, as shown in Figure \ref{fig_quiver_derived_A4} (see also \cite[Example~B.1]{HFOO_iso_quant_groth_ring_clust_alg}). Notice that we have labelled the vertex associated to the $i$-box $\mathfrak{c}$ as $V(m)$, where $m$ is the dominant monomial such that $[L(m)]=M(\mathfrak{c})$.
By adding the vertices $\widetilde{S}_{j,z}$ with the same procedure as in Construction \ref{con_QSDH_quiver}, we obtain the quiver $\widetilde{Q}(\mathfrak{C})$, as shown in Figure \ref{fig_quiver_semiderived_A4} (notice that we have simplified the notation for the vertices).
Each frozen vertex of $\widetilde{Q}(\mathfrak{C})$ is the target and the source of an infinite number of arrows.
\end{ex}

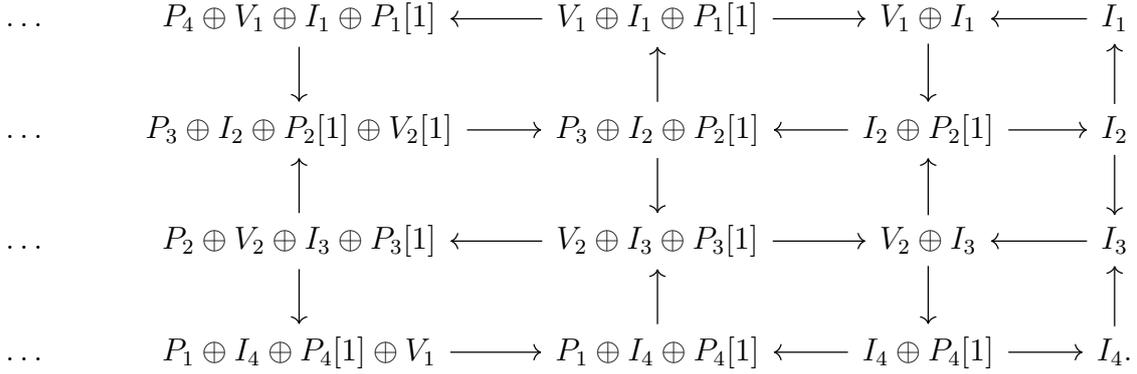
\begin{figure}
\centering
\begin{tikzcd}[ampersand replacement=\&]
\dots \& P_4\oplus V_1\oplus I_1  \oplus P_1[1]  \arrow [d] \& V_1\oplus I_1  \oplus P_1[1] \arrow[l]\arrow[r] \& V_1\oplus I_1 \arrow[d] \& I_1 \arrow[l]  \\
\dots \& P_3\oplus I_2\oplus P_2[1]\oplus V_2[1] \arrow[r]\& P_3\oplus I_2\oplus P_2[1] \arrow[u]\arrow[d] \& I_2\oplus P_2[1]\arrow[l]\arrow[r] \& I_2 \arrow[u]\arrow[d]\\
\dots \& P_2\oplus V_2\oplus I_3 \oplus P_3[1] \arrow[d]\arrow[u] \&  V_2\oplus I_3 \oplus P_3[1]\arrow[l]\arrow[r] \& V_2\oplus I_3 \arrow[u] \arrow[d] \& I_3 \arrow[l]  \\
\dots \& P_1\oplus I_4\oplus P_4[1]\oplus V_1 \arrow[r] \& P_1\oplus I_4\oplus P_4[1] \arrow[u] \& I_4\oplus P_4[1]\arrow[r]\arrow[l] \& I_4 \arrow[u] .
\end{tikzcd}
\caption{The quiver ${Q}(\mathfrak{C})$ for $\mathfrak{g}$ of type $A_4$}
  \label{fig_quiver_derived_A4}
\end{figure}

\begin{figure}
    \centering
\begin{tikzpicture}[xscale= 2.2, yscale=2.2]
    \node (N11) at (7,4) {$I_1$};
    \node (N12) at (6,4) {$V_1$};
    \node (N13) at (5,4) {$P_1[1]$};
    \node (N14) at (4,4) {$P_4$};
    \node (N15) at (3,4) {$I_4[1]$};
    \node (N16) at (2,4) {$I_4[-1]$};
    \node (N17) at (1,4) {$P_4[2]$};
    
    \node (N21) at (7,3) {$I_2$};
    \node (N22) at (6,3) {$P_2[1]$};
    \node (N23) at (5,3) {$P_3$};
    \node (N24) at (4,3) {$V_2[1]$};
    \node (N25) at (3,3) {$I_3[-1]$};
    \node (N26) at (2,3) {$I_3[1]$};
    \node (N27) at (1,3) {$V_2[-1]$};

    \node (N31) at (7,2) {$I_3$};
    \node (N32) at (6,2) {$V_2$};
    \node (N33) at (5,2) {$P_3[1]$};
    \node (N34) at (4,2) {$P_2$};
    \node (N35) at (3,2) {$I_2[1]$};
    \node (N36) at (2,2) {$I_2[-1]$};
    \node (N37) at (1,2) {$P_2[2]$};

    \node (N41) at (7,1) {$I_4$};
    \node (N42) at (6,1) {$P_4[1]$};
    \node (N43) at (5,1) {$P_1$};
    \node (N44) at (4,1) {$V_1[1]$};
    \node (N45) at (3,1) {$I_1[-1]$};
    \node (N46) at (2,1) {$I_1[1]$};
    \node (N47) at (1,1) {$V_1[-1]$};

    \node (d1) at (0.5,4) {$\dots$};
    \node (d2) at (0.5,3) {$\dots$};
    \node (d3) at (0.5,2) {$\dots$};
    \node (d4) at (0.5,1) {$\dots$};

    \node[blue, scale=0.8] (K11) at (5.5, 4.5) {$\widetilde{S}_{1,0}$};
    \node[green, scale=0.8] (K12) at (2.5, 4.5) {$\widetilde{S}_{4,-1}$};
    \node[blue, scale=0.8] (K13) at (1.5, 4.5) {$\widetilde{S}_{4,1}$};

    \node[blue, scale=0.8] (K21) at (6.5, 3.5) {$\widetilde{S}_{2,0}$};
    \node[green, scale=0.8] (K22) at (3.5, 3.5) {$\widetilde{S}_{3,-1}$};
    \node[blue, scale=0.8] (K23) at (0.5, 3.5) {$\widetilde{S}_{3,1}$};

    \node[blue, scale=0.8] (K31) at (5.5, 2.5) {$\widetilde{S}_{3,0}$};
    \node[green, scale=0.8] (K32) at (2.5, 2.5) {$\widetilde{S}_{2,-1}$};
    \node[blue, scale=0.8] (K33) at (1.5, 2.5) {$\widetilde{S}_{2,-1}$};

     \node[blue, scale=0.8] (K41) at (6.5, 1.5) {$\widetilde{S}_{4,0}$};
    \node[green, scale=0.8] (K42) at (3.5, 1.5) {$\widetilde{S}_{1,-1}$};
    \node[blue, scale=0.8] (K43) at (0.5, 1.5) {$\widetilde{S}_{1,1}$};

    \draw[->] (N11) -- (N12);
    \draw[->] (N12) -- (N22);
    \draw[->] (N13) -- (N12);
    \draw[->] (N13) -- (N14);
    \draw[->] (N14) -- (N24);
    \draw[->] (N15) -- (N14);
    \draw[->] (N15) -- (N16);
    \draw[->] (N16) -- (N26);
    \draw[->] (N17) -- (N16);
    
    \draw[->] (N21) -- (N11);
    \draw[->] (N21) -- (N31);
    \draw[->] (N22) -- (N21);
    \draw[->] (N22) -- (N23);
    \draw[->] (N23) -- (N13);
    \draw[->] (N23) -- (N33);
    \draw[->] (N24) -- (N23);
    \draw[->] (N24) -- (N25);
    \draw[->] (N25) -- (N15);
    \draw[->] (N25) -- (N35);
    \draw[->] (N26) -- (N25);
    \draw[->] (N26) -- (N27);
    \draw[->] (N27) -- (N17);
    \draw[->] (N27) -- (N37);

    \draw[->] (N31) -- (N32);
    \draw[->] (N32) -- (N22);
    \draw[->] (N32) -- (N42);
    \draw[->] (N33) -- (N32);
    \draw[->] (N33) -- (N34);
    \draw[->] (N34) -- (N24);
    \draw[->] (N34) -- (N44);
    \draw[->] (N35) -- (N34);
    \draw[->] (N35) -- (N36);
    \draw[->] (N36) -- (N26);
    \draw[->] (N36) -- (N46);
    \draw[->] (N37) -- (N36);

    \draw[->] (N41) -- (N31);
    \draw[->] (N42) -- (N41);
    \draw[->] (N42) -- (N43);
    \draw[->] (N43) -- (N33);
    \draw[->] (N44) -- (N43);
    \draw[->] (N44) -- (N45);
    \draw[->] (N45) -- (N35);
    \draw[->] (N46) -- (N45);
    \draw[->] (N46) -- (N47);
    \draw[->] (N47) -- (N37);

    \draw[->, blue] (K11) -- (N13);
    \draw[->, blue] (K11) -- (N15);
    \draw[->, blue] (K11) -- (N17);
    \draw[->, blue] (N12) -- (K11);
    \draw[->, blue] (N14) -- (K11);
    \draw[->, blue] (N16) -- (K11);

    \draw[->, green] (K12) -- (N15);
    \draw[->, green] (K12) -- (N17);;
    \draw[->, green] (N16) -- (K12);

    \draw[->, blue] (K13) -- (N17);
    \draw[->, blue] (N16) -- (K13);

    \draw[->, blue] (K21) -- (N22);
    \draw[->, blue] (K21) -- (N24);
    \draw[->, blue] (K21) -- (N26);
    \draw[->, blue] (N21) -- (K21);
    \draw[->, blue] (N23) -- (K21);
    \draw[->, blue] (N25) -- (K21);
    \draw[->, blue] (N27) -- (K21);
    
   \draw[->, green] (K22) -- (N24);
    \draw[->, green] (K22) -- (N26);
    \draw[->, green] (N25) -- (K22);
    \draw[->, green] (N27) -- (K22);

    \draw[->, blue] (N27) -- (K23);

    \draw[->, blue] (K31) -- (N33);
    \draw[->, blue] (K31) -- (N35);
    \draw[->, blue] (K31) -- (N37);
    \draw[->, blue] (N32) -- (K31);
    \draw[->, blue] (N34) -- (K31);
    \draw[->, blue] (N36) -- (K31);

    \draw[->, green] (K32) -- (N35);
    \draw[->, green] (K32) -- (N37);
    \draw[->, green] (N36) -- (K32);
    
    \draw[->, blue] (K33) -- (N37);
    \draw[->, blue] (N36) -- (K33);

    \draw[->, blue] (K41) -- (N42);
    \draw[->, blue] (K41) -- (N44);
    \draw[->, blue] (K41) -- (N46);
    \draw[->, blue] (N41) -- (K41);
    \draw[->, blue] (N43) -- (K41);
    \draw[->, blue] (N45) -- (K41);
    \draw[->, blue] (N47) -- (K41);
    
    \draw[->, green] (K42) -- (N44);
    \draw[->, green] (K42) -- (N46);
    \draw[->, green] (N45) -- (K42);
    \draw[->, blue] (N47) -- (K43);

\end{tikzpicture}
\caption{The quiver $\widetilde{Q}(\mathfrak{C})$ for $\mathfrak{g}$ of type $A_4$}
  \label{fig_quiver_semiderived_A4}
\end{figure}

\section{A braid group action on the semi-derived Hall algebra}
\label{sec_braid}
Let $A=(a_{ij})_{i,j\in I_\mathfrak{g}}$ be the Cartan matrix of $\mathfrak{g}$ and let $(-,-)$ be the symmetric bilinear such on the root lattice of $\mathfrak{g}$ such that $(\alpha_i,\alpha_j)=a_{ij}$, where the $(\alpha_i)_{i,j\in I_\mathfrak{g}}$ are the the simple roots of $\mathfrak{g}$.
Let $S_i$ be the simple module associated to the vertex $i$ of $Q$. We will use the notations
\[K_{i,m}=K_{\overline{S}_i,m},\ \ \ E_{i,m}=v^{1/2}(v-v^{-1})E_{S_i,m},\]
for any $i$ is in $I_\mathfrak{g}$ and $m$ in $\mathbb{Z}$.
\begin{rem}
The coefficient $v^{1/2}(v-v^{-1})$ is the rescaling factor used in the proof of \cite[Thm.~8.2]{HL_quantum_Groth_rings_derived_Hall}
\end{rem}

\begin{prop}\cite{HL_quantum_Groth_rings_derived_Hall}
\label{prop_presenentation_semiderived_E_i}
The twisted semi-derived Hall algebra is generated by the elements $E_{i,m}$ and $K_{i,m}$, where $i\in I_\mathfrak{g}, m\in \mathbb{Z}$, subject to the following relations: 

    \begin{align}
        \text{ the } K_{i,m} &\text{ are central };\label{sd1} \tag{SD1}\\
        E_{i,m}E_{j,m}-E_{j,m}E_{i,m} &= 0, \ \ \ \text{if }(\alpha_i,\alpha_j)=0;\label{sd2} \tag{SD2} \\
        E_{i,m}^2E_{j,m}-(v+v^{-1})E_{i,m}E_{j,m}E_{i,m}+&E_{j,m}E_{i,m}^2=0,\ \ \ \text{if }(\alpha_i,\alpha_j)=-1;\label{sd3}\tag{SD3} \\
        E_{i,m+1}E_{j,m} &=v^{(\alpha_i,\alpha_j)}E_{j,m}E_{i+1,m}+\delta_{ij}(1-v^{2})K_{i,m};\label{sd4}\tag{SD4}\\
        E_{j,r}E_{i,l}&=v^{-(-1)^{r-l}(\alpha_i,\alpha_j)}E_{i,l}E_{j,r},\ \ \ \forall r>l+1.\label{sd5} \tag{SD5}
    \end{align}
\end{prop}
\begin{proof}
We follow the arguments of the proof of \cite[Thm.~7.3]{HL_quantum_Groth_rings_derived_Hall}.
Let $\mathscr{A}$ be the $\mathbb{Q}(v^{1/2})$-algebra generated by the elements $E_{i,m}$ and $K_{i,m}$, $(i\in I_\mathfrak{g}, m\in \mathbb{Z})$, subject to the relations in the statetement of the theorem. Since the above relations are a particular case of those appearing in \ref{teo_relations_semiderived} (see \cite[Prop.~8.1]{HL_quantum_Groth_rings_derived_Hall}), we have an algebra homomorpshim $\alpha: \mathscr{A}\rightarrow \mathcal{SDH}_{tw}(C^b(\mathcal{P}))$. For $m\in \mathbb{Z}$, define $\mathcal{SDH}_{tw}(C^b(\mathcal{P}))(m)$ as the subalgebra of the twisted semi-derived Hall algebra generated by the $E_{M,m}$, for $M\in \mathrm{mod}(\Bbbk Q)$ and $m\in \mathbb{Z}$. Notice that $\mathcal{SDH}_{tw}(C^b(\mathcal{P}))(m)$ is isomorphic to  the (twisted) Hall algebra $\mathcal{H}(Q)$ of the category $\mathrm{mod}(\Bbbk Q)$. Since by Ringel's theorem (cf. \cite{Ringel_Hall_polynomials, Ringel_Hall_algebra_quant_group_1990}) the Hall algebra $\mathcal{H}(Q)$ is generated by the representatives of the simple modules, we deduce that the map $\alpha$ is surjective. In order to prove the injectivity, fix a total order on the set of pairs $(i,m), i\in I_\mathfrak{g}, m\in \mathbb{Z}$, such that 
\[ (i,m) < (i',m') \text{ if } m \geq m' \text{ or } (m=m' \text{ and } \mathrm{Ext}^1_{\Bbbk Q}(S_{i'}, S_i)=0).\] 

consider the set $\mathscr{A}_b$ of elements of $\mathscr{A}$ of the form
\[   E_{i_s,m_s}E_{i_{s-1},m_{s-1}}\dots E_{i_1,m_1}K_{i'_{s},m'_{s}} K_{i'_{s'-1},m'_{s'-1}}\dots K_{i'_{1},m'_{1}},
    \]
where $E_{i_k,m_k}$ (resp. $K_{i_k,m_k}$)  precedes  $E_{i_j,m_j}$ (resp. $K_{i_j,m_j}$) if and only if $(i_k,m_k)\leq (i_j,m_j)$. By applying relations (\ref{sd4}) and (\ref{sd5}), we can express each element of $\mathscr{A}$ as a linear combination of some elements of $\mathscr{A}_b$. Moreover, since, as discussed above in the proof, the relations in the presentation of $\mathscr{A}$ can be seen as a subset of the relations in the presentation of $\mathcal{SDH}_{tw}(C^b(\mathcal{P}))$, it follows from Proposition \ref{prop_basis_semiderived} that the elements of $\mathscr{A}_b$ are linearly independent. 
\end{proof}

\begin{teo}
\label{teo_braid}
The braid group $B_\mathfrak{g}$ of $\mathfrak{g}$ acts on the twisted semiderived Hall algebra $\mathcal{SDH}_{tw}(Q)$ by the following formulas
\[\sigma_i(E_{j,m})=\begin{cases}
E_{i,m+1}K_{i,m}^{-1}, & (\alpha_i,\alpha_j)=2,\\[0.4cm]
\frac{v^{1/2}E_{i,m}E_{j,m}-v^{-1/2}E_{j,m}E_{i,m}}{v-v^{-1}}, &(\alpha_i,\alpha_j)=-1,\\[0.4cm]
E_{j,m}, &\text{otherwise}.
\end{cases}\]

\[\sigma_i(K_{j,m})=\begin{cases}
K_{i,m}^{-1}, & (\alpha_i,\alpha_j)=2,\\
K_{i,m}K_{j,m}, &(\alpha_i,\alpha_j)=-1,\\
K_{j,m}, &\text{otherwise}.
\end{cases}\]

The inverse action is given by

\[\sigma_i^{-1}(E_{j,m})=\begin{cases}
E_{i,m-1}K_{i,m-1}^{-1}, & (\alpha_i,\alpha_j)=2,\\[0.4cm]
\frac{v^{1/2}E_{j,m}E_{i,m}-v^{-1/2}E_{i,m}E_{j,m}}{v-v^{-1}}, &(\alpha_i,\alpha_j)=-1,\\[0.4cm]
E_{j,m}, &\text{otherwise}.
\end{cases}\]

\[\sigma_i^{-1}(K_{j,m})=\begin{cases}
K_{i,m}^{-1}, & (\alpha_i,\alpha_j)=2,\\
K_{i,m}K_{j,m}, &(\alpha_i,\alpha_j)=-1,\\
K_{j,m}, &\text{otherwise}.
\end{cases}\]

\end{teo}

\begin{rem} 
Notice that, specializing the $K_{i,m}$ to 1, Proposition~\ref{prop_presenentation_semiderived_E_i} reduces to the presentation of the quantum Grothendieck ring by generators and relations given in \cite[Thm.~7.3]{HL_quantum_Groth_rings_derived_Hall} and the braid group action of Theorem \ref{teo_braid} reduced to the braid group action on the quantum Grothendieck ring announced in \cite[Thm. 2.3]{KKOP_braid_2020} and proved in \cite[Thm.~8.1]{Jang-Lee_Oh_braid_virtual_2023}. 
\end{rem}

\begin{proof}

We start by checking that, for any $k\in I$, the operator $\sigma_k$ is an algebra homomorphism, showing that it respects the relations:
\begin{itemize}
    \item  Relations (\ref{sd1}): they are respected since $\sigma_k$ preserves the subalgebra generated by the $K_{i,m}$.

    \item Relation (\ref{sd2}), with $(\alpha_i,\alpha_k)=-1$, $(\alpha_j,\alpha_k)=-1$:
    We start by computing
    \begin{align*}
        \sigma_k(E_{i,m}E_{j,m})&=\frac{1}{(v-v^{-1})^2}(v^{1/2}E_{k,m}E_{i,m}-v^{-1/2}E_{i,m}E_{k,m})(v^{1/2}E_{k,m}E_{j,m}-v^{-1/2}E_{j,m}E_{k,m})\\
        &=\frac{1}{(v-v^{-1})^2}(vE_{k,m}E_{i,m}E_{k,m}E_{j,m}-E_{k,m}E_{i,m}E_{j,m}E_{k,m}+\\
        &-E_{i,m}E_{k,m}E_{k,m}E_{j,m}+v^{-1}E_{i,m}E_{k,m}E_{j,m}E_{k,m}).
    \end{align*}
    By applying (\ref{sd3}) to the first and the fourth summand, we find that the last expression is equal to
    
    \begin{align*}
    \frac{1}{(v-v^{-1})^2}(\frac{v}{(v+v^{-1})}E_{k,m}^2E_{j,m}E_{i,m}-E_{k,m}E_{j,m}E_{i,m}E_{k,m}+\frac{v^{-1}}{(v+v^{-1})}E_{i,m}E_{j,m}E_{k,m}^2),
    \end{align*}

    which, by further application of (\ref{sd3}) to the first and the third summand, is equal to
    
\begin{align*}
     \frac{1}{(v-v^{-1})^2}&(vE_{k,m}E_{i,m}E_{k,m}E_{j,m}-E_{k,m}E_{i,m}E_{j,m}E_{k,m}\\
    &-E_{i,m}E_{k,m}E_{k,m}E_{j,m}+v^{-1}E_{i,m}E_{k,m}E_{j,m}E_{k,m})=\sigma_k(E_{j,m}E_{i,m}).
\end{align*}

\item Relation (\ref{sd2}), with $(\alpha_i,\alpha_k)=-1$, $(\alpha_j,\alpha_k)=0$:
\begin{align*}
    \sigma_k(E_{i,m}E_{j,m})=\frac{v^{1/2}E_{k,m}E_{i,m}-v^{-1/2}E_{i,m}E_{k,m}}{v-v^{-1}}E_{j,m}=E_{j,m}\frac{v^{1/2}E_{k,m}E_{i,m}-v^{-1/2}E_{i,m}E_{k,m}}{v-v^{-1}}=\sigma_k(E_{j,m}E_{i,m}).
\end{align*}

\item Relation (\ref{sd2}) with $(\alpha_i,\alpha_k)=0$, $(\alpha_j,\alpha_k)=0$: $\sigma_k$ leaves the relation unchanged.

\item Relation (\ref{sd3}), with $k=j$:
\begin{align*}
    &\sigma_j(E_{i,m}E_{j,m}E_{i,m})=\\
    &\frac{K_{j,m}^{-1}}{(v^{-1}-v)^2}(v^{1/2}E_{j,m}E_{i,m}-v^{-1/2}E_{i,m}E_{j,m})E_{j,m+1}(v^{1/2}E_{j,m}E_{i,m}-v^{-1/2}E_{i,m}E_{j,m}).      
\end{align*}

Notice that, applying relation (\ref{sd4}), we get
\begin{align*}
    &v^{-1}E_{j,m+1}(v^{1/2}E_{j,m}E_{i,m}-v^{-1/2}E_{i,m}E_{j,m})=\\
    &(1-v^{2})(v^{-1/2}-v^{-5/2})K_{j,m}E_{i,m}+v^{1/2}E_{j,m}E_{i,m}E_{j,m+1}-v^{-1/2}E_{i,m}E_{j,m}E_{j,m+1}.
\end{align*}
Therefore,
\begin{align*}
    &v^{-1}\sigma_j(E_{i,m}E_{j,m}E_{i,m})=-v^{-1/2}(v^{1/2}E_{j,m}E_{i,m}-v^{-1/2}E_{i,m}E_{j,m})E_{i,m}+\\
    &K_{j,m}^{-1}(v^{1/2}E_{j,m}E_{i,m}-v^{-1/2}E_{i,m}E_{j,m})^2E_{i,m+1}/(v^{-1}-v)^2=\\
    &-(E_{j,m}E_{i,m}-v^{-1}E_{i,m}E_{j,m})E_{i,m}+\sigma_j(E_{i,m}^2E_{j,m}).
\end{align*}

Similarly,
\begin{align*}
    &v(v^{1/2}E_{j,m}E_{i,m}-v^{-1/2}E_{i,m}E_{j,m})E_{j,m+1}=\\
    &(1-v^{2})(v^{1/2}-v^{-3/2})K_{j,m}E_{i,m}+v^{1/2}E_{j,m+1}E_{j,m}E_{i,m}-v^{-1/2}E_{j,m+1}E_{i,m}E_{j,m},
\end{align*}

and 
\begin{align*}
    &v\sigma_j(E_{i,m}E_{j,m}E_{i,m})=-E_{i,m}(vE_{j,m}E_{i,m}-E_{i,m}E_{j,m})+\sigma_j(E_{i,j}E_{i,m}^2).
\end{align*}


Therefore, cancelling the necessary terms, we end up with
\begin{align*}
    &\sigma_j(E_{i,m}^2E_{j,m}-(v+v^{-1})E_{i,m}E_{j,m}E_{i,m}+E_{j,m}E_{i,m}^2)=\\
    &v^{-1/2}(v^{1/2}E_{j,m}E_{i,m}-v^{-1/2}E_{i,m}E_{j,m})E_{i,m}+v^{1/2}E_{i,m}(v^{1/2}E_{j,m}E_{i,m}-v^{-1/2}E_{i,m}E_{j,m})=\\
    &E_{i,m}^2E_{j,m}-(v+v^{-1})E_{i,m}E_{j,m}E_{i,m}+E_{j,m}E_{i,m}^2=0.
\end{align*}

\item Relation (\ref{sd3}) with $k=i$:
we compute 
\begin{align*}
    &\sigma_i(E_{i,m}E_{j,m}E_{i,m})=\frac{K_{i,m}^{-2}}{(v-v^{-1})}(\underbrace{v^{1/2}E_{i,m+1}E_{i,m}E_{j,m}E_{i,m+1}}_a-\underbrace{v^{-1/2}E_{i,m+1}E_{j,m}E_{i,m}E_{i,m+1}}_b);\\
    &\sigma_i(E_{i,m}^2E_{j,m})=\frac{K_{i,m}^{-2}}{(v-v^{-1})}E_{i,m+1}^2(v^{1/2}E_{i,m}E_{j,m}-v^{-1/2}E_{j,m}E_{i,m}).
\end{align*}
Notice that, by using relation (\ref{sd4}),
\begin{align*}
&a=\underbrace{v^{3/2}E_{i,m}E_{j,m}E_{i,m+1}^2}_{a_1}+\underbrace{v^{1/2}(1-v^{2})K_{i,m}E_{j,m}E_{i,m+1}}_{a_2};\\
&\ =\underbrace{v^{-1/2}E_{i,m+1}^2E_{i,m}E_{j,m}}_{a_3}-\underbrace{v^{-1/2}(1-v^{2})K_{i,m}E_{i,m+1}E_{j,m}}_{a_4};\\
&b=\underbrace{v^{1/2}E_{j,m}E_{i,m}E_{i,m+1}^2}_{b_1}+\underbrace{v^{-3/2}(1-v^{2})K_{i,m}E_{j,m}E_{i,m+1}}_{b_2};\\
&\ =\underbrace{v^{-3/2}E_{i,m+1}^2E_{j,m}E_{i,m}}_{b_3}-\underbrace{v^{-5/2}(1-v^{2})K_{i,m}E_{i,m+1}E_{j,m}}_{b_4}.
\end{align*}
Then, we have that 
\begin{align*}
&\sigma_i(E_{i,m}^2E_{j,m})=v^{-1}(a_1-b_1); \ \ \ \sigma_i(E_{j,m}E_{i,m}^2)=v(a_3+b_3);\\
&v^{-1}a_2-va_4=0; \ \ \ v^{-1}b_2-vb_4=0.
\end{align*}
Therefore, all the terms cancel out and the relation is preserved.

\item relation (\ref{sd3}), with $(\alpha_i,\alpha_k)=-1$, $(\alpha_j,\alpha_k)=0$:
We start by computing
\begin{align*}
    \sigma_k(E_{i,m}E_{j,m}E_{i,m})=&\frac{1}{(v-v^{-1})^2}(\underbrace{vE_{k,m}E_{i,m}E_{j,m}E_{k,m}E_{i,m}}_a-\underbrace{E_{k,m}E_{i,m}E_{j,m}E_{i,m}E_{k,m}}_b+\\
&\underbrace{E_{i,m}E_{k,m}E_{j,m}E_{k,m}E_{i,m}}_c+\underbrace{v^{-1}E_{i,m}E_{k,m}E_{j,m}E_{i,m}E_{k,m}}_d).
\end{align*}
By repeated applications of (\ref{sd3}), we compute
\begin{align*}
    &a=\frac{v}{v+v^{-1}}(\underbrace{\frac{E_{k,m}^2E_{i,m}^2E_{j,m}}{v+v^{-1}}}_{a_1}+\underbrace{\frac{E_{j,m}E_{k,m}^2E_{i,m}^2}{v+v^{-1}}}_{a_2}+\underbrace{E_{i,m}E_{k,m}^2E_{j,m}E_{i,m}}_{a_3}) ; \\
    &b=\frac{1}{v+v^{-1}}(\underbrace{E_{k,m}E_{j,m}E_{i,m}^2E_{k,m}}_{b_1}+\underbrace{E_{k,m}E_{i,m}^2E_{j,m}E_{k,m}}_{b_2}); \\
    &d=\frac{v^{-1}}{v+v^{-1}}(\underbrace{\frac{E_{j,m}E_{i,m}^2E_{k,m}^2}{v+v^{-1}}}_{d_1}+\underbrace{\frac{E_{i,m}^2E_{k,m}^2E_{j,m}}{v+v^{-1}}}_{d_2}+\underbrace{E_{i,m}E_{k,m}^2E_{j,m}E_{i,m}}_{d_3}).
\end{align*}
Notice that 
\[\frac{va_2}{v+v^{-1}}+\frac{v^{-1}d_3}{v+v^{-1}}-c=0.\]
Then, we compute
\begin{align*}
    \sigma_k(E_{i,m}^2E_{j,m})&=(vE_{k,m}E_{i,m}E_{k,m}E_{i,m}+v^{-1}E_{i,m}E_{k,m}E_{i,m}E_{k,m}-E_{k,m}E_{i,m}^2E_{k,m}-E_{i,m}E_{k,m}^2E_{i,m})E_{j,m}\\
    &=(\frac{v}{(v+v^{-1})}(E_{k,m}^2E_{i,m}^2+E_{i,m}E_{k,m}^2E_{i,m})+\frac{v^{-1}}{(v+v^{-1}}(E_{i,m}^2E_{k,m}^2+E_{i,m}E_{k,m}^2E_{i,m})+\\
    &\ \ \ \ \ -E_{k,m}E_{i,m}^2E_{k,m}-E_{i,m}E_{k,m}^2E_{i,m})E_{j,m}\\
    &=va_1+v^{-1}d_2-b_2.
\end{align*}
Making a similar computation for $\sigma_k(E_{j,m}E_{i,m}^2)$, we see that all the terms cancel out, and so the relation is preserved.

\item Relation (\ref{sd3}) with $(\alpha_k,\alpha_i)=0$, $(\alpha_k,\alpha_j)=-1$:
\begin{align*}
    &(v+v^{-1})\sigma_k(E_{i,m}E_{j,m}E_{i,m})=\frac{(v+v^{-1})}{(v-v^{-1})}E_{i,m}(v^{1/2}E_{k,m}E_{j,m}-v^{-1/2}E_{j,m}E_{k,m})E_{i,m}\\
    &=\frac{1}{(v-v^{-1})}(v^{1/2}E_{k,m}E_{i,m}^2E_{j,m}+E_{k,m}E_{j,m}E_{i,m}^2-v^{-1/2}E_{i,m}^2E_{j,m}E_{k,m}-v^{-1/2}E_{j,m}E_{k,m}E_{i,m}^2)\\
    &=\sigma_k(E_{i,m}^2E_{j,m}+E_{j,m}E_{i,m}^2).
\end{align*}

\item Relation (\ref{sd3}) with $(\alpha_i,\alpha_k)=0$, $(\alpha_j,\alpha_k)=0$: $\sigma_k$ leaves the relation unchanged.

\item Relation (\ref{sd4}), with $j=i$ and $k=i$:
We compute
\begin{align*}
\sigma_i(E_{i,m+1}E_{i,m})&=K_{i,m+1}^{-1}K_{i,m}^{-1}E_{i,m+2}E_{i,m+1} \\
& =v^{2}K_{i,m+1}^{-1}K_{i,m}^{-1}E_{i,m+1}E_{i,m+2}+(1-v^{2})K_{i,m+1}^{-1}K_{i,m}^{-1}K_{i,m+1}\\
&=\sigma_i(v^{2}E_{i,m}E_{i,m+1}+(1-v^{2})K_{i,m})
\end{align*}

\item Relation (\ref{sd4}), with $j=i$ and $(\alpha_i,\alpha_k)=-1$:
\begin{align*}
\sigma_k(E_{i,m+1}E_{i,m})&=\frac{1}{(v-v^{-1})^2}(v^{1/2}E_{k,m+1}E_{i,m+1}-v^{-1/2}E_{i,m+1}E_{k,m+1})(v^{1/2}E_{k,m}E_{i,m}-v^{-1/2}E_{i,m}E_{k,m})\\
&=\frac{1}{(v^{1}-v^{-1})^2}(\underbrace{vE_{k,m+1}E_{i,m+1}E_{k,m}E_{i,m}}_a-\underbrace{E_{k,m+1}E_{i,m+1}E_{i,m}E_{k,m}}_b+\\
&\ \ \ -\underbrace{E_{i,m+1}E_{k,m+1}E_{k,m}E_{i,m}}_c+\underbrace{v^{-1}E_{i,m+1}E_{k,m+1}E_{i,m}E_{k,m}}_d).\\
\end{align*}
By repeated applications of (\ref{sd4}), we compute
\begin{align*}
    a&=v^{3}E_{k,m}E_{i,m}E_{k,m+1}E_{i,m+1}+v^{2}(1-v^{2})K_{k,m}E_{i,m}E_{i,m+1}\\
    &\ \ \ +v^{2}(1-v^{2})K_{i,m}E_{k,m}E_{k,m+1}+(1-v^{2})^2K_{i,m}K_{k,m}.\\
    b&=v^{2}E_{i,m}E_{k,m}E_{k,m+1}E_{i,m+1}+(1-v^{2})K_{k,m}E_{i,m}E_{i,m+1}\\
    &\ \ \ +v^{2}(1-v^{2})K_{i,m}E_{k,m}E_{k,m+1}+(1-v^{2})^2K_{i,m}K_{k,m}.\\
    c&=v^{2}E_{k,m}E_{i,m}E_{i,m+1}E_{k,m+1}+(1-v^{2})K_{i,m}E_{k,m}E_{k,m+1}\\
    &\ \ \ +v^{2}(1-v^{2})K_{k,m}E_{i,m}E_{i,m+1}+(1-v^{2})^2K_{i,m}K_{k,m}.\\
    a&=vE_{i,m}E_{k,m}E_{i,m+1}E_{k,m+1}+(1-v^{2})K_{i,m}E_{k,m}E_{k,m+1}\\
    &\ \ \ +(1-v^{2})K_{k,m}E_{i,m}E_{i,m+1}+v^{-2}(1-v^{2})^2K_{i,m}K_{k,m}.
\end{align*}

Therefore, cancelling up terms when possible, we end up with
\begin{align*}
    \sigma_k(E_{i,m+1}E_{i,m})&=v^{2}\frac{1}{(v-v^{-1})^2}(vE_{k,m}E_{i,m}E_{k,m+1}E_{i,m+1}-E_{i,m}E_{k,m}E_{k,m+1}E_{i,m+1}+\\
    &\ \ \ -E_{k,m}E_{i,m}E_{i,m+1}E_{k,m+1}+v^{-1}E_{i,m}E_{k,m}E_{i,m+1}E_{k,m+1})+(1-v^{2})K_{i,m}K_{k,m}\\
    &=\sigma_k(v^{2}E_{i,m}E_{i,m+1}))+\sigma_k(K_{i,m}).
\end{align*}

\item Relation (\ref{sd4}), with $(\alpha_i,\alpha_k)=0$ and $(\alpha_j,\alpha_k)=0$: $\sigma_k$ leaves the relation unchanged.

\item Relation (\ref{sd4}), with $(\alpha_i,\alpha_j)=-1$ and $k=i$:

\begin{align*}
    \sigma_k(E_{i,m+1}E_{j,m})&=\frac{K_{i,m+1}^{-1}E_{i,m+2}}{(v-v^{-1})}(v^{1/2}E_{i,m}E_{j,m}-v^{-1/2}E_{j,m}E_{i,m})\\
    &=\frac{v^{-(\alpha_i,\alpha_i)-(\alpha_i,\alpha_j)}}{(v-v^{-1})}(v^{1/2}E_{i,m}E_{j,m}-v^{-1/2}E_{j,m}E_{i,m})K_{i,m+1}^{-1}E_{i,m+2}=\sigma_k(v^{(\alpha_i,\alpha_j)}E_{j,m}E_{i,m+1}).
\end{align*}

Notice that we have used (\ref{sd5}).

\item Relation (\ref{sd4}), with $(\alpha_i,\alpha_j)=-1$ and $k=j$:

\begin{align*}
    \sigma_k(E_{i,m+1}E_{j,m})&=\frac{1}{(v-v^{-1})}(v^{1/2}E_{j,m+1}E_{i,m+1}-v^{-1/2}E_{i,m+1}E_{j,m+1})K_{j,m}^{-1}E_{j,m+1}\\
    &=\frac{v^{-1}K_{j,m}^{-1}E_{j,m+1}}{(v-v^{-1})}(v^{1/2}E_{j,m+1}E_{i,m+1}-v^{-1/2}E_{i,m+1}E_{j,m+1})=\sigma_k(v^{(\alpha_i,\alpha_j)}E_{j,m}E_{i,m+1}).
\end{align*}

Notice that we have used (\ref{sd3}).

\item  Relation (\ref{sd4}), with $(\alpha_i,\alpha_j)=0$ and $k=j$: it follows from the commutativity of $E_{i,m+1}$ and $E_{j,m+1}$.

\item  Relation (\ref{sd4}), with $(\alpha_i,\alpha_j)=0$ and $k=i$: it follows from the commutativity of $E_{i,m+2}$ and $E_{j,m}$.

\item Relation (\ref{sd4}), with $(\alpha_i,\alpha_j)=a_{ij}=0$ and $(\alpha_i,\alpha_k)=-1$, $(\alpha_j,\alpha_k)=-1$:

\begin{align*}
    \sigma_k(E_{i,m+1}E_{j,m})&=\frac{1}{(v-v^{-1})^2}(v^{1/2}E_{k,m+1}E_{i,m+1}-v^{-1/2}E_{i,m+1}E_{k,m+1})(v^{1/2}E_{k,m}E_{j,m}-v^{-1/2}E_{j,m}E_{k,m})\\
    &=\frac{1}{(v-v^{-1})^2}(\underbrace{vE_{k,m+1}E_{i,m+1}E_{k,m}E_{j,m}}_a-\underbrace{E_{k,m+1}E_{i,m+1}E_{j,m}E_{k,m}}_b\\
    &\ \ -\underbrace{E_{i,m+1}E_{k,m+1}E_{k,m}E_{j,m}}_c+\underbrace{v^{-1}E_{i,m+1}E_{k,m+1}E_{j,m}E_{k,m}}_d).
\end{align*}

By repeated applications of (\ref{sd4}), we have 
\begin{align*}
    &a=vv^{((\alpha_i,\alpha_k)+a_{ij}+(\alpha_k,\alpha_k)+(\alpha_i,\alpha_k))}E_{k,m}E_{j,m}E_{k,m+1}E_{i,m+1}+(1-v^{2})vv^{(\alpha_i,\alpha_k)}K_{k,m}E_{i,m+1}E_{j,m};\\
    &b=v^{((\alpha_i,\alpha_k)+a_{ij}+(\alpha_k,\alpha_k)+(\alpha_j,\alpha_k))}E_{j,m}E_{k,m}E_{k,m+1}E_{i,m+1}+(1-v^{2})v^{(\alpha_i,\alpha_k)+(\alpha_j,\alpha_k)}K_{k,m}E_{i,m+1}E_{j,m};\\
    &c=v^{((\alpha_i,\alpha_k)+a_{ij}+(\alpha_k,\alpha_k)+(\alpha_i,\alpha_k))}E_{k,m}E_{j,m}E_{i,m+1}E_{k,m+1}+(1-v^{2})K_{k,m}E_{i,m+1}E_{j,m};\\
    &d=v^{-1}v^{((\alpha_i,\alpha_k)+a_{ij}+(\alpha_k,\alpha_k)+(\alpha_i,\alpha_k))}E_{j,m}E_{k,m}E_{i,m+1}E_{k,m+1}+(1-v^{2})v^{(\alpha_i,\alpha_k)}K_{k,m}E_{i,m+1}E_{j,m};
\end{align*}
Therefore, cancelling terms when possible, 

\begin{align*}
    \sigma_k(E_{i,m+1}E_{j,m})&=\frac{v^{a_{ij}}}{(v-v^{-1})^2}(vE_{k,m}E_{j,m}E_{k,m+1}E_{i,m+1}-E_{j,m}E_{k,m}E_{k,m+1}E_{i,m+1}\\
    &\ \ -E_{k,m}E_{j,m}E_{i,m+1}E_{k,m+1}+v^{-1}E_{j,m}E_{k,m}E_{i,m+1}E_{k,m+1}=\sigma_k(E_{j,m}E_{i,m+1}).
\end{align*}

\item Relation (\ref{sd4}), with $(\alpha_i,\alpha_j)=a_{ij}$, $a_{ij}\neq 2$, and $(\alpha_i,\alpha_k)=-1$, $(\alpha_j,\alpha_k)=0$:

\begin{align*}
    \sigma_k(E_{i,m+1}E_{j,m})&=\frac{1}{(v-v^{-1})}(v^{1/2}E_{k,m+1}E_{i,m+1}-v^{-1/2}E_{i,m+1}E_{k,m+1})E_{j,m}\\
    &=\frac{v^{a_{ij}}E_{j,m}}{(v-v^{-1})}(v^{1/2}E_{k,m+1}E_{i,m+1}-v^{-1/2}E_{i,m+1}E_{k,m+1})=\sigma_k(v^{a_{ij}}E_{j,m}E_{i,m+1}).
\end{align*}

\item Relation (\ref{sd4}), with $(\alpha_i,\alpha_j)=a_{ij}$, $a_{ij}\neq 2$, and $(\alpha_i,\alpha_k)=0$, $(\alpha_j,\alpha_k)=-1$:

\begin{align*}
    \sigma_k(E_{i,m+1}E_{j,m})&=\frac{E_{i,m+1}}{(v-v^{-1})}(v^{1/2}E_{k,m}E_{j,m}-v^{-1/2}E_{j,m}E_{k,m})\\
    &=\frac{v^{a_{ij}}}{(v-v^{-1})}(v^{1/2}E_{k,m}E_{j,m}-v^{-1/2}E_{j,m}E_{k,m})E_{i,m+1}=\sigma_k(v^{a_{ij}}E_{j,m}E_{i,m+1}).
\end{align*}

\item Relation (\ref{sd5}), with $j=i$ and $k=i$:
We compute
\begin{align*}
\sigma_i(E_{i,r}E_{i,l})&=K_{i,r}^{-1}K_{i,l}^{-1}E_{i,r+1}E_{i,l+1} \\
& =v^{-(-1)^{(r-l)}(\alpha_i,\alpha_i)}K_{i,l}^{-1}K_{i,r}^{-1}E_{i,l+1}E_{i,r+1}\\
&=\sigma_i(v^{-(-1)^{(r-l)}(\alpha_i,\alpha_i)}E_{i,m}E_{i,m+1})
\end{align*}

\item Relation (\ref{sd5}), with $(\alpha_i,\alpha_k)=0$ and $(\alpha_j,\alpha_k)=0$: $\sigma_k$ leaves the relation unchanged.

\item Relation (\ref{sd5}), with $(\alpha_i,\alpha_j)=a_{ij}$ and $(\alpha_i,\alpha_k)=-1$, $(\alpha_j,\alpha_k)=-1$:

\begin{align*}
    \sigma_k(E_{i,r}E_{j,l})&=\frac{1}{(v-v^{-1})^2}(v^{1/2}E_{k,r}E_{i,r}-v^{-1/2}E_{i,r}E_{k,r})(v^{1/2}E_{k,l}E_{j,l}-v^{-1/2}E_{j,l}E_{k,l})\\
    &=\frac{1}{(v-v^{-1})^2}(\underbrace{vE_{k,r}E_{i,r}E_{k,l}E_{j,l}}_a-\underbrace{E_{k,r}E_{i,r}E_{j,l}E_{k,l}}_b\\
    &\ \ -\underbrace{E_{i,r}E_{k,r}E_{k,l}E_{j,l}}_c+\underbrace{v^{-1}E_{i,r}E_{k,r}E_{j,l}E_{k,l}}_d).
\end{align*}

By repeated applications of (\ref{sd5}), we have 
\begin{align*}
    a&=vv^{-((-1)^{(r-l)}(\alpha_i,\alpha_k)+(-1)^{(r-l)}(\alpha_i,\alpha_j)+(-1)^{(r-l)}(\alpha_k,\alpha_k)+(-1)^{(r-l)}(\alpha_k,\alpha_j))}E_{k,l}E_{i,l}E_{k,r}E_{i,r}.\\
    b&=v^{-((-1)^{(r-l)}(\alpha_i,\alpha_k)+(-1)^{(r-l)}(\alpha_i,\alpha_i)+(-1)^{(r-l)}(\alpha_k,\alpha_k)+(-1)^{(r-l)}(\alpha_k,\alpha_j))}E_{i,l}E_{k,l}E_{k,r}E_{i,r}.\\
    c&=v^{-((-1)^{(r-l)}(\alpha_i,\alpha_k)+(-1)^{(r-l)}(\alpha_i,\alpha_i)+(-1)^{(r-l)}(\alpha_k,\alpha_k)+(-1)^{(r-l)}(\alpha_k,\alpha_j))}E_{k,l}E_{i,l}E_{i,r}E_{k,r}.\\
    a&=v^{-1}v^{-((-1)^{(r-l)}(\alpha_i,\alpha_k)+(-1)^{(r-l)}(\alpha_i,\alpha_i)+(-1)^{(r-l)}(\alpha_k,\alpha_k)+(-1)^{(r-l)}(\alpha_k,\alpha_j))}E_{i,l}E_{k,l}E_{i,r}E_{k,r}.
\end{align*}

Notice that $(-1)^{(r-l)}(\alpha_i,\alpha_k)+(-1)^{(r-l)}(\alpha_j,\alpha_k)+(-1)^{(r-l)}(\alpha_k,\alpha_k)=0$. Therefore, cancelling terms when possible, 

\begin{align*}
    \sigma_k(E_{i,r}E_{j,l})&=\frac{v^{-(-1)^{(r-l)}(\alpha_i,\alpha_j)}}{(v-v^{-1})^2}(vE_{k,l}E_{j,l}E_{k,r}E_{i,r}-E_{j,l}E_{k,l}E_{k,r}E_{i,r}+\\
    &\ \ \ -E_{k,l}E_{j,l}E_{i,r}E_{k,r}+v^{-1}E_{j,l}E_{k,l}E_{i,r}E_{k,r})\\
    &=\sigma_k(v^{-(-1)^{(r-l)}(\alpha_i,\alpha_j)}E_{j,l}E_{i,r})).
\end{align*}

\item Relation (\ref{sd5}), with $(\alpha_i,\alpha_j)=-1$ and $k=i$:

By using (\ref{sd5}), we compute 
\begin{align*}
    \sigma_k(E_{i,r}E_{j,l})&=\frac{K_{i,r}^{-1}E_{i,r+1}}{(v-v^{-1})}(v^{1/2}E_{i,l}E_{j,l}-v^{-1/2}E_{j,l}E_{i,l})\\
    &=\frac{v^{-(-1)^{(r+1-l)}(\alpha_i,\alpha_i)-(-1)^{(r+1-l)}(\alpha_i,\alpha_j)}}{(v-v^{-1})}(v^{1/2}E_{i,l}E_{j,l}-v^{-1/2}E_{j,l}E_{i,l})K_{i,r}^{-1}E_{i,r}\\
    &=\sigma_k(v^{-(-1)^{(r-l)}(\alpha_i,\alpha_j)}E_{j,l}E_{i,r}).
\end{align*}

\item Relation (\ref{sd5}), with $(\alpha_i,\alpha_j)=-1$, $k=j$ and $r>l+2$:

By using (\ref{sd5}), we compute 

\begin{align*}
    \sigma_k(E_{i,r}E_{j,l})&=\frac{1}{(v-v^{-1})}(v^{1/2}E_{j,r}E_{i,r}-v^{-1/2}E_{i,r}E_{j,r})K_{j,l}^{-1}E_{j,l+1}\\
    &=\frac{v^{-(-1)^{(r-l-1)}(\alpha_i,\alpha_i)-(-1)^{(r-l-1)}(\alpha_i,\alpha_j)}K_{j,l}^{-1}E_{j,l+1}}{(v-v^{-1})}(v^{1/2}E_{j,r}E_{i,r}-v^{-1/2}E_{i,r}E_{j,r})\\
    &=\sigma_k(v^{-(-1)^{(r-l)}(\alpha_i,\alpha_j)}E_{j,r}E_{i,l}).
\end{align*}

\item Relation (\ref{sd5}), with $(\alpha_i,\alpha_j)=-1$, $k=j$ and $r=l+2$:

By using (\ref{sd4}), we compute 

\begin{align*}
    \sigma_k(E_{i,l+2}E_{j,l})&=\frac{1}{(v-v^{-1})}(v^{1/2}E_{j,l+2}E_{i,l+2}-v^{-1/2}E_{i,l+2}E_{j,l+2})K_{j,l}^{-1}E_{j,l+1}\\
    &=\frac{v^{(\alpha_i,\alpha_i)+(\alpha_i,\alpha_j)}K_{j,l}^{-1}E_{j,l+1}}{(v-v^{-1})}(v^{1/2}E_{j,r}E_{i,r}-v^{-1/2}E_{i,r}E_{j,r})+\\
    &\ \ \ +\frac{(1-v^{2})K_{j,l}^{-1}K_{j,l+1}^{-1}}{(v-v^{-1})}(v^{-1/2-(\alpha_i,\alpha_j)}v^{1/2})\\
    &=\sigma_k(v^{-(-1)^{(r-l)}(\alpha_i,\alpha_j)}E_{j,r}E_{i,l}).
\end{align*}

\item  Relation (\ref{sd5}), with $(\alpha_i,\alpha_j)=0$ and $k=j$: it follows from the commutativity of $E_{i,r}$ and $E_{j,l+1}$.

\item  Relation (\ref{sd5}), with $(\alpha_i,\alpha_j)=0$ and $k=i$: it follows from the commutativity of $E_{i,r+1}$ and $E_{j,l}$.

\item Relation (\ref{sd5}), with $(\alpha_i,\alpha_j)=a_{ij}$, $a_{ij}\neq 2$, and $(\alpha_i,\alpha_k)=-1$, $(\alpha_j,\alpha_k)=0$:

\begin{align*}
    \sigma_k(E_{i,r}E_{j,l})&=\frac{1}{(-v^{-1})}(v^{1/2}E_{k,r}E_{i,r}-v^{-1/2}E_{i,r}E_{k,r})E_{j,l}\\
    &=\frac{v^{-(-1)^{(r-l)}(\alpha_i,\alpha_j)}E_{j,l}}{(v-v^{-1})}(v^{1/2}E_{k,r}E_{i,r}-v^{-1/2}E_{i,r}E_{k,r})=\sigma_k(v^{-(-1)^{(r-l)}(\alpha_i,\alpha_j)}E_{j,l}E_{i,r}).
\end{align*}

\item Relation (\ref{sd5}), with $(\alpha_i,\alpha_j)=a_{ij}$, $a_{ij}\neq 2$, and $(\alpha_i,\alpha_k)=0$, $(\alpha_j,\alpha_k)=-1$:

\begin{align*}
    \sigma_k(E_{i,r}E_{j,l})&=\frac{E_{i,r}}{(v-v^{-1})}(v^{1/2}E_{k,l}E_{j,l}-v^{-1/2}E_{j,l}E_{k,l})\\
    &=\frac{v^{-(-1)^{(r-l)}(\alpha_i,\alpha_j)}}{(v-v^{-1})}(v^{1/2}E_{k,l}E_{j,l}-v^{-1/2}E_{j,l}E_{k,l})E_{i,r}=\sigma_k(v^{-(-1)^{(r-l)}(\alpha_i,\alpha_j)}E_{j,l}E_{i,r}).
\end{align*}

\end{itemize}

Therefore, each $\sigma_i$ is an algebra homomorphism. 
We check now that the $\sigma_i$ satisfy the braid group relations: 

\begin{itemize}

\item[$\circ$] Case $(\alpha_i,\alpha_j)=-1$, $(\alpha_i,\alpha_k)=-1$ $(\alpha_j,\alpha_k)=0$, $\sigma_i\sigma_j\sigma_i(E_{k,m})=\sigma_j\sigma_i\sigma_j(E_{k,m})$:
\begin{align*}
\sigma_i\sigma_j\sigma_i(E_{k,m})=\frac{\sigma_i\sigma_j(v^{1/2}E_{i,m}E_{k,m}-v^{-1/2}E_{k,m}E_{i,m})}{v-v^{-1}} = \\
\frac{\sigma_i((vE_{j,m}E_{i,m}-E_{i,m}E_{j,m})E_{k,m}-E_{k,m}(E_{j,m}E_{i,m}-v^{-1}E_{k,m}E_{i,m}E_{j,m}))}{(v-v^{-1})^2}=\\
    \frac{K_{i,m}^{-1}}{(v-v^{-1})^4}\bigl( 
    \underbrace{v^{2}E_{i,m}E_{j,m}E_{i,m+1}E_{i,m}E_{k,m}}_{a_1}-\underbrace{vE_{i,m}E_{j,m}E_{i,m+1}E_{k,m}E_{i,m}}_{a_2}\\
    -\underbrace{vE_{j,m}E_{i,m}E_{i,m+1}E_{i,m}E_{k,m}}_{a_3}+\underbrace{E_{j,m}E_{i,m}E_{i,m+1}E_{k,m}E_{i,m}}_{a_4}\\
    -\underbrace{vE_{i,m+1}E_{i,m}E_{j,m}E_{i,m}E_{k,m}}_{a_5}+\underbrace{E_{i,m+1}E_{i,m}E_{j,m}E_{k,m}E_{i,m}}_{a_6}\\
    +\underbrace{E_{i,m+1}E_{j,m}E_{i,m}E_{i,m}E_{k,m}}_{a_7}-\underbrace{v^{-1}E_{i,m+1}E_{j,m}E_{i,m}E_{k,m}E_{i,m}}_{a_8}\\
    -\underbrace{vE_{i,m}E_{k,m}E_{i,m}E_{j,m}E_{i,m+1}}_{a_9}+\underbrace{E_{i,m}E_{k,m}E_{j,m}E_{i,m}E_{i,m+1}}_{a_{10}}\\
    +\underbrace{E_{k,m}E_{i,m}E_{i,m}E_{j,m}E_{i,m+1}}_{a_{11}}-\underbrace{v^{-1}E_{k,m}E_{i,m}E_{j,m}E_{i,m}E_{i,m+1}}_{a_{12}}\\
    +\underbrace{E_{i,m}E_{k,m}E_{i,m+1}E_{i,m}E_{j,m}}_{a_{13}}-\underbrace{v^{-1}E_{i,m}E_{k,m}E_{i,m+1}E_{j,m}E_{i,m}}_{a_{14}}\\
    -\underbrace{v^{-1}E_{k,m}E_{i,m}E_{i,m+1}E_{i,m}E_{j,m}}_{a_{15}}+\underbrace{v^{-2}E_{k,m}E_{i,m}E_{i,m+1}E_{j,m}E_{i,m}}_{a_{16}}.
\end{align*}
Using the relations, we obtain
\begin{align*}
    a_5=a_1+v(1-v^{2})K_{i,m}E_{j,m}E_{i,m}E_{k,m},\\
    a_6=a_2+(1-v^{2})K_{i,m}E_{j,m}E_{k,m}E_{i,m},\\
    a_7=a_3+v^{-1}(1-v^{2})K_{i,m}E_{j,m}E_{i,m}E_{k,m},\\
    a_8=a_4+v^{-2}(1-v^{2})K_{i,m}E_{j,m}E_{k,m}E_{i,m},\\
    a_{13}=a_9+(1-v^{2})E_{i,m}E_{k,m}E_{j,m},\\
    a_{14}=a_{10}+v^{-2}(1-v^{2})E_{i,m}E_{k,m}E_{j,m},\\
    a_{15}=a_{11}+v^{-1}(1-v^{2})E_{k,m}E_{i,m}E_{j,m},\\
    a_{16}=a_{12}+v^{-3}(1-v^{2})E_{k,m}E_{i,m}E_{j,m}.
\end{align*}
Therefore, we remain with
\begin{align*}
    \sigma_i\sigma_j\sigma_i(E_{k,m})=\frac{-1}{(v-v^{-1})^{-3}}(-v^{2}E_{j,m}E_{i,m}E_{k,m}+vE_{j,m}E_{k,m}E_{i,m}+E_{j,m}E_{i,m}E_{k,m}+\\-v^{-1}E_{j,m}E_{k,m}E_{i,m}
    +vE_{i,m}E_{k,m}E_{j,m}-v^{-1}E_{i,m}E_{k,m}E_{j,m}-E_{k,m}E_{i,m}E_{j,m}+v^{-2}E_{k,m}E_{i,m}E_{j,m})=\\
    \frac{1}{(v-v^{-1})^{-2}}(vE_{j,m}E_{i,m}E_{k,m}-E_{j,m}E_{k,m}E_{i,m}-E_{i,m}E_{k,m}E_{j,m}+v^{-1}E_{k,m}E_{i,m}E_{j,m}).
\end{align*}
On the other hand

\begin{align*}
\sigma_j\sigma_i\sigma_j(E_{k,m})=\sigma_j\sigma_i(E_{k,m})=\sigma_j(v^{1/2}E_{i,m}E_{k,m}-v^{-1/2}E_{i,m}E_{k,m})/(v-v^{-1})=\\
\frac{1}{(v-v^{-1})^2}(v^{1/2}(v^{1/2}E_{j,m}E_{i,m}-v^{-1/2}E_{i,m}E_{j,m})E_{k,m}-v^{-1/2}E_{k,m}(v^{-1/2}E_{j,m}E_{i,m}-v^{-1/2}E_{i,m}E_{j,m}))=\\
\frac{1}{(v-v^{-1})^2}(vE_{j,m}E_{i,m}E_{k,m}-E_{j,m}E_{k,m}E_{i,m}-E_{i,m}E_{j,m}E_{k,m}+v^{-1}E_{k,m}E_{i,m}E_{j,m}).
\end{align*}
Using that $E_{j,m}$ and $E_{k,m}$ commute, we conclude that $\sigma_j\sigma_i\sigma_j(E_{k,m})=\sigma_i\sigma_j\sigma_i(E_{k,m})$

\item[$\circ$] Case $(\alpha_i,\alpha_j)=-1$, $(\alpha_i,\alpha_k)=0$ $(\alpha_j,\alpha_k)=0$, $\sigma_i\sigma_j\sigma_i(E_{k,m})=\sigma_j\sigma_i\sigma_j(E_{k,m})$: There is nothing to check, since $\sigma_i$ and $\sigma_j$ stabilize $E_{k,m}$. 

\item[$\circ$] \ Case $(\alpha_i,\alpha_j)=-1$, $\sigma_i\sigma_j\sigma_i(E_{i,m})=\sigma_j\sigma_i\sigma_j(E_{i,m})$ :  

\begin{align*}
\sigma_i\sigma_j\sigma_i(E_{i,m})=\sigma_i\sigma_j(E_{i,m+1}K_{i,m}^{-1})=\sigma_i(\frac{K_{i,m}^{-1}K_{j,m}^{-1}}{v-v^{-1}}(v^{1/2}E_{j,m+1}E_{i,m+1}-v^{-1/2}E_{i,m+1}E_{j,m+1}))=\\
\frac{K_{j,m}^{-1}K_{i,m+1}^{-1}}{(v-v^{-1})^2}((vE_{i,m+1}E_{j,m+1}-E_{j,m+1}E_{i,m+1})E_{i,m+2}-E_{i,m+2}(E_{i,m+1}E_{j,m+1}-v^{-1}E_{j,m+1}E_{i,m+1}))=\\
\frac{K_{j,m}^{-1}K_{i,m+1}^{-1}}{(v-v^{-1})^2}(-(1-v)K_{i,m+1}E_{j,m+1})+v^{-2}(1-v^{2})K_{i,m+1}E_{j,m+1})=E_{j,m+1}K_{j,m}^{-1}.
\end{align*}

\begin{align*}
\sigma_j\sigma_i\sigma_j(E_{i,m})=\sigma_j\sigma_i((v^{1/2}E_{j,m}E_{i,m}-v^{-1/2}E_{i,m}E_{j,m}))/(v-v^{-1})=\\
\frac{1}{(v-v^{-1})^2}\sigma_j((vE_{i,m}E_{j,m}-E_{j,m}E_{i,m})E_{i,m+1}K_{i,m}^{-1}-E_{i,m+1}K_{i,m}^{-1}((E_{i,m}E_{j,m}-v^{-1}E_{j,m}E_{i,m}))=\\
\frac{1}{(v-v^{-1})^2}\sigma_j(-(1-v)E_{j,m}+v^{-2}(1-v)E_{j,m})=\sigma_j(E_{j,m})=E_{j,m+1}K_{j,m}^{-1}
\end{align*}

\item[$\circ$] \ Case $(\alpha_i,\alpha_j)=-1$, $\sigma_i\sigma_j\sigma_i(K_{i,m})=\sigma_j\sigma_i\sigma_j(K_{i,m})$: 
\begin{align*}   \sigma_i\sigma_j\sigma_i(K_{i,m})=\sigma_i\sigma_j(K^{-1}_{i,m})=\sigma_i(K_{i,m}^{-1}K_{j,m}^{-1})=K_{j,m}^{-1}=\\
    \sigma_j(K_{j,m})=\sigma_j\sigma_i(K_{i,m}K_{j,m})=\sigma_j\sigma_i\sigma_j(K_{i,m}).
\end{align*}

\item[$\circ$] Relation $\sigma_j\sigma_i(E_{i,m})=\sigma_i\sigma_j(E_{i,m})$, with $(\alpha_i,\alpha_j)=0$ : 

\begin{align*}
\sigma_j\sigma_i(E_{i,m})&=\sigma_j(E_{i,m+1}K_{i,m}^{-1})=E_{i,m+1}K_{i,m}^{-1}=\sigma_i(E_{i,m})=\sigma_i\sigma_j(E_{i,m}).
\end{align*}

$\sigma_j\sigma_i(E_{i,m})=\sigma_i\sigma_j(E_{i,m})$, with $(\alpha_i,\alpha_j)=0$, $(\alpha_i,\alpha_k)=-1$, $(\alpha_j,\alpha_k)=-1$ : 

\end{itemize}

\end{proof}

Finally, we show that the inverse to $\sigma_i$ is given by the formula in the statement. 
\begin{align*}
    &\sigma_i^{-1}\sigma_i(E_{i,m})=\sigma_i^{-1}(E_{i,m+1}K_{i,m}^{-1})=E_{i,m}K_{i,m}^{-1}K_{i,m}=E_{i,m}; \\
    &\sigma_i\sigma_i^{-1}(E_{i,m})=\sigma_i(E_{i,m-1}K_{i,m-1}^{-1})=E_{i,m}K_{i,m-1}^{-1}K_{i,m-1}=E_{i,m}.
\end{align*}

For $(\alpha_i,\alpha_j)=-1$
\begin{align*}
    &\sigma_i^{-1}\sigma_i(E_{j,m})=\sigma_i^{-1}(\frac{v^{1/2}E_{i,m}E_{j,m}-v^{-1/2}E_{j,m}E_{i,m}}{v-v^{-1}})=\\
    &\ \ \ \frac{E_{i,m-1}K_{i,m-1}^{-1}(vE_{j,m}E_{i,m}-E_{i,m}E_{j,m})-(E_{j,m}E_{i,m}-v^{-1}E_{i,m}E_{j,m})E_{i,m-1}K_{i,m-1}^{-1}}{(v-v^{-1})^2}\\
    &=\frac{K_{i,m-1}^{-1}(vE_{i,m-1}E_{j,m}E_{i,m}-E_{i,m-1}E_{i,m}E_{j,m}-E_{j,m}E_{i,m}E_{i,m-1}+v^{-1}E_{i,m}E_{j,m}E_{i,m+1})}{(v-v^{-1})^2}
\end{align*}
By applying the relation (\ref{sd4}) to the last two terms and cancelling summands when possible, we find that the last expression is equal to
\begin{align*}
    &\frac{K_{i,m-1}^{-1}(-(1-v^{2})K_{i,m-1}E_{j,m}+(1-v^{2})v^{-2}K_{i,m-1}E_{j,m}}{(v-v^{-1})^2}=E_{j,m}.
\end{align*}

\appendix
\section{A quick reminder on (quantum) cluster algebras}
\label{appendix_clust_alg}

\subsection{The definition of cluster algebra}
Let $J$ be a (possibly infinite) countable set and let $J^{\text{fr}}$ be a subset of $J$. 
Let $Q$ be a quiver with vertex set $J$. The vertices in $J^{fr}$ are called \emph{frozen}. A vertex which is not frozen is said to be \emph{mutable}. We suppose that
\begin{itemize}
    \item between two vertices, there are finitely many arrows;
    \item $Q$ does not have loops nor $2$-cycles;
    \item each non-frozen vertex is incident with at most a finite number of arrows.
\end{itemize} 

For any vertex $k\in J\backslash J^{\text{fr}}$, the \emph{mutation} of the quiver $Q$ at $k$ is the quiver $\mu_k(Q)$ obtained from $Q$ performing, in order, the following operations: 

\begin{enumerate}
    \item[(i)]  for any path $i\rightarrow k\rightarrow j$ in $Q$, add an arrow $i\rightarrow j$;
    \item[(ii)] reverse all the arrows incident to $k$; 
    \item[(iii)] remove the arrows in a maximal subset of disjoint $2$-cycles.
\end{enumerate}
Let $B=(b_{ij})_{i,j\in J}$ be the \emph{exchange matrix} of $Q$, that is, the skew-symmetric $J\times J$-matrix whose entry $b_{ij}$ is the difference between the numbers of arrows from $i$ to $j$ and the numbers of arrows from $j$ to $i$ in $Q$. For any mutable vertex $k$, the exchange matrix of $\mu_k(Q)$, denoted by $\mu_k(B)=(b_{ij}')_{i,j\in J}$, is given by

\[
b_{ij}' = 
\begin{cases}
    -b_{ij}, & \text{if } i = k \text{ or } j = k, \\
    \\
    \displaystyle b_{ij} + \frac{|b_{ik}| b_{kj} + b_{ik} |b_{kj}|}{2}, & \text{otherwise}.
\end{cases}
\]
We call $\mu_k(B)$ the \emph{matrix mutation} of $B$ at $k$.
Equivalently, if we define the matrices $E_k=(e^k_{ij})_{i,j\in J}$ and $F_k=(f^k_{ij})_{i,j\in J}$ by

\[
e_{ij}^k=\begin{cases}
    \delta_{ij} & \text{if } j\neq k, \\
    -1 & \text{if } i=j=k,\\
    \mathrm{max}(0,-b_{ik}) & \text{if } i\neq j=k,
\end{cases}\ \text{ and }\ f_{ij}^k=\begin{cases}
    \delta_{ij} & \text{if } j\neq k, \\
    -1 & \text{if } i=j=k,\\
    \mathrm{max}(0,-b_{ik}) & \text{if } i=j\neq k,
\end{cases}\
\]
we can express the mutation of the matrix $B$ at $k$ as 

\[\mu_k(B)=E_kBF_k.\]

Let $\mathbb{F}$ be the field of rational functions with rational coefficients in the indeterminates $x_i, (i\in J)$. A \emph{seed} in $\mathbb{F}$ is a pair $(u,R)$ where $R$ is a quiver as above and $u=(u_i)_{i\in J}$ is a set of algebraically independent elements of $\mathbb{F}$ which generates $\mathbb{F}$ as a field. We call $((x_i)_{i\in J},Q))$ the \emph{initial seed}.
For any seed $(u,R)$ and any mutable vertex $k$ of $R$, the \emph{seed mutation} of $(u,R)$ at $k$ is the seed 

\[ (u', R'), \]
where $R'=\mu_k(R)$ and $u'$ is defined by

\[
u'_i = 
\begin{cases}
    u_i, & \text{if } i \neq k, \\
    \\
    \displaystyle \frac{
    \left( \prod\limits_{b_{ij} > 0} u_j^{b_{ij}} \right) + \left( \prod\limits_{b_{ij} < 0} u_j^{-b_{ij}} \right)
    }{u_k}, & \text{if } i = k.
\end{cases}
\]

We introduce the following terminology:

\begin{itemize}
    \item a \emph{cluster} is a set of elements $u=(u_j)_{j\in J}$ appearing in a seed $(u,R)$ obtained from the initial seed by iterated mutations. 
    \item the elements $u_j$ of a a cluster $u$ are called  \emph{cluster variables}.
    \item a \emph{cluster monomial} is a product of non-negative powers of cluster variables belonging to the same cluster.
\end{itemize}

\begin{de}
The \emph{cluster algebra} $\mathcal{A}(Q)$ associated with $Q$ is the $\mathbb{Z}$-subalgebra of $\mathbb{F}$ generated by the cluster variables.
\end{de}

The next result, known as the \emph{Laurent phenomenon}, is central in the theory of cluster algebras.

\begin{teo}[{\cite[Thm.~3.1]{Fomin_Zelevinsky_clustalg_I}}]
Let $u=(u_i)_{i\in J}$ be a cluster of a cluster algebra $\mathcal{A}$. Each cluster variable of $\mathcal{A}$ can be expressed as a Laurent polynomial in the variables $(u_i)_{i\in J\backslash J_f}$ with coefficients in $\mathbb{Z}[(u_i)_{i\in J_{f}}]$.
\end{teo}

\subsection{Green and red vertices} 
\label{subsect_green_red_vertices}
We now introduce a convenient way to parametrize the seeds and the clusters variables of a cluster algebra. Let $\mathbb{T}$\label{symb_T_regular_tree} be the regular tree whose edges incident to each vertex are parametrized by the elements of $J\backslash J^{fr}$. We call it the \emph{regular $(J\backslash J^{fr})$-tree}. Pick a vertex $e_0$ of $\mathbb{T}$ and associate to it the initial seed of $\mathcal{A}(Q)$. We associate to each other vertex $e$ of $\mathbb{T}$ a seed $(x(e),Q(e))$ of $\mathcal{A}(Q)$ in such a way that, if $e$ is connected to an other vertex $e'$ by an edge labeled $k$, then the associated seeds $(x(e),Q(e))$  and $(x(e'),Q(e'))$ are related by the mutation at $k$. We denote by $(x_i(e))_{i\in J}$ the cluster variables of the cluster $x(e)$.

The \emph{framed quiver} of $Q$ is the quiver $\overline{Q}$ obtained by adding, for any mutable vertex $i$ of $Q$, a frozen vertex $i'$ and an arrow $i'\rightarrow i$. For any vertex $e$ of $\mathbb{T}$ and any $i\in J\backslash J^{fr}$, the vertex $i$ of $Q(e)$ is \emph{green} (resp. \emph{red}) with respect to $e_0$ if in $\overline{Q}(e)$ there are no arrows $i\rightarrow j'$ (resp. $j' \rightarrow i$), where $j$ is in $J\backslash J^{fr}$.
Notice that every mutable vertex of $Q$ is green. Remarkably, for any vertex $e$ of $\mathbb{T}$, each mutable vertex of $Q(e)$ is either green or red \cite{Derksen_Weyman_Zelevinsky_quivpot2}.

The framed quiver $\overline{Q}$ was used in \cite{Fomin_Zelevinsky_clustalg_IV} to give the definition of $g$-vector of a cluster variable. Here, we provide an equivalent characterization, conjectured in \cite[Conj.~7.12]{Fomin_Zelevinsky_clustalg_IV} and proved in \cite[Thm.~1.7]{Derksen_Weyman_Zelevinsky_quivpot2}. Let $(z_i)_{i\in J}$ be the standard basis of the lattice $\mathbb{Z}^J$. For any vertex $e$ of $\mathbb{T}$ and any $i\in J$, the \emph{$g$-vector} $g_i(e)$ of the cluster variable $x_i(e)$ is recursively defined as follows:
\begin{itemize}
    \item if $e=e_0$, then $g_i(e_0)=z_i$;
    \item if $e$ is connected to an other vertex $e'$ by an edge labeled $k$, then 
\begin{equation}
\label{eq_g_vector}
g_i(e')=\begin{cases}
    g_i(e) & \text{if $i\neq k$},\\
    -g_k(e)+\sum_{k\rightarrow k'} g_{k'}(e) & \text{if $i=k$ and $k$ is green in $Q(e)$ with respect to $e_0$},\\
    -g_k(e)+\sum_{k'\rightarrow k} g_{k'}(e) & \text{if $i=k$ and $k$ is red in $Q(e)$ with respect to $e_0$}.
\end{cases}
\end{equation}
\end{itemize}

\subsection{Quantum cluster algebras}
Let $t$ be an indeterminate with a formal square root $t^{1/2}$ and let $\Lambda=(\lambda_{ij})_{i,j\in J}$ be a $\mathbb{Z}$-valued skew symmetric $J\times J$-matrix.

\begin{de}
    The \emph{quantum torus} $\mathcal{T}(\Lambda)$ associated to the matrix $\Lambda$ is the $\mathbb{Z}[t^{\pm 1/2}]$-algebra generated by the indeterminates $X_{i}^{\pm 1}$ $(i\in J)$, subject to the following relations
\end{de}

\begin{enumerate}
    \item[(i)]
     $X_{i}X_{i}^{-1} = 1 = X_{i}^{-1}X_{i},\ \ i\in J$;
    \item[(ii)]
    $X_{i}X_{j} = t^{\Lambda_{ij}} X_{j}X_{i},\ \ i,j\in J$.
\end{enumerate}

For any $J$-tuple of integers $\alpha=(\alpha_i)_{i\in J}$ and an arbitrary total ordering of the set $J$, we define the \emph{commutative monomial}
\[ X^{\alpha}= t^{-\frac{1}{2}\sum_{i<j}\alpha_i\alpha_j\Lambda_{ij}}\overrightarrow{\prod_{j\in J}}X_j^{\alpha_j}.\]
This definition ensures that the resulting element does not depend on the choice of the total ordering.

Let $B=(b_{ij})_{i,j\in J}$ be an exchange matrix. The pair $(B,\Lambda)$ is \emph{compatible} if there exist positive integers $(d_j)_{j\in J\backslash J^{fr}}$ such that

\[ \sum_{k\in J} b_{kj}\lambda_{ki} = d_j\delta_{ij} \ \ \ j\in J\backslash J^{fr}, i\in J.\]
For any $k\in J\backslash J^{fr}$, the \emph{mutation} of $\Lambda$ at $k$, denoted $\mu_k(\Lambda)$, is defined to be
\[ \mu_k(\Lambda) = E_k^T \Lambda E_k.\]
The pair $(\mu_k(B),\mu_k(\Lambda))$, denoted $\mu_k(B,\Lambda)$, is still a compatible pair. 

The \emph{initial quantum seed} associated to the compatible pair $(B,\Lambda)$ is the triple
\[ ((X_i)_{i\in J},B,\Lambda),\]
where the $X_i$ are the generators of the quantum torus $\mathcal{T}(\Lambda)$.

Simililary to the construction of cluster algebras, starting from the inital quantum seed we construct recursively a set of triples $( X(e),B(e), \Lambda(e))$ labelled by the vertices of the regular $(J\backslash J^{fr})$-tree, where $(B(e),\Lambda(e))$ is a compatible pair and $X(t)=(X_i(e))_{i\in J}$ is a tuple of algebraically independent elements of $\mathcal{T}(\Lambda)$.
We proceed in the following way:

\begin{enumerate}
    \item the triple $(X(e_0),B(e_0), \Lambda(e_0))$ is the initial quantum seed;
    \item when the vertices $e$ and $e'$ are linked by an edge labelled by $k$, we have 
    \[(X(e'),B(e'), \Lambda(e'))=(\mu_k(X(e)),\mu_k(B(e)), \mu_k(\Lambda(e))),\]
    with $\mu_k(X(e)))=(X_i(e'))_{i\in J}$ defined by
    \[ X_i(e') = \begin{cases}
        t^{-\frac{1}{2}\sum_{i<j,i}\alpha_i\alpha_j\Lambda_{ij}(t)}\overrightarrow{\prod_{j\in J}}X_j^{\alpha_j}(e)+t^{-\frac{1}{2}\sum_{i<j,i}\beta_i\beta_j\Lambda_{ij}(e)}\overrightarrow{\prod_{j\in J}}X_j^{\beta_j}(t)&\\
        X_i(e) & \text{if } i\neq k,
    \end{cases}\]
    where, for any $i\in J$,
    \[ \alpha_i=\begin{cases}
        -1 & \text{if } i=k,\\
        \mathrm{max}(0,b_{ik}(e)) \text{if } i\neq k,
    \end{cases} \text{ and } \beta_i=\begin{cases}
        -1 & \text{if } i=k,\\
        \mathrm{max}(0,-b_{ik}(e)) \text{if } i\neq k.
    \end{cases}\]
\end{enumerate}

For any $i\in J$ and vertex $e$ of the vertex regular $(J\backslash J^{fr})$-tree, we call $X_i(e)$ a \emph{quantum cluster variables}. For any $i\in J^{fr}$, we say that the quantum cluster variable $X_i$ \emph{frozen}. By The quantum analogue of the Laurent phenomenon \cite{Berenstein_Zelevinsky_quant_clust_alg}, we have that each quantum cluster variable is indeed an element of the quantum torus $\mathcal{T}(\Lambda)$. 

\begin{de}
The \emph{quantum cluster algebra} $\mathcal{A}(Q,\Lambda)$ associated with $Q$ is the $\mathbb{Z}[t^{\pm 1/2}]$-subalgebra of the skew field $\mathcal{F}(\mathcal{T}(\Lambda))$ generated by the quantum cluster variables.
\end{de}

In the following, when dealing with quantum cluster algebras, we will still make use of the notation and terminology of subsection \ref{subsect_green_red_vertices}. 

The quantum cluster algebra $\mathcal{A}(Q,\Lambda)$ can be thought of as a non-commutative deformation of the cluster algebra $\mathcal{A}(Q)$. Consider the evaluation morphism
\[\mathrm{ev}_{t=1}: \mathcal{T}(\Lambda) \rightarrow \mathbb{Z}[(x^{\pm 1}_i)_{i\in J}], \ X_i \mapsto x_i,\ t^{1/2}\mapsto 1.\]
As shown in \cite[Lemma~3.3]{GLS_quant_clust_alg_spec}, it restricts to an epimorphism of $\mathbb{Z}$-algebras 
\[ \mathrm{ev}_{t=1}: \mathcal{A}(Q,\Lambda) \rightarrow \mathcal{A}(Q).\]
Moreover, under this surjection, the set of quantum cluster variables of $\mathcal{A}(Q,\Lambda)$ and the set of cluster variables of $\mathcal{A}(Q)$ are in bijection (see \cite[Lemma~A.4]{HFOO_iso_quant_groth_ring_clust_alg} and \cite[Thm.~6.1]{Berenstein_Zelevinsky_quant_clust_alg}).

\section*{Acknowledgement}
This article is part of the author's PhD thesis. The author is grateful to his supervisor Bernhard Keller for his guidance. Moreover, he is indebted to Ryo Fujita, Mikhail Gorsky, David Hernandez, Geoffrey Janssens and Bernand Leclerc for discussions and useful comments.


\providecommand{\bysame}{\leavevmode\hbox to3em{\hrulefill}\thinspace}
\providecommand{\MR}{\relax\ifhmode\unskip\space\fi MR }
\providecommand{\MRhref}[2]{%
  \href{http://www.ams.org/mathscinet-getitem?mr=#1}{#2}
}
\providecommand{\href}[2]{#2}

\end{document}